\documentclass[11pt]{article}

\usepackage[a4paper,left=20mm,top=20mm,right=20mm,bottom=25mm]{geometry}
\usepackage{enumitem}

\usepackage{amsmath, amsfonts, amssymb, amsthm}
\usepackage{algorithm}
\usepackage[noend]{algorithmic}
\usepackage{algorithm}
\usepackage{authblk}

\usepackage{graphicx}  
\usepackage{mathtools}
\usepackage{subfig}
\usepackage[mathscr]{euscript}
\newcommand{\email}{\texttt}

\usepackage{color}
\definecolor{gray}{RGB}{180,180,180}

\usepackage[font=footnotesize,width=.85\textwidth,labelfont=bf]{caption}

\newcommand{\NEW}[1]{#1}

\newcommand{\citet}{\cite}
\newcommand{\child}{\mathsf{child}}
\newcommand{\parent}{\mathsf{par}}
\newcommand{\bmr}{\mathrel{\rightarrow}}
\newcommand{\rthin}{\mathrel{\mathrel{\ooalign{\hss\raisebox{-0.17ex}{$\sim$}\hss\cr\hss\raisebox{0.720ex}{\scalebox{0.75}{$\bullet$}}\hss}}}}

\newcommand{\AX}[1]{\textnormal{#1}}

\DeclareMathOperator{\lca}{lca}
\DeclareMathOperator{\Aho}{Aho}

\providecommand{\keywords}[1]{\textbf{\textit{Keywords: }} #1}

\newtheorem{theorem}{Theorem}
\newtheorem{lemma}{Lemma}
\newtheorem{proposition}{Proposition}
\newtheorem{corollary}{Corollary}

\newtheorem{definition}{Definition}
\newtheorem{fact}[theorem]{Observation}

\begin{document}
\setlength{\marginparsep}{-0.6cm}
\setlength{\marginparwidth}{2.3cm}

\title{Best Match Graphs}

\author[1]{Manuela Gei{\ss}}
\author[8]{Edgar Ch{\'a}vez}
\author[8]{Marcos Gonz{\'a}lez}
\author[8]{Alitzel L{\'o}pez }
\author[9]{Dulce Valdivia }
\author[8]{Maribel Hern{\'a}ndez Rosales}
\author[5]{B{\"a}rbel M.\ R.\ Stadler}
\author[2,3]{Marc Hellmuth}
\author[1,5,6,7]{Peter F.\ Stadler}

\affil[1]{\footnotesize Bioinformatics Group, Department of Computer Science; and
		    Interdisciplinary Center of Bioinformatics, University of Leipzig, 
			 H{\"a}rtelstra{\ss}e 16-18, D-04107 Leipzig}
\affil[2]{Dpt.\ of Mathematics and Computer Science, University of Greifswald, Walther-
  Rathenau-Strasse 47, D-17487 Greifswald, Germany 	}
\affil[3]{Saarland University, Center for Bioinformatics, Building E 2.1, P.O.\ Box 151150, D-66041 Saarbr{\"u}cken, Germany }
\affil[4]{Department of Mathematics and Computer Science,
		    University of Southern Denmark, Denmark }
\affil[5]{Max-Planck-Institute for Mathematics in the Sciences, 
  Inselstra{\ss}e 22, D-04103 Leipzig}
\affil[6]{Inst.\ f.\ Theoretical Chemistry, University of Vienna, 
  W{\"a}hringerstra{\ss}e 17, A-1090 Wien, Austria}
\affil[7]{Santa Fe Institute, 1399 Hyde Park Rd., Santa Fe, USA} 
\affil[8]{        CONACYT-Instituto de Matem{\'a}ticas, UNAM Juriquilla,
        Blvd. Juriquilla 3001,
        76230 Juriquilla, Quer{\'e}taro, QRO, M{\'e}xico}
\affil[9]{Universidad Aut{\'o}noma de Aguascalientes, Centro de Ciencias B{\'a}sicas,
        Av. Universidad 940,
        20131 Aguascalientes, AGS, M{\'e}xico;
        Instituto de Matem{\'a}ticas, UNAM Juriquilla,
        Blvd. Juriquilla 3001, 
        76230 Juriquilla, Quer{\'e}taro, QRO, M{\'e}xico.}
\date{}
\normalsize

\maketitle

\marginpar{\color{blue}\scriptsize
  A Corrigendum is appended to this preprint, see margin notes.} 
\abstract{   
  Best match graphs arise naturally as the first processing intermediate in
  algorithms for orthology detection.  Let $T$ be a phylogenetic (gene)
  tree $T$ and $\sigma$ an assignment of leaves of $T$ to species.  The
  best match graph $(G,\sigma)$ is a digraph that contains an arc from $x$
  to $y$ if the genes $x$ and $y$ reside in different species and $y$ is
  one of possibly many (evolutionary) closest relatives of $x$ compared to
  all other genes contained in the species $\sigma(y)$.  Here, we
  characterize best match graphs and show that it can be decided in cubic
  time and quadratic space whether $(G,\sigma)$ derived from a tree in this
  manner. If the answer is affirmative, there is a unique least resolved
  tree that explains $(G,\sigma)$, which can also be constructed in cubic
  time.
}

\bigskip
\noindent

\sloppy

\section{Introduction}

Symmetric best matches \cite{Tatusov:97}, also known as bidirectional best
hits (BBH) \cite{Overbeek:99}, reciprocal best hits (RBH) \cite{Bork:98},
or reciprocal smallest distance (RSD) \cite{Wall:03} are the most commonly
employed method for inferring orthologs
\cite{Altenhoff:09,Altenhoff:16}. Practical applications typically produce,
for each gene from species $A$, a list of genes found in species $B$,
ranked in the order of decreasing \NEW{sequence similarity}.  From these
lists, reciprocal best \NEW{hits} are readily obtained. Some software
tools, such as \texttt{ProteinOrtho} \cite{Lechner:11a,Lechner:14a},
explicitly construct a digraph whose arcs are the (approximately)
co-optimal best matches. \NEW{Empirically, the pairs of genes that are
  identified as reciprocal best hits depend on the details of the
  computational method for quantifying sequence similarity. Most commonly,
  \texttt{blast} or \texttt{blat} scores are used. Sometimes exact pairwise
  alignment algorithms are used to obtain a more accurate estimate of the
  evolutionary distance, see \citet{MorenoHagelsieb:08} for a detailed
  investigation. Independent of the computational details, however,
  reciprocal best match are of interest because they approximate the
  concept of pairs of \emph{reciprocal evolutionarily most closely related}
  genes. It is this notion that links best matches directly to orthology:
  Given a gene $x$ in species $a$ (and disregarding horizontal gene
  transfer), all its co-orthologous genes $y$ in species $b$ are by
  definition closest relatives of $x$.}

Evolutionary relatedness is a phylogenetic property \NEW{and thus is
  defined relative to the phylogenetic tree $T$ of the genes under
  consideration. More precisely, we consider a set of genes $L$ (the leaves
  of the phylogenetic tree $T$), a set of species $S$, and a map $\sigma$
  assigning to each gene $x\in L$ the species $\sigma(x)\in S$ within which
  it resides. A gene $x$ is more closely related to gene $y$ than to gene
  $z$ if $\lca(x,y)\prec\lca(x,z)$. As usual, $\lca$ denotes the last
  common ancestor, and $p\prec q$ denotes the fact that $q$ is located
  above $p$ along the path connecting $p$ with the root of $T$. The partial
  order $\preceq$ (which also allows equality) is called the ancestor order
  on $T$. We can now make the notion of a \emph{best match} precise:}
\begin{definition}
  \label{def:genclose}
  Consider a tree $T$ with leaf set $L$ and a surjective map
  $\sigma:L\to S$.  Then $y\in L$ is a \emph{best match} of $x\in L$, in
  symbols $x\bmr y$, if and only if $\lca(x,y)\preceq \lca(x,y')$ holds for
  all leaves $y'$ from species $\sigma(y')=\sigma(y)$.
\end{definition}
  
\begin{figure}[tb]
  \begin{center} 
    \includegraphics[width=\textwidth]{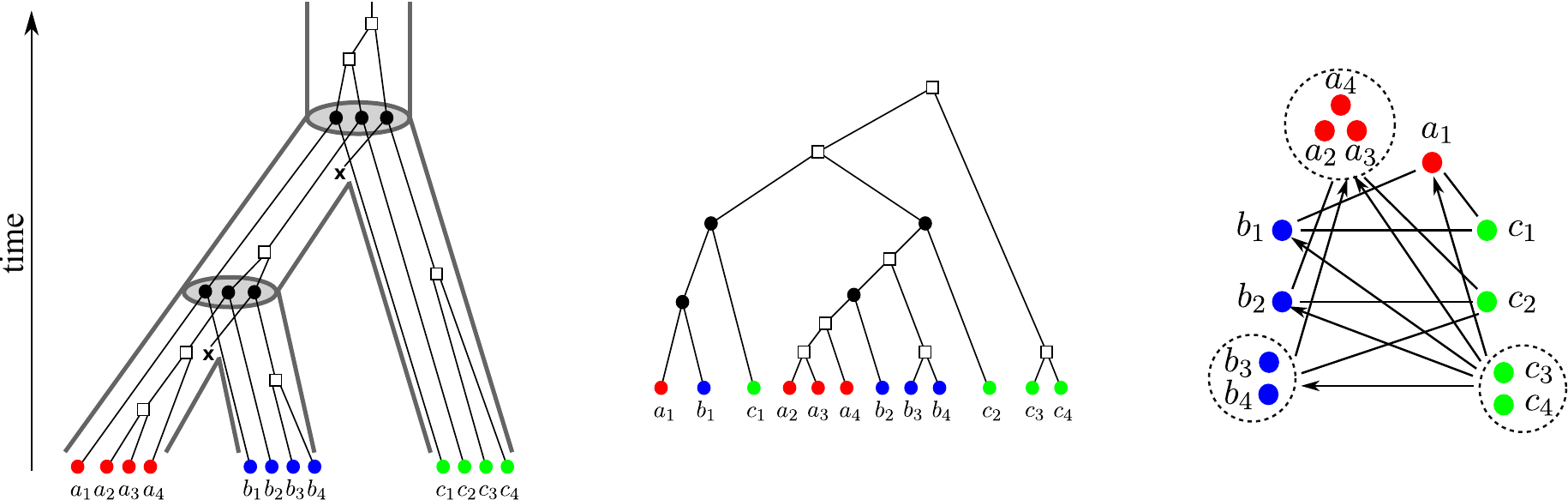}
  \end{center}
  \caption[]{\NEW{An evolutionary scenario (left) consists of a gene tree
      whose inner vertices are marked by the event type ($\bullet$ for
      speciations, $\square$ for gene duplications, and $\times$ for gene
      loss) together with its embedding into a species tree (drawn as
      tube-like outline). All events are placed on a time axis. The middle
      panel shows the observable part of the gene tree $(T,\sigma)$; it is
      obtained from the gene tree in the full evolutionary scenario by
      removing all leaves marked as loss events and suppression of all
      resulting degree two vertices
      \cite{HernandezRosales:12a,Hellmuth2017}. The r.h.s.\ panel shows the
      colored best match graph $(G,\sigma)$ that is explained by
      $(T,\sigma)$. Directed arcs indicate the best match relation
      $\bmr$. Bi-directional best matches ($x\bmr y$ and $y\bmr x$) are
      drawn as solid lines without arrow heads instead of pairs of arrows.
      Dotted circles collect sets of leaves that have the same in- and
      out-neighborhood. The corresponding arcs are shown only once.}  }
  \label{fig:exmpl-distance}
\end{figure}

\NEW{In order to understand how best matches (in the sense of Def.\
  \ref{def:genclose}) are approximated by best hits computed by mean
  sequence similarity we first observe that best matches can be expressed
  in terms of the evolutionary time. Denote by $t(x,y)$ the temporal
  distance along the evolutionary tree, as in
  Fig.~\ref{fig:exmpl-distance}. By definition $t(x,y)$ is twice the time
  elapsed between $\lca(x,y)$ and $x$ (or $y$), assuming that all leaves of
  $T$ live in the present.  Instead of Def.\ \ref{def:genclose} we can then
  use ``$x\bmr y$ holds if and only if $t(x,y)\le t(x,y')$ for all $y'$ with
  $\sigma(y')=\sigma(y)\ne\sigma(x)$.''
  Mathematically, this is equivalent to Def.\ \ref{def:genclose} whenever $t$
  is an ultrametric distance on $T$. For the temporal distance $t$ this is
  the case.} Best match heuristics therefore assume (often tacitly) that
the \emph{molecular clock hypothesis} \cite{Zuckerkandl:62,Kumar:05} is at
least a reasonable approximation.

While this strong condition is violated more often than not, best match
heuristics still perform surprisingly well on real-life data, in particular
in the context of orthology prediction \cite{Wolf:12}. \NEW{Despite
  practical problems, in particular in applications to Eukaryotic genes
  \cite{Dalquen:13}, reciprocal best heuristics perform at least as good
  for this task as methods that first estimate the gene phylogeny
  \cite{Altenhoff:16,Setubal:18a}. One reason for their resilience is that
  the identification of best matches only requires inequalities between
  sequence similarities. In particular, therefore they are invariant under
  monotonic transformations and, in contrast e.g.\ to distance based
  phylogenetic methods, does not require additivity. Even more generally, it
  suffices that the evolutionary rates of the different members of a gene
  family are roughly the same within each lineage.}

\NEW{Best match methods are far from perfect, however. Large differences in
  evolutionary rates between paralogs, as predicted by the DDC model
  \cite{Force:99}, for example, may lead to false negatives among
  co-orthologs and false positive best matches between members of slower
  subfamilies. Recent orthology detection methods recognize the sources of
  error and complement sequence similarity by additional sources of
  information. Most notably, synteny is often used to support or reject
  reciprocal best matches \cite{Lechner:14a,Jahangiri:17}. Another class of
  approaches combine the information of small sets of pairwise matches to
  improve orthology prediction \cite{Yu:11,Train:17}. In the Concluding
  Remarks we briefly sketch a simple quartet-based approach to identify
  incorrect best match assignments.

  Extending the information used for the correction of initial reciprocal
  best hits to a global scale, it is possible to improve orthology
  prediction by enforcing the global cograph of the orthology relation
  \cite{Hellmuth:15a,Lafond:16}. This work originated from an analogous
  question: Can empirical reciprocal best match data be improved just by
  using the fact that ideally a best match relation should derive from a
  tree $T$ according to Def.\ \ref{def:genclose}? To answer this question we
  need to understand the structure of best match relations.}

The best match relation is conveniently represented as a colored digraph.
\begin{definition} 
  \label{def:cBMG}
  Given a tree $T$ and a map $\sigma:L\to S$, the \emph{colored best match
    graph} (cBMG) $G(T,\sigma)$ has vertex set $L$ and arcs $xy\in E(G)$ if
  $x\ne y$ and $x\bmr y$. Each vertex $x\in L$ obtains the color
  $\sigma(x)$.  \newline The rooted tree $T$ \emph{explains} the
  vertex-colored graph $(G,\sigma)$ if $(G,\sigma)$ is isomorphic to the
  cBMG $G(T,\sigma)$.  
\end{definition}
To emphasize the number of colors used in $G(T,\sigma)$,
that is, the number of species in $S$, we will write $|S|$-cBMG.

\begin{figure}
  \begin{center} 
        \includegraphics[width=0.2\textwidth]{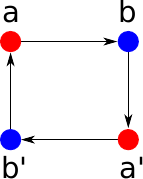}
        \caption{\NEW{Not every graph with non-empty out-neighborhoods is
            is a colored best match graph. The 4-vertex graph $(G,\sigma)$
            shown here is the smallest connected counterexample: there is
            no leaf-colored tree $(T,\sigma)$ that explains $(G,\sigma)$.}}
        \label{fig:counter} 
  \end{center}
\end{figure} 

\NEW{The purpose of this contribution is to establish a characterization of
  cBMGs as an indispensable prerequisite for any method that attempts to
  directly correct empirical best match data.}  After settling the notation
we establish a few simple properties of cBMGs and show that key problems
can be broken down to the connected components of 2-colored BMGs. These are
considered in detail in section~\ref{sect:2cBMG}. \NEW{The characterization
  of 2-BMGs is not a trivial task. Although the existence of at least one
  out-neighbor for each vertex is an obvious necessary condition, the
  example in Fig.\ \ref{fig:counter} shows that it is not sufficient.  In
  Section~\ref{sect:2cBMG} we} prove our main results on 2-cBMGs: the
existence of a unique least resolved tree that explains any given 2-cBMG
(Thm.~\ref{thm:lr-2}), a characterization in terms of informative triples
that can be extracted directly from the input graph (\NEW{Thm.}\
\ref{thm:2cbmg-triples}), and a characterization in terms of three simple
conditions on the out-neighborhoods (Thm.~\ref{thm:char2}).  In
section~\ref{sect:ncBMG} we provide a complete characterization of a
general cBMG: It is necessary and sufficient that the subgraph induced by
each pair of colors is a 2-cBMG and that the union of the triple sets of
their least resolved tree representations is consistent. After a brief
discussion of algorithmic considerations we close with a brief introduction
into questions for future research.

\section{Preliminaries} 

\subsection{Notation} 

Given a rooted tree $T=(V,E)$ with root $\rho$, we say that a vertex
$v\in V$ is an \emph{ancestor} of $u\in V$, in symbols $u\preceq v$, $v$
lies one the path from $\rho$ to $u$.  For an edge $e=uv$ in the rooted
tree $T$ we assume that $u$ is closer to the root of $T$ than $v$. In this
case, we call $v$ a child of $u$, and $u$ the parent of $v$ and denote with
$\child(u)$ the set of children of $u$.  Moreover, $e=uv$ is an \emph{outer
  edge} if $v\in L(T)$ and an \emph{inner edge} otherwise.  We write $T(v)$
for the subtree of $T$ rooted at $v$, $L(T')$ for the leaf set of some
subtree $T'$ and $\sigma(L')=\{\sigma(x)\mid x\in L'\}$.  To avoid dealing
with trivial cases we will assume that $\sigma(L)=S$ contains at least two
distinct colors. Furthermore, for $|S|=1$, the edge-less graphs are
explained by any tree. Hence, we will assume $|S|\ge2$ in the
following. Without loosing generality we may assume throughout this
contribution that all trees are phylogenetic, i.e., all inner vertices of
$T$ (except possibly the root) have at least two children. A tree is binary
if each inner vertex has exactly two children.

We follow the notation used e.g.\ in \cite{Semple:03} and say that $T'$ is
\emph{displayed} by $T$, in symbols $T'\le T$, if the tree $T'$ can be
obtained from a subtree of $T$ by contraction of edges.  In addition, we
will consider trees $T$ with a coloring map $\sigma: L(T)\to S$ of its
leaves, in short $(T,\sigma)$.  We say that $(T,\sigma)$ \emph{displays} or
\emph{is a refinement of} $(T',\sigma')$, whenever $T'\le T$ and
$\sigma(v)=\sigma'(v)$ for all $v\in L(T')$.

We write $T_{L'}$ for the \emph{restriction} of $T$ to a subset
$L'\subseteq L$. We denote by $\lca(A)$ the last common ancestor of all
elements of any set $A$ of vertices in $T$. For later reference we note
that $\lca(A\cup B)=\lca(\lca(A),\lca(B))$.  We sometimes write $\lca_T$
instead of $\lca$ to avoid ambiguities.  We will often write $A \preceq x$,
in case that $\lca(A)\preceq x$ and therefore, that $x$ is an ancestor of
all $a\in A$.

A binary tree on three leaves is called a \emph{triple}. In particular, we
write $xy|z$ for the triple on the leaves $x,y$ and $z$ if the path from
$x$ to $y$ does not intersect the path from $z$ to the root. We write
$r(T)$ for the set of all triples that are displayed by the tree $T$. In
particular, we call a triple set $R$ \emph{consistent} if there exists a
tree $T$ that displays $R$, i.e., $R\subseteq r(T)$. A rooted triple
$xy|z\in r(T)$ \emph{distinguishes} an edge $(u,v)$ in $T$ if and only if
$x$, $y$ and $z$ are descendants of $u$, $v$ is an ancestor of $x$ and $y$
but not of $z$, and there is no descendant $v'$ of $v$ for which $x$ and
$y$ are both descendants. In other words, the edge $(u,v)$ is distinguished
by $xy|z\in r(T)$ if $\lca(x,y)=v$ and $\lca(x,y,z)=u$.

By a slight abuse of notation we will retain the symbol $\sigma$ also for
the restriction of $\sigma$ to a subset $L'\subseteq L$. We write
$L[s]=\{x\in L \mid \sigma(x)=s\}$ for the color classes on the leaves of
$(T,\sigma)$ and denote by $\overline{\sigma(x)}=S\setminus \{\sigma(x)\}$
the set of colors different from the color of the leaf $x$.

All (di-)graphs considered here do not contain loops, i.e., there are no
arcs of the form $xx$.  For a given (di-)graph $G=(V,E)$ and a subset
$W\subseteq V$, we write $G[W]$ for the \emph{induced subgraph} of $G$ that
has vertex set $W$ and contains all edges $xy$ of $G$ for which $x,y\in W$.
A digraph $G=(V,E)$ is \emph{connected} if for any pairs of vertices
$x,y\in V$ there is a path $x=v_1-v_2-\dots-v_k=y$ such that (i)
$v_iv_{i+1} \in E$ or (ii) $v_{i+1}v_i \in E$, $1\leq i<k$.  The graph
$G(V,E)$ is strongly connected if for all $x,y\in V$ there is a sequence
$P_{xy}$ that always satisfies Condition (i). For a vertex $x$ in a digraph
$G$ we write $N(x)=\{z \mid xz \in E(G)\}$ and $N^-(x)=\{z \mid zx \in
E(G)\}$ for the out- and in-neighborhoods of $x$, respectively.
\NEW{For any set of vertices $A\subseteq L$ we write 
  $N(A):=\bigcup_{x\in A} N(x)$  and $N^-(A):=\bigcup_{x\in A} N^-(x)$.}

\subsection{Basic Properties of Best Match Relations}

The best match relation $\bmr$ is reflexive because
$\lca(x,x)=x\prec\lca(x,y)$ for all genes $y$ with $\sigma(x)=\sigma(y)$.
For any pair of distinct genes $x$ and $y$ with $\sigma(x)=\sigma(y)$ we
have $\lca(x,y)\notin\{x,y\}$, hence the relation $\bmr$ has off-diagonal
pairs only between genes from different species. There is still a 1-1
correspondence between cBMGs (Def.\ \ref{def:cBMG}) and best match
relations (Def.\ \ref{def:genclose}): In the cBMG the reflexive loops are
omitted, in the relation $\bmr$ they are added.

The tree $(G,\sigma)$ and the corresponding cBGM $G(T,\sigma)$ employ the
same coloring map $\sigma: L\to S$, i.e., our notion of isomorphy requires
the preservation of colors. The usual definition of isomorphisms of colored
graphs also allows an arbitrary bijection between the color sets. This is
not relevant for our discussion: if $(G',\sigma')$ and $G(T,\sigma)$ are
isomorphic in the usual sense then there is -- by definition -- a bijective
relabeling of the colors in $(G',\sigma')$ that makes them coincide with
the vertex coloring of $G(T,\sigma)$.  In other words, if $\varphi$ is an
isomorphism from $(G',\sigma')$ to $G(T,\sigma)$ we assume w.l.o.g.\ that
$\sigma'(x) = \sigma(\varphi(x))$, i.e., each vertex $x\in V(G')$ has the
same color as the vertex $\varphi(x)\in V(G)$.

\subsection{Thinness} 

In undirected graphs, equivalence classes of vertices that share the same
neighborhood are considered in the context of thinness of the graph
\cite{McKenzie:71,Sumner:73,Bull:89}. The concept naturally extends to
digraphs \cite{Hellmuth:15}. For our purposes the following variation on
the theme is most useful:
\begin{definition}\label{def:rthin}
  Two vertices $x,y\in L$ are in relation $\rthin$ if $N(x)=N(y)$ and
  $N^-(x)=N^-(y)$. 
\end{definition} 
For each $\rthin$ class $\alpha$ we have $N(x)=N(\alpha)$ and
$N^-(x)=N^-(\alpha)$ for all $x\in\alpha$. It is obvious, therefore, that
$\rthin$ is an equivalence relation on the vertex set of $G$. Moreover,
since we consider loop-free graphs, one can easily see that $G[\alpha]$ is
always edge-less.  We write $\mathcal{N}$ for the corresponding partition,
i.e., the set of $\rthin$ classes of $G$. Individual $\rthin$ classes will
be denoted by lowercase Greek letters. Moreover, we write $N_s(x)=\{z\mid
z\in N(x) \text{ and } \sigma(z)=s\}$ and $N^-_s(x)=\{z\mid z\in N^-(x)
\text{ and } \sigma(z)=s\}$ for the in- and out-neighborhoods of $x$
restricted to a color $s\in S$.  For the graphs considered here, we always
have $N_{\sigma(x)}(x) = N^-_{\sigma(x)}(x) = \emptyset$. When considering
sets $N_s(x)$ and $N^-_s(x)$ we always assume that $s\neq\sigma(x)$.
Furthermore, $\mathcal{N}_s$ denotes the set of $\rthin$ classes with color
$s$.

By construction, the function $N: V(G)\to$ \NEW{$\mathcal{P}(V(G))$, where
  $\mathcal{P}(V(G))$ is the power set of $V(G)$,} is isotonic, i.e.,
$A\subseteq B$ implies $N(A)\subseteq N(B)$. In particular, therefore, we
have for $\alpha,\beta\in\mathcal{N}$:
\begin{description}
\item[(i)] $\alpha \subseteq N(\beta)$ implies $N(\alpha) \subseteq
  N(N(\beta))$
\item[(ii)] $N(\alpha)\subseteq N(\beta)$ implies
  $N(N(\alpha))\subseteq N(N(\beta))$.
\end{description} 
These observations will be useful in the proofs below. 

By construction every vertex in a cBMG has at least one out-neighbor of
every color except its own, i.e., $|N(x)| \ge |S|-1$ holds for all $x$. In
contrast, $N^-(x)=\emptyset$ is possible.

\subsection{Some Simple Observations} 

The color classes $L[s]$ on the leaves of $T$ are independent sets in
$G(T,\sigma)$ since arcs in $G(T,\sigma)$ connect only vertices with
different colors. For any pair of colors $s,t\in S$, therefore, the induced
subgraph $G[L[s]\cup L[t]]$ of $G(T,\sigma)$ is bipartite.  Since the
definition of $x \bmr y$ does not depend on the presence or absence of
vertices $u$ with $\sigma(u)\notin\{\sigma(x),\sigma(y)\}$, we have
\begin{fact} \label{obs-1}
  Let $(G,\sigma)$ be a cBMG explained by $T$ and let $L' :=
  \bigcup_{s\in S'} L[s]$ be the subset of vertices with a restricted
  color set $S'\subseteq S$. Then the induced subgraph $(G[L'],\sigma)$
  is explained by the restriction $T_{L'}$ of $T$ to the leaf set $L'$. 
\end{fact} 
It follows in particular that $G[L[s]\cup L[t]]$ is explained by the
restriction $T_{L[s]\cup L[t]}$ of $T$ to the colors $s$ and $t$.
Furthermore, $G$ is the edge-disjoint union of bipartite subgraphs
corresponding to color pairs, i.e., 
\begin{equation*}
  E(G) = \dot\bigcup_{\{s,t\}\in \binom{S}{2}} E(G_{s,t}).
\end{equation*}
In order to understand when arbitrary graphs $(G,\sigma)$ are cBMGs, it is
sufficient, therefore, to characterize 2-cBMGs. \NEW{A formal proof will be
  given later on in section \ref{sect:ncBMG}.}

\begin{figure}[htb]
\begin{center}
  \includegraphics[width=0.95\textwidth]{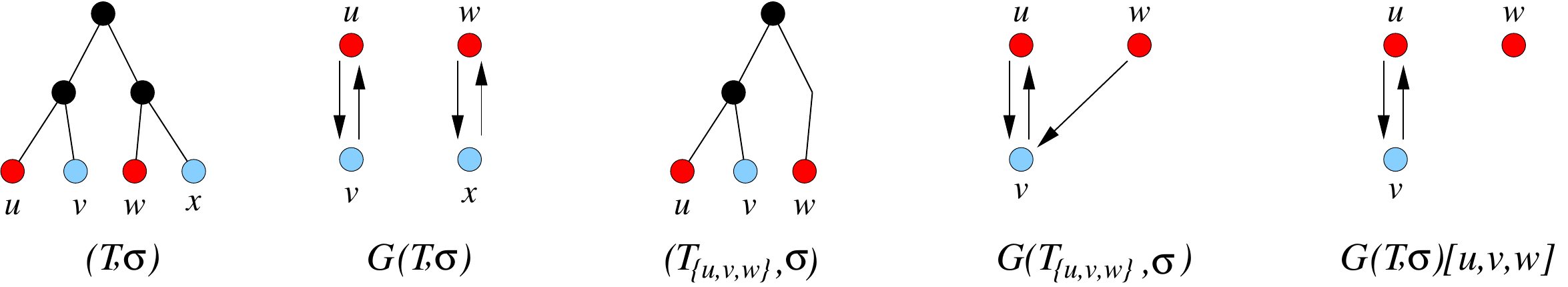} 
\end{center}
\caption{ $T_{\{u,v,w\}}$ is displayed by $T$ but $G(T_{\{u,v,w\}},\sigma)$
  is not isomorphic to the induced subgraph $G(T,\sigma)[\{u,v,w\}]$ of
  $G(T,\sigma)$, since $G(T_{\{u,v,w\}},\sigma)$ contains the additional arc
  $w\to v$. } 
\label{fig:counter1}
\end{figure}

Note the condition that ``$T$ explains $(G,\sigma)$'' does not imply that
$(T_{L'},\sigma)$ explains $(G[L'],\sigma)$ for arbitrary subsets of
$L'\subseteq L$. Fig.~\ref{fig:counter1} shows that, indeed, not every
induced subgraph of a cBMG is necessarily a cBMG.  However, we have the
following, weaker property:
\begin{lemma} \label{lem:sub}
  Let $(G,\sigma)$ be the cBMG explained by $(T,\sigma)$, let $T'=T_{L'}$
  and let $(G',\sigma)$ be the cBMG explained by $(T',\sigma)$. Then $u,v
  \in L'$ and $uv\in E(G)$ implies $uv\in E(G')$.  In other words,
  $(G[L'],\sigma)$ is always a subgraph of $(G'[L'],\sigma)$.
\end{lemma}
\begin{proof} 
  If $uv\in E(G)$ then $\lca_T(u,v) \preceq_T \lca_T(u,z)$ for all $z\in
  L[\sigma(v)]$, and thus the inequality $\lca_{T'}(u,v)\preceq_{T'}
  \lca_{T'}(u,z)$ is true in particular for all $z\in L'\cap
  L[\sigma(v)]=L'[\sigma(v)]$.
\end{proof}

\subsection{Connectedness} 

We briefly present some results concerning the connectedness of cBMGs.  In
particular, it turns out that connected cBMGs have a simple
characterization in terms of their representing trees.

\begin{theorem} \label{thm:connected} Let $(T,\sigma)$ be a leaf-labeled
  tree and $G(T,\sigma)$ its cBMG. Then $G(T,\sigma)$ is connected if and
  only if there is a child $v$ of the root $\rho$ such that
  $\sigma(L(T(v)))\ne S$.  Furthermore, if $G(T,\sigma)$ is not connected,
  then for every connected component $C$ of $G(T,\sigma)$ there is a child
  $v$ of the root $\rho$ such that $V(C)\subseteq L(T(v))$.
\end{theorem} 
\begin{proof} 
  For convenience we write $L_v:=L(T(v))$.  Suppose $\sigma(L_v)=S$ holds
  for all children $v$ of the root. Then for any pair of colors $s,t\in S$
  we find for a leaf $x\in L_v$ with $\sigma(x)=s$ a leaf $y\in L_v$ with
  $\sigma(y)=t$ within $T(v)$; thus $\lca(x,y)$ is in $T(v)$ and thus
  $\lca(x,y)\prec \rho$. Hence, all best matching pairs are confined to the
  subtrees below the children of the root. The corresponding leaf sets are
  thus mutually disconnected in $G(T,\sigma)$.

  Conversely, suppose that one of the children $v$ of the root $\rho$
  satisfies $\sigma(L_v)\neq S$.  Therefore, there is a color $t\in S$ with
  $t\notin \sigma(L_v)$.  Then for every $x\in L_v$ there is an arc $x\bmr
  z$ for all $z\in L[t]$ since for all such $z$ we have $\lca(x,z)=\rho$.
  If $L[t] = L\setminus L_v$, we can conclude that $G(T,\sigma)$ is a
  connected digraph. Otherwise, every leaf $y\in L\setminus L_v$ with a
  color $\sigma(y)\neq t$ has an out-arc $y\bmr z$ to some $z\in L[t]$ and
  thus there is a path $y\rightarrow z \leftarrow x$ connecting $y\in
  L\setminus L_v$ to every $x\in L_v$. Finally, for any two vertices
  $y,y'\in L\setminus (L_v\cup L[t])$ there are vertices $z,z'\in L[t]$
  such that arcs exist that form a path $y\rightarrow z \leftarrow x
  \rightarrow z' \leftarrow y'$ connecting $z$ with $z'$ and both to any
  $x\in L_v$. In summary, therefore, $G(T,\sigma)$ is a connected digraph.

  For the last statement, we argue as above and conclude that if
  $\sigma(L_v) = S$ for all children $v$ of the root (or, equivalently, if
  $G(T,\sigma)$ is not connected), then all best matching pairs are
  confined to the subtrees below the children of the root $\rho$. Thus, the
  vertices of every connected component of $G(T,\sigma)$ must be leaves of
  a subtree $T(v)$ for some child $v$ of the root $\rho$.   
\end{proof}
 
The following result shows that cBMGs can be characterized by their
connected components: the disjoint union of vertex disjoint cBMGs is again
a cBMG if and only if they all share the same color set. It suffices
therefore, to consider each connected component separately.

\begin{proposition} 
  Let $(G_i,\sigma_i)$ be vertex disjoint cBMGs with vertex sets $L_i$ and
  color sets $S_i=\sigma_i(L_i)$ for $1\le i \le k$. Then the disjoint union
  $(G,\sigma):=\dot\bigcup_{i=1}^k(G_i,\sigma_i)$ is a cBMG if and only if
  all color sets are the same, i.e., $\sigma_i(L_i)=\sigma_j(L_j)$ for
  $1\le i,j\le k$.
\end{proposition} 
\begin{proof}
  The statement is trivially fulfilled for $k=1$. For $k\ge 2$, the
  disjoint union $(G,\sigma)$ is not connected. Assume that
  $\sigma_i(L_i)=\sigma_j(L_j)$ for all $i,j$. Let $(T_i,\sigma_i)$ be
  trees explaining $(G_i,\sigma_i)$ for $1\le i \le k$. We construct a tree
  $(T,\sigma)$ as follows: Let $\rho$ be the root of $(T,\sigma)$ with
  children $r_1, \dots r_k$. Then we identify $r_i$ with the root of $T_i$
  and retain all leaf colors. In order to show that $(T,\sigma)$ explains
  $(G,\sigma)$ we recall from Thm.\ \ref{thm:connected} that all best
  matching pairs are confined to the subtrees below the children of the
  root and hence, each connected component of $(G,\sigma)$ forms a subset
  of one of the leaf sets $L_i$. Since each $(T_i,\sigma_i)$ explains
  $(G_i,\sigma_i)$, we conclude that the cBMG explained by $(T,\sigma)$ is
  indeed the disjoint union of the $(G_i,\sigma_i)$, i.e.,
  $(G,\sigma)$. Thus $(G,\sigma)$ is a cBMG.

  Conversely, assume that $(G,\sigma)$ is a cBMG but
  $\sigma_i(L_i)\neq\NEW{\sigma_k(L_k)}$ for some $\NEW{k}\neq i$. By
  construction, $\sigma(L_i) = \sigma_i(L_i)$ and
  $\sigma(\NEW{L_k})=\NEW{\sigma_k(L_k)}$.  In particular, for every color
  $t\notin \sigma(L_i)$ and every vertex $x\in L_i$, there is a $j\neq i$
  with $t\in \sigma(L_j)$ such that there exists an outgoing arc form $x$
  to some vertex $y\in L_j$ with color $\sigma(y)=t$. Thus $(x,y)$ is an
  arc connecting $L_i$ with some $L_j$, $j\ne i$, contradicting the
  assumption that each $L_i$ forms a connected component of
  $(G,\sigma)$. Hence, the color sets cannot differ between connected
  components.   
\end{proof}

The example $(G(T_{\{u,v,w\}}),\sigma)$ in Fig.~\ref{fig:counter1} already
shows however that $G(T,\sigma)$ is not necessarily strongly connected.

\section{Two-Colored Best Match Graphs (2-cBMGs)} 
\label{sect:2cBMG}

\emph{\textbf{Through this section we assume that $\sigma(L)=\{s,t\}$
    contains exactly two colors.}}

\marginpar{\color{blue}\scriptsize All graphs in this section are assumed
  to be sink-free.}
\subsection{Thinness Classes} 

A connected 2-cBMG contains at least two $\rthin$ classes, since all in-
and out-neighbors $y$ of $x$ by construction have a color $\sigma(y)$
different from $\sigma(x)$. Consequently, a 2-cBMG is
bipartite. Furthermore, if $\sigma(x)\ne\sigma(y)$ then
$N(x)\cap N(y)=\emptyset$. Since $N(x)\ne\emptyset$ and all members of
$N(x)$ have the same color, we observe that $N(x)=N(y)$ implies
$\sigma(x)=\sigma(y)$. \NEW{By a slight abuse of notation we will often
  write $\sigma(x)=\sigma(\alpha)$ for an element $x$ of some $\rthin$
  class $\alpha$.} Two leaves $x$ and $y$ of the same color that have the
same last common ancestor with all other leaves in $T$, i.e., that satisfy
$\lca(x,u)=\lca(y,u)$ for all $u\in L\setminus\{x,y\}$ by construction have
the same in-neighbor and the same out-neighbors in $G(T,\sigma)$; hence
$x \rthin y$.

\begin{fact}\label{fact:class-species}
  Let $(G,\sigma)$ be a connected 2-cBMG and $\alpha\in\mathcal{N}$
  be a $\rthin$ class.  Then, $\sigma(x) =\sigma(y)$ for any
  $x,y\in \alpha$.
\end{fact}

The following result shows that the out-neighborhood of any $\rthin$ class
is a disjoint union of $\rthin$ classes.

\begin{lemma}\label{lem:rthin-prop}
  Let $(G,\sigma)$ be a connected 2-cBMG. Then any two $\rthin$
  classes $\alpha,\beta\in\mathcal{N}$ satisfy
  \begin{description} 
    \item[\AX{(N0)}] $\beta\subseteq N(\alpha)$ or 
                $\beta\cap N(\alpha)=\emptyset$.
  \end{description}
\end{lemma}
\begin{proof}
  For any $y\in \beta$, the definition of $\rthin$ classes implies that
  $y\in N(\alpha)$ if and only if $\beta\subseteq N(\alpha)$. Hence,
  either all or none of the elements of $\beta$ are contained in
  $N(\alpha)$.
 
\end{proof}

The connection between the $\rthin$ classes of $G(T,\sigma)$ and the tree
$(T,\sigma)$ is captured by identifying an internal node in $T$ that is, as
we shall see, in a certain sense characteristic for a given equivalence
class.
\begin{definition}
  The \emph{root $\rho_{\alpha}$ of the $\rthin$ class $\alpha$} is 
  \begin{equation*}
    \rho_\alpha=\max_{\substack{x\in\alpha \\ y\in N(\alpha)}} \lca(x,y).
  \end{equation*}
\end{definition}

\begin{corollary}\label{cor:rthin-def}
  Let $\rho_\alpha$ be the root of a $\rthin$ class $\alpha$. Then, for any
  $y\in N(\alpha)$ holds
 \begin{equation*}
  \rho_\alpha=\max_{x\in \alpha} \lca(x,y).
 \end{equation*}
In particular, $\lca(x,y) = \lca(x,z)$ for all $y,z\in  N(\alpha)$.
\end{corollary}
\begin{proof}
  For any $y\in N(\alpha)$ it holds by definition of $N(\alpha)$ that
  $\lca(x,y)\preceq\lca(x,z)$ for $x\in \alpha$ and any $z$ with
  $\sigma(z)=\sigma(y)$. This together with Observation
  \ref{fact:class-species} implies that $\lca(x,y)=\lca(x,z)$ for any two
  $y,z\in N(\alpha)$ and $x\in\alpha$.   
\end{proof}

\begin{figure}
      \begin{center}
        \includegraphics[width=0.2\textwidth]{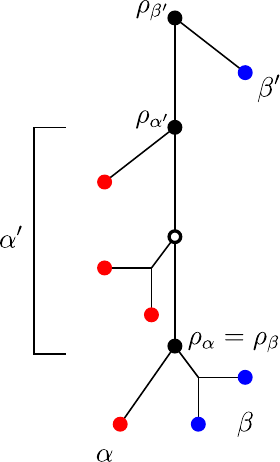}
      \end{center}
      \caption[]{Relationship between $\rthin$ classes and their roots. A
        tree with two colors (red and blue) and four $\rthin$ classes
        $\alpha$, $\alpha'$ (red) and $\beta$, $\beta'$ (blue) together
        with their corresponding roots $\rho_\alpha$, $\rho_{\alpha'}$,
        $\rho_\beta$ and $\rho_{\beta'}$ are shown. }
      \label{fig:rthin_root}
\end{figure}

The following lemma collects some simple properties of the roots of
$\rthin$ classes that will be useful for the proofs of the main results.

\begin{lemma}
  \label{lem:prp1}
  Let $(G,\sigma)$ be a connected 2-cBMG explained by $(T,\sigma)$ and let
  $\alpha$, $\beta$ be $\rthin$ classes with roots $\rho_{\alpha}$ and
  $\rho_{\beta}$, respectively. Then the following statements hold
  \begin{itemize}
    \setlength{\itemindent}{0.15in}
  \item[(i)] $\rho_{\alpha}\preceq \lca(\alpha,\beta)$ and
    $\rho_\beta\preceq\lca(\alpha,\beta)$; equality holds for at least one
    of them if and only if $\rho_\alpha, \rho_\beta$ are comparable, i.e.,
    $\rho_\alpha \preceq \rho_\beta$ or $\rho_\beta \preceq \rho_\alpha$.
  \item[(ii)] The subtree $T(\rho_{\alpha})$ contains leaves of both
    colors.
  \item[(iii)] $N(\alpha)\preceq \rho_{\alpha}$. 
  \item[(iv)] If $\beta\subseteq N(\alpha)$ then $\rho_{\beta}\preceq
    \rho_{\alpha}$.  
  \item[(v)]If $\rho_\alpha=\rho_\beta$ and $\alpha\neq\beta$, then
    $\sigma(\alpha)\neq\sigma(\beta)$.
  \item[(vi)] $N(\alpha)=\{y \mid y\in L(T(\rho_\alpha)) \textrm{ and }
    \sigma(y)\neq\sigma(\alpha)\}$
  \item[(vii)] $N(N(\alpha))\preceq \rho_\alpha$.
   \end{itemize}	
\end{lemma}
\begin{proof}
  (i) By Condition \AX{(N0)} in Lemma \ref{lem:rthin-prop} we have either
  $\beta\subseteq N(\alpha)$ or $\beta\cap N(\alpha)=\emptyset$.  By
  definition of $N(\beta)$, we have $\lca(x',y)\preceq \lca(x,y)$ where
  $y\in \beta$, $x'\in N(\beta)$, and $x\in \alpha$.
  Therefore, if $\beta\subseteq N(\alpha)$, then
  $\rho_\beta=\max_{x'\in N(\beta)}\lca(x',\beta)\preceq
  \max_{x\in\alpha}\lca(x,\beta)=\lca(\alpha,\beta)$.  Moreover, Cor.\
  \ref{cor:rthin-def} implies
  $\rho_\alpha=\max_{y\in N(\alpha)} \lca(\alpha,y)= \max_{y \in
    \beta}\lca(\alpha,y)= \lca(\alpha,\beta)$.
	
  If $\beta\cap N(\alpha)=\emptyset$, then
  $\lca(\alpha,y)\succ\max_{y'\in N(\alpha)}\lca(\alpha,y')=\rho_\alpha$
  for all $y\in\beta$, i.e., $\lca(\alpha,\beta)\succ \rho_\alpha$.
  Moreover, by definition of $\rho_\beta$, we have $\rho_\beta=\max_{x\in
    N(\beta)}\lca(x,\beta)\preceq \max_{x\in
    \alpha}\lca(x,\beta)=\lca(\alpha,\beta)$.

  Now assume that $\rho_\alpha$ and $\rho_\beta$ are comparable. W.l.o.g.\ 
  we assume $\rho_\alpha \succeq \rho_\beta$. Since 
  $\alpha\preceq \rho_\alpha$ and  $\beta\preceq \rho_\beta$ is true by 
  definition, we obtain $\lca(\alpha,\beta)=\rho_\alpha\succeq\rho_\beta$. 
  Conversely, if $\rho_\alpha =\lca(\alpha,\beta)\succeq \rho_\beta$, 
  then $\rho_\alpha$ and $\rho_\beta$ are necessarily comparable.

  \smallskip\par\noindent(ii) As argued above, $N(x)\neq \emptyset$ for
  all vertices $x$.  Let $x\in \alpha$ and $y\in N(x)$ such that
  $\rho_\alpha = \lca(x,y)$.  By definition, $\sigma(x)\neq
  \sigma(y)$. Since $\rho_\alpha$ is an ancestor of both $x$ and $y$, the
  statement follows.

  \smallskip\par\noindent(iii) 
  Since $T(\rho_{\alpha})$ contains leaves of both colors, there is
  in particular a leaf $y$ with $\sigma(y)\ne\sigma(x)$ within
  $T(\rho_{\alpha})$. It satisfies $\lca(x,y)\preceq \rho_{\alpha}$ and
  thus all arcs going out from $x\in\alpha$ are confined to leaves of
  $T(\rho_{\alpha})$, i.e., $N(\alpha)\preceq\rho_{\alpha}$.

  \smallskip\par\noindent(iv) 
  is a direct consequence of (i) and (iii).

  \smallskip\par\noindent(v) Assume for contradiction that
  $\sigma(\alpha)=\sigma(\beta)$.  There is some $y\in N(\alpha)$ with
  $\lca(\alpha,y)=\rho_\alpha$.  Since
  $\rho_\alpha=\rho_\beta=\lca(\alpha,\beta)$ by (i), we have
  $\lca(\alpha,y)\succeq\lca(\beta,y)$. By definition of $\rho_\beta$,
  there is a $z\in N(\beta)$ such that $\lca(\beta,z)=\rho_\beta$.  Thus,
  $\lca(\beta,y)\preceq\lca(\beta,z)$, which implies that $y$ is a best
  match of $\beta$, i.e., $y\in N(\beta)$. Hence, $N(\alpha)=N(\beta)$. On
  the other hand, since $\lca(\alpha,\beta)=\rho_\alpha$, we have
  $\lca(\alpha,y)=\lca(\beta,y)$ for any $y$ with
  $\lca(\alpha,y)\succeq\rho_\alpha$. As a consequence, since
  $\rho_\alpha\preceq\lca(\alpha,y')$ for all $y'\in N^-(\alpha)$, it is
  true that $\lca(y',\beta)=\lca(y',\alpha)\preceq\lca(y',z)$, for all $z$
  with $\sigma(z)=\sigma(\alpha)$. Hence $y\in N^-(\alpha)$ if and only if
  $y\in N^-(\beta)$. It follows that $\alpha=\beta$, a contradiction.

  \smallskip\par\noindent(vi) 
  Let $y\in N(\alpha)$, then $\sigma(y)\neq\sigma(\alpha)$ by 
  definition. In addition, we have $y\preceq\rho_\alpha$ by (iii). 
  Conversely, suppose that $y\in L(T(\rho_\alpha))$ and 
  $\sigma(y)\neq\sigma(\alpha)$. Since $y\in L(T(\rho_\alpha))$, it is true 
  that $y,\alpha\preceq\rho_\alpha$ and therefore, 
  $\lca(\alpha,y)\preceq\rho_\alpha$. By definition of the root of
  $\alpha$, there exist $x'\in\alpha$ and $y'\in N(\alpha)$ such that 
  $\rho_\alpha=\lca(x',y')\preceq\lca(x',z)$ for all $z$ with 
  $\sigma(z)=\sigma(y')$. Since $\lca(\alpha,y)\preceq\rho_\alpha$, 
  this implies $y\in N(\alpha)$.

  \smallskip\par\noindent(vii) 
  Lemma \ref{lem:rthin-prop} and (iv) imply that $N(\alpha)$ 
  is a disjoint union of $\rthin$ classes $\gamma$ with 
  $\rho_\gamma\preceq\rho_\alpha$ and $\sigma(\gamma)\neq\sigma(\alpha)$. 
  Thus, $N(N(\alpha))=
  \bigcup_{\substack{\gamma\in\mathcal{N}\\ 
      \gamma \subseteq N(\alpha)}}  N(\gamma)= 
  N(\bigcup_{\substack{\gamma\in\mathcal{N}\\ 
      \gamma \subseteq N(\alpha)}} \gamma)$. 
  By (iii) and (iv), we have $N(\gamma)\preceq \rho_\alpha$ for any 
  such $\gamma$, thus $N(N(\alpha))\preceq \rho_\alpha$. 
   
\end{proof}

\AX{(N0)} implies that there are four distinct ways in which two $\rthin$
classes $\alpha$ and $\beta$ with distinct colors can be related to each
other. These cases distinguish the relative location of their roots
$\rho_\alpha$ and $\rho_\beta$:

\begin{lemma} 
  \label{lem:rthin-cases} 
  If $(G,\sigma)$ is a connected 2-cBMG, and $\alpha$, $\beta$ are $\rthin$
  classes with $\sigma(\alpha)\ne\sigma(\beta)$.  Then exactly one of the
  following four cases is true
  \begin{itemize}
    \setlength{\itemindent}{0.15in} 
  \item[(i)] $\alpha\subseteq N(\beta)$ and $\beta\subseteq N(\alpha)$.
    In this case $\rho_{\alpha}=\rho_{\beta}$.
  \item[(ii)] $\alpha\subseteq N(\beta)$ and $\beta\cap
    N(\alpha)=\emptyset$.
    In this case $\rho_{\alpha}\prec\rho_{\beta}$.
  \item[(iii)] $\beta\subseteq N(\alpha)$ and $\alpha\cap
    N(\beta)=\emptyset$.
    In this case $\rho_{\beta}\prec\rho_{\alpha}$.
  \item[(iv)] $\alpha\cap N(\beta)=\beta\cap N(\alpha)=\emptyset$.
    In this case $\rho_{\alpha}$ and $\rho_{\beta}$ are not 
      $\preceq$-comparable.
  \end{itemize}
\end{lemma}
\begin{proof}
  Set $\sigma(\alpha)=s$ and $\sigma(\beta)=t$, $s\ne t$, and consider the
  roots $\rho_\alpha$ and $\rho_\beta$ of the two $\rthin$ classes. Then,
  there are exactly four cases:

  (i) For $\rho_\alpha=\rho_\beta$, Lemma \ref{lem:prp1}(i) implies
  $\rho_\alpha=\rho_\beta=\lca(\alpha,\beta)$. By definition of
  $\rho_\alpha$, $y\in N(\alpha)$ for all $y\in L(T(\rho_\alpha))$ with 
  $\sigma(y)\neq\sigma(\alpha)$ by Lemma \ref{lem:prp1}(vi).
  A similar result holds for $\rho_\beta$. It follows immediately that
  $\alpha\subseteq N(\beta)$ and $\beta\subseteq N(\alpha)$.

  (ii) In the case $\rho_\alpha\succ\rho_\beta$, Lemma \ref{lem:prp1}(i) 
  implies $\rho_\alpha=\lca(\alpha,\beta)$ and thus, similarly to case (i), 
  $\beta\subseteq N(\alpha)$. On the other hand, by Lemma
  \ref{lem:prp1}(ii) and $\rho_\alpha\succ\rho_\beta$, there is a leaf 
  $x'\in L(T(\rho_\beta))\setminus \alpha$ with $\sigma(x')=s$. 
  Hence, $\lca(x',\beta)\prec\rho_\alpha=\lca(\alpha,\beta)$, which 
  implies $\alpha\cap N(\beta)=\emptyset$.

  (iii) The case $\rho_\alpha\prec\rho_\beta$ is symmetric to (ii).

  (iv) If $\rho_\alpha,\rho_\beta$ are incomparable, it yields
  $\rho_\alpha, \rho_\beta\neq\rho$ and $\lca(\alpha,\beta)=\rho$, where
  $\rho$ denotes the root of $T$. Since $\beta\preceq \rho_\beta$, 
  Lemma \ref{lem:rthin-prop} implies $\beta\cap  N(\alpha)=\emptyset$. 
  Similarly, $\alpha\cap N(\beta)=\emptyset$. 
   
\end{proof}

\subsection{Least resolved Trees}

In general, there are many trees that explain the same 2-cBMG. We next show
that there is a unique ``smallest'' tree among them, which we will call the
least resolved tree for $(G,\sigma)$. Later-on, we will derive a hierarchy
of leaf sets from $(G,\sigma)$ whose tree representation coincides with
this least resolved tree. We start by introducing a systematic way of
simplifying trees. Let $e$ be an interior edge of $(T,\sigma)$. Then the
tree $T_e$ obtained by contracting the edge $e=uv$ is derived by
identifying $u$ and $v$. Analogously, we write $T_A$ for the tree obtained
by contracting all edges in $A$.

\begin{definition} 
  Let $(G,\sigma)$ be a cBMG and let $(T,\sigma)$ be a tree explaining
  $(G,\sigma)$. An interior edge $e$ in $(T,\sigma)$ is \emph{redundant} if
  $(T_e,\sigma)$ also explains $(G,\sigma)$. Edges that are not redundant
  are called \emph{relevant}.
\end{definition}

The next two results characterize redundant edges and show that such edges
can be contracted in an arbitrary order.

\begin{lemma}\label{lem:red}
  Let $(T,\sigma)$ be a tree that explains a connected 2-cBMG
  $(G,\sigma)$. Then, the edge $e=uv$ is redundant if and only if $e$ is an
  inner edge and there exists no $\rthin$ class $\alpha$ such that
  $v=\rho_\alpha$.
\end{lemma}
\begin{proof}
  First we note that $e=uv$ must be an inner edge. Otherwise, i.e., if $e$
  is an outer edge, then $v\notin L(T_e)$ and thus, $(T_e,\sigma)$ does not
  explain $(G,\sigma)$.  Now suppose that $e$ is an inner edge, which in
  particular implies $L(T_e)=L(T)$, \NEW{ and that $e$ is redundant. Assume
    for contradiction that} there is a $\rthin$ class $\alpha$ such that
  $v=\rho_\alpha$. Since $(T,\sigma)$ is phylogenetic, $T(u)\setminus T(v)$
  has to be non-empty.  If there is a leaf $y\in L(T(u)\setminus T(v))$
  with $\sigma(y)\neq\sigma(\alpha)$ in $(T,\sigma)$, then
  $y\notin N(\alpha)$ by Lemma \ref{lem:prp1}(vi). But then, contraction of
  $e$ implies $y\in T(\rho_\alpha)$ and therefore $y\in N(\alpha)$, thus
  $(T_e,\sigma)$ does not explain $(G,\sigma)$.  Consequently,
  $T(u)\setminus T(v)$ can only contain leaves $x$ with
  $\sigma(x)=\sigma(\alpha)$. Indeed, for any $y' \in T(v)$ it is true that
  $v=\rho_\alpha=\lca(\alpha,y')\prec \lca(x,y')$, i.e.,
  $N^-(x)\neq N^-(\alpha)$ and thus $x\notin \alpha$. By contracting $e$,
  we obtain $\lca(x,z)\succeq uv=\rho_\alpha$ which implies
  $N(x)=N(\alpha)$ and $N^-(x)=N^-(\alpha)$, and therefore $x\in
  \alpha$. Hence, $(T_e,\sigma)$ does not explain $(G,\sigma)$.

  Conversely, assume that $e$ is an inner edge and there is no $\rthin$
  class $\alpha$ such that $v=\rho_{\alpha}$, i.e., \NEW{for each
    $\alpha\in\mathcal{N}$ it either holds (i) $v\prec \rho_\alpha$, (ii)
    $v\succ \rho_\alpha$, or (iii) $v$ and $\rho_\alpha$ are
    incomparable. In the first and second case,} contraction of $e$ implies
  either $v\preceq \rho_\alpha$ or $v\succeq \rho_\alpha$. Thus, \NEW{since
    $L(T(w))=L(T_e(w))$ is clearly satisfied if $w$ and $v$ are
    incomparable,} we have $L(T(w))=L(T_e(w))$ for every $w\neq
  v$. Moreover,
  $N(\alpha)=\{y \mid y\in L(T(\rho_\alpha)),\sigma(y)\neq\sigma(\alpha)\}$
  by Lemma \ref{lem:prp1}(vi). Together these facts imply for every
  $\rthin$ class $\alpha$ with $\rho_\alpha\neq v$ that $N(\alpha)$ remains
  unchanged in $(T_e,\sigma)$ after contraction of $e$. Since the
  out-neighborhoods of all $\rthin$ classes are unaffected by contraction
  of $e$, all in-neighborhoods also remain the same in
  $(T_e,\sigma)$. Therefore, $(T,\sigma)$ and $(T_e,\sigma)$ explain the
  same graph $(G,\sigma)$.   
\end{proof}

\begin{lemma}\label{lem:red2}
  Let $(T,\sigma)$ be a tree that explains a connected 2-cBMG $(G,\sigma)$
  and let $e$ be a redundant edge. Then the edge $f\neq e$ is redundant in
  $(T_e,\sigma)$ if and only if $f$ is redundant in $(T,\sigma)$. Moreover,
  if two edges $e\ne f$ are redundant in $(T,\sigma)$, then
  $((T_e)_f,\sigma)$ also explains $(G,\sigma)$.
\end{lemma}
\begin{proof}
  Let $e=uv$ be a redundant edge in $(T,\sigma)$. Then, for any vertex
  $w\neq u,v$ in $(T,\sigma)$ it is true that $w$ is the root of a $\rthin$
  class $\alpha$ in $(T_e,\sigma)$ if and only if $w$ is the root of
  $\alpha$ in $(T,\sigma)$. In particular, the vertex $uv$ in
  $(T_e,\sigma)$ is the root of a $\rthin$ class $\alpha'$ if and only if
  $u=\rho_{\alpha'}$ in $(T,\sigma)$. Consequently, $f$ is redundant in
  $(T,\sigma)$ if and only if $f$ is redundant in $(T_e,\sigma)$.   
\end{proof}
As an immediate consequence, contraction of edges is commutative, i.e., the
order of the contractions is irrelevant. We can therefore write $T_A$ for
the tree obtained by contracting all edges in $A$ in arbitrary order:
\begin{corollary}\label{cor:redset}
  Let $(T,\sigma)$ be a tree that explains a 2-cBMG $(G,\sigma)$ and let
  $A$ be a set of redundant edges of $(T,\sigma)$. Then, $(T_A,\sigma)$
  explains $(G,\sigma)$. In particular, $((T_A)_B,\sigma)$ explains
  $(G,\sigma)$ if and only if $B$ is a set of redundant edges of
  $(T,\sigma)$.
\end{corollary}

\begin{definition}
  Let $(G,\sigma)$ be a cBMG explained by $(T,\sigma)$. We say that
  $(T,\sigma)$ is \emph{least resolved} if $(T_A,\sigma)$ does not explain
  $(G,\sigma)$ for any non-empty set $A$ of interior edges of $(T,\sigma)$.
\end{definition} 
We are now in the position to formulate the main result of this section:

\begin{theorem}\label{thm:lr-2}
  For any connected 2-cBMG $(G,\sigma)$, there exists a unique least
  resolved tree $(T',\sigma)$ that explains $(G,\sigma)$. $(T',\sigma)$ is
  obtained by contraction of all redundant edges in an arbitrary tree
  $(T,\sigma)$ that explains $(G,\sigma)$. The set of all redundant edges
  in $(T,\sigma)$ is given by
  \begin{equation*}
    \mathfrak{E}_T=\{e=uv \mid v\notin L(T) \text{ and there is no } 
    \rthin \text{ class } \alpha \text{ such that } v=\rho_\alpha \}.
  \end{equation*}
  Moreover, $(T',\sigma)$ is displayed by $(T,\sigma)$.
\end{theorem}
\begin{proof}
  Any edge in a least resolved tree $(T',\sigma)$ is non-redundant and
  therefore, as a consequence of Cor. \ref{cor:redset}, $(T',\sigma)$ is
  obtained from $(T,\sigma)$ by contraction of all redundant edges of
  $(T,\sigma)$. According to Lemma \ref{lem:red}, the set of redundant
  edges is exactly $\mathfrak{E}_T$. Since the order of contracting the
  edges in $\mathfrak{E}_T$ is arbitrary, there is a least resolved tree
  for every given tree $(T,\sigma)$.

  Now assume for contradiction that there exist colored digraphs that are
  explained by two distinct least resolved trees. Let $(G,\sigma)$ be a
  minimal graph (w.r.t.\ the number of vertices) that is explained by two
  distinct least resolved trees $(T_1,\sigma)$ and $(T_2,\sigma)$ and let
  $v\in L$ with $\sigma(v)=s$. By construction, the two trees
  $(T_1',\sigma')$ and $(T_2',\sigma')$ with
  $T_1':=T_1{_{|L\setminus\{v\}}}$, $T_2':=T_2{_{|L\setminus\{v\}}}$ and
  leaf labeling $\sigma':=\sigma_{|L\setminus \{v\}}$, each explain a
  unique graph, which we denote by $(G_1',\sigma')$ and $(G_2',\sigma')$,
  respectively.  Lemma \ref{lem:sub} implies that
  $(G',\sigma'):=(G[L\setminus \{v\}],\sigma')$ is a subgraph of both
  $(G_1',\sigma')$ and $(G_2',\sigma')$.

  We next show that $(G_1',\sigma')$ and $(G_2',\sigma')$ are equal by
  characterizing the additional edges that are inserted in both graphs
  compared to $(G',\sigma')$. Assume that there is an additional edge $uy$
  in one of the graphs, say $(G_1',\sigma)$. Since $uy$ is not an edge in
  $(G,\sigma)$, we have $\lca_T(u,y)\succ_T \lca_T(u,y')$ for some $y'\in
  L(T)$ with $\sigma(y)=\sigma(y')$.  However, $uy\in E(G_1')$ implies that
  $\lca_{T_1}(u,y)\preceq_{T_1}\lca_{T_1}(u,y'')$ for all $y''\in
  L\setminus \{v\}$ with $\sigma(y)=\sigma(y')$.  Since $T_1':=T_1\setminus
  \{v\}$, we obtain $\lca_T(u,y') \prec_T \lca_T(u,y) \preceq_{T}
  \lca_{T}(u,y'')$, which implies that $y'=v$ and, in particular, $uv\in
  E(G)$ and $N(u) = \{v\}$.
	
  In particular, we have $\sigma(u)=t\neq s$. In this case, $u$ has no
  out-neighbors in $(G',\sigma')$ but it has outgoing arcs in
  $(G_1',\sigma')$ and $(G_2',\sigma')$. In order to determine these
  outgoing arcs explicitly, we will reconstruct the local structure of
  $(T_1,\sigma)$ and $(T_2,\sigma)$ in the vicinity of the leaf $v$. The
  following argumentation is illustrated in Fig.\ \ref{fig:lr_unique}.

  \begin{figure}[h]
    \begin{center} 
          \includegraphics[width=0.2\textwidth]{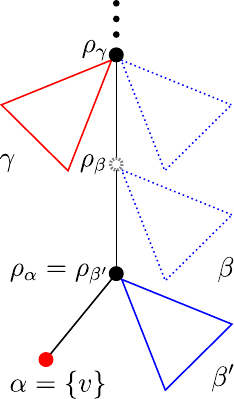}
          \caption[]{Illustration of the proof of Theorem~\ref{thm:lr-2},
            showing the local subtrees of $(T_1,\sigma)$ and
            $(T_2,\sigma)$, immediately above $\alpha=\{v\}$. The relevant
            portion extends to the root $\rho_{\gamma}$ of the $\rthin$
            class $\gamma$ that is located immediately above of $\alpha$
            and has the same color as $\alpha$, here red. Clearly, the
            deletion of $\alpha$ can affect only pairs of vertices $x,y$
            with $\lca(x,y)$ below $\rho_{\gamma}$.  Triangles denote the
            subtree that consists of all leaves of the corresponding class
            which are attached to the root of the class by an outer
            edge. Dashed triangles and nodes denote subtrees which may or
            may not be present in $(T_1,\sigma)$ and $(T_2,\sigma)$.}
          \label{fig:lr_unique} 
    \end{center}
  \end{figure}    
    
  Since $N(u)=\{v\}$, there is a $\rthin$ class $\alpha=\{v\}$. Let $\beta$
  be the $\rthin$ class of $(G,\sigma)$ to which $u$ belongs. It satisfies
  $N(\beta)=\{v\}$. Therefore, $L(T_1(\rho_\beta))\cap L[s]=\{v\}$ and
  $L(T_2(\rho_\beta))\cap L[s]=\{v\}$. In particular, this implies
  $L(T_1(\rho_\alpha))\cap L[s]=\{v\}$ and $L(T_2(\rho_\alpha))\cap
  L[s]=\{v\}$. The children of $\rho_\alpha$ in both $T_1$ and $T_2$ must
  be leaves: otherwise, Lemma \ref{lem:prp1}(ii) would imply that there are
  inner vertices $\rho_{\alpha'}$ and $\rho_{\beta'}$ below $\rho_\alpha$,
  which in turn would contradict to $L(T_1(\rho_\alpha))\cap L[s]=\{v\}$
  and $L(T_2(\rho_\alpha))\cap L[s]=\{v\}$.  

  Moreover, the subtrees $T_1(\rho_\alpha)$ and $T_2(\rho_\alpha)$ must
  contain leaves of both colors. Thus there exists a $\rthin$ class
  $\beta'$ with color $t$ whose root $\rho_{\beta'}$ coincides with
  $\rho_\alpha$ in both $(T_1,\sigma)$ and $(T_2,\sigma)$. More precisely,
  we have $\child(\rho_\alpha)=\alpha\cup \beta'$. We now distinguish two
  cases:
    
  \par\noindent(i)\quad If $N^-(\beta)\cap\{v\}\neq \emptyset$ in
  $(G,\sigma)$, we have $\rho_\beta=\rho_\alpha$, i.e., $\beta=\beta'$.
    
  \par\noindent (ii)\quad Otherwise if $N^-(\beta)\cap\{v\}= \emptyset$, 
  then $\lca(v,\beta')\prec \lca(v,\beta)$, hence $\rho_\beta\succ
  \rho_\alpha$. In particular, since $N(\beta)=\{v\}$, Lemma
  \ref{lem:prp1}(vi) implies that there cannot be any other $\rthin$ class
  $\alpha'\neq \alpha$ of $(G,\sigma)$ with color $s$ and
  $\rho_\beta\succeq\rho_{\alpha'}$. Moreover, there cannot be any other
  class $\beta''$ of color $t$ \NEW{such that $\rho_{\beta''}$ is
    contained} in the unique path from $\rho_\beta$ to
  $\rho_\alpha$, otherwise it holds $N(\beta'')=N(\beta)$ and
  $N^-(\beta'')=N^-(\beta)$ by Lemma \ref{lem:prp1}(vi), i.e.,
  $\beta''\rthin\beta$. Therefore, we conclude that $\rho_\beta \rho_\alpha
  \in E(T_1)$ as well as $\rho_\beta \rho_\alpha \in E(T_2)$.
  
  \par\noindent If $v$ is the only leaf of color $s$ in $(G,\sigma)$,
  it follows from (i) and (ii) that
  $(\NEW{T_1'},\sigma')=(T_1(\rho_\beta),\sigma')=(T_2(\rho_\beta),\sigma')
  =(\NEW{T_2'},\sigma')$; a contradiction\NEW{, hence there is a unique
    tree representation for $(G,\sigma)$ if $|L[s]|=1$.}.

  \par\noindent Now suppose that $L[s]>1$. \NEW{Then, both in case (i) and
    case (ii) there is} a parent \NEW{of} $\parent(\rho_\beta)$,
  \NEW{because} otherwise $(G_1',\sigma')$ and $(G_2',\sigma')$ would not
  contain color $s$. \NEW{In either case} the parent of $\rho_\beta$ is an
  inner node of the least resolved tree $(T_1,\sigma')$ and
  $(T_2,\sigma')$\NEW{, respectively}. We claim that $\parent(\rho_\beta)$
  is the root of $\rthin$ class $\gamma$ of color $s$. Suppose this is not
  the case, i.e., $\sigma(\gamma)=t$ and there is no other
  $\gamma'\in\mathcal{N}$ such that $\sigma(\gamma')=s$ and
  $\parent(\rho_\beta)=\rho_{\gamma'}$. Then $N(\gamma)=N(\beta)$ and
  $N^-(\gamma)=N^-(\beta)$ by Lemma \ref{lem:prp1}(vi), which implies that
  $\beta\rthin \gamma$ and $\rho_\beta$ is not the root of $\beta$; a
  contradiction.

  We therefore conclude that the local subtrees of $(T_1,\sigma')$ and
  $(T_2,\sigma')$ immediately above $\alpha$\NEW{, that is
    $(T_1(\rho_\gamma), \sigma'_{|L(T_1(\rho_\gamma))})$ and
    $(T_2(\rho_\gamma), \sigma'_{|L(T_2(\rho_\gamma))})$,}
	as indicated in Fig. \ref{fig:lr_unique},
  are identical.  Moreover, it follows that
  $\lca(u,\gamma)\preceq \lca(u,w)$ for any $w\in L[s]\setminus \{v\}$.
  Hence, the additionally inserted edges in $(G_1',\sigma)$ and
  $(G_2',\sigma)$ are exactly the edges $uc$ for all $c\in\gamma$. We
  therefore conclude that $(G_1',\sigma)=(G_2',\sigma)$, which implies
  $(T_1',\sigma')=(T_2',\sigma')$. Since $v$ has been chosen arbitrarily,
  this implies $(T_1,\sigma)=(T_2,\sigma)$; a contradiction.   
\end{proof}

Finally, we consider a few simple properties of least resolved trees that
will be useful in the following sections.

\begin{corollary}\label{cor:rel_root}
  Let $(G,\sigma)$ be a connected 2-cBMG that is explained by a least
  resolved tree $(T,\sigma)$. Then all elements of $\alpha\in\mathcal{N}$
  are attached to $\rho_\alpha$, i.e., $\rho_\alpha a\in E(T)$ for all
  $a\in \alpha$.
\end{corollary}
\begin{proof}
  Assume that $\rho_\alpha a \notin E(T)$. Since by definition $\alpha
  \prec \rho_\alpha$, there exists an inner node $v$ with $\rho_\alpha v
  \in E(T)$ such that $v$ lies in the unique path from $\rho_\alpha$ to
  $a$. In particular $v\neq a$. Theorem \ref{thm:lr-2} implies that each
  inner vertex (except possibly the root) of the least resolved tree
  $(T,\sigma)$ must be the root of some $\rthin$ class of
  $(G,\sigma)$. Hence, there is a $\rthin$ class $\beta\in\mathcal{N}$ with
  $\rho_\beta=v$. According to Lemma \ref{lem:prp1}(ii), the subtree $T(v)$
  contains leaves of both colors, i.e., there exists some leaf $c\in
  L(T(v))$ with $\sigma(c)\neq \sigma(a)$. It follows that $\lca(a,c)\prec
  \rho_\alpha$, which contradicts the definition of $\rho_\alpha$.
  
\end{proof}	
This result remains true also for 2-cBMGs that are not connected.

\subsection{Characterization of 2-cBMGs}

We will first establish necessary conditions for a colored digraph to be a
2-cBMG. The key construction for this purpose is the reachable set of a
$\rthin$ class, that is, the set of all leaves that can be reached from
this class via a path of directed edges in $(G,\sigma)$. Not unexpectedly,
the reachable sets should forms a hierarchical structure. However, this
hierarchy does not quite determine a tree that explains $(G,\sigma)$. We
shall see, however, that the definition of reachable sets can be modified
in such a way that the resulting hierarchy defines the unique least
resolved tree w.r.t.\ $(G,\sigma)$.

\subsubsection{Necessary Conditions} 

We start by deriving some graph properties of 2-cBMGs. We shall see later
that these are in fact sufficient to characterize 2-cBMGs.

\begin{theorem} \label{thm:N+} Let $(G,\sigma)$ be a connected 2-cBMG.
  \marginpar{\color{blue}\scriptsize $G$ is sink free iff\\ \AX{(N4)}
    $N(\alpha)\ne\emptyset$\\ for all $\alpha$.}
  Then, for any two $\rthin$ classes $\alpha$ and $\beta$ of $G$ holds
  \begin{itemize}
    \setlength{\itemindent}{0.15in}
    \item[\AX{(N1)}] $\alpha\cap N(\beta)=\beta\cap N(\alpha)=\emptyset$
      implies \\
      $N(\alpha) \cap N(N(\beta))=N(\beta)  \cap N(N(\alpha))=\emptyset$.
    \item[\AX{(N2)}] $N(N(N(\alpha))) \subseteq N(\alpha)$
    \item[\AX{(N3)}] 
      $\alpha\cap N(N(\beta))=\beta\cap N(N(\alpha))=\emptyset$ and 
      $N(\alpha)\cap N(\beta)\neq \emptyset$ implies 
      $N^-(\alpha)=N^-(\beta)$ and 
      $N(\alpha)\subseteq N(\beta)$ or 
      $N(\beta)\subseteq N(\alpha)$. 
    \end{itemize}
  \end{theorem}
\begin{proof}
  \AX{(N1)} 
  For $\sigma(\alpha)=\sigma(\beta)$ this is trivial, thus suppose 
  $\sigma(\alpha)\neq\sigma(\beta)$. By Lemma \ref{lem:prp1}(vi), 
  $\alpha$ is not contained in the subtree $T(\rho_\beta)$ and $\beta$ 
  is not contained in the subtree $T(\rho_\alpha)$. Therefore, 
  $\rho_\alpha$ and $\rho_\beta$ must be incomparable. 
  Since $N(\alpha), N(N(\alpha))\preceq \rho_\alpha$ and 
  $N(\beta), N(N(\beta))\preceq \rho_\beta$ by 
  Lemma \ref{lem:prp1}(iii) and (vii), we conclude that 
  $N(\alpha) \cap N(N(\beta))=
   N(\beta) \cap N(N(\alpha))=\emptyset$.

  \smallskip
  \par\noindent\AX{(N2)} For contradiction, assume that there is $q\in
  N(N(N(\alpha)))\setminus N(\alpha)$. Since
  $\sigma(q)=\sigma(u)\ne\sigma(x)$ for all $x\in \alpha$ and $u\in
  N(\alpha)$, any such $q$ must satisfy $\lca(x,q)\succ \lca(x,u)$ for
  all $x\in\alpha$ and $u\in N(\alpha)$. Otherwise it would be contained
  in $N(\alpha)$.  Since $N(x)\preceq \rho_{\alpha}$ by Lemma
  \ref{lem:prp1}(iii), the definition of $\rho_{\alpha}$ implies that there
  is some pair $x\in\alpha$ and $y\in\beta\subseteq N(\alpha)$ with
  $\lca(x,y)=\rho_{\alpha}$. Therefore $\lca(x,q)\succ \rho_{\alpha}$.

  Now consider $\beta\subseteq N(\alpha)$. Since
  $\sigma(\beta)\ne\sigma(\alpha)$ and $\lca(\alpha,\beta)\preceq
  \rho_{\alpha}$, we infer that $N(N(\alpha)) \preceq\rho_{\alpha}$.
  Repeating the argument yields $N(N(N(\alpha)))\preceq\rho_{\alpha}$
  and thus there cannot be a pair of leaves $x\in\alpha$ and $q\in
  N(N(N(\alpha)))$ with $\lca(x,q)\succ \rho_{\alpha}$.

  \smallskip
  \par\noindent\AX{(N3)} We first note that \AX{(N3)} is trivially
    true for $\alpha=\beta$. Hence, assume $\alpha\ne\beta$ and 
  suppose $N(\alpha)\cap N(\beta)\neq\emptyset$. 
  Since $T$ is a tree, Lemma \ref{lem:prp1}(vi) implies that either 
  $N(\alpha)\subseteq N(\beta)$ or $N(\beta)\subseteq N(\alpha)$. 
  Assume $N(\beta)\subseteq N(\alpha)$. Hence, 
  $\rho_\beta\preceq\rho_\alpha$. Consequently, for any 
  $\gamma\subseteq N^-(\alpha)$ holds 
  $\lca(\gamma,\beta)\preceq\lca(\gamma,\alpha)\preceq \lca(\gamma,x)$ 
  for all $x$ with $\sigma(x)=\sigma(\alpha)$ and therefore, 
  $N^-(\alpha)\subseteq N^-(\beta)$. 
  Assume for contradiction that there is a 
  $\gamma'\subseteq N^-(\beta)\setminus N^-(\alpha)$. 
  By definition, we have 
  $\rho_\alpha\succeq\lca(\gamma',\beta)\succeq\rho_\beta$ in this case. 
  But then, Lemma \ref{lem:prp1}(vi) implies 
  $N(\gamma')\subseteq N(\alpha)$ and 
  $\beta\subseteq N(\gamma')\subseteq N(N(\alpha))$; a contradiction. 
   
\end{proof}

\begin{definition}
  For any digraph $(G,\sigma)$ we define the \emph{reachable set}
  $R(\alpha)$ for a $\rthin$ class $\alpha$ by
\begin{equation}
  R(\alpha) = N(\alpha) \cup N(N(\alpha)) 
                \cup N(N(N(\alpha))) \cup ...  
\end{equation}
Moreover, we write 
$\mathcal{W}:=\{\alpha\in\mathcal{N} \mid N^-(\alpha)=\emptyset\}$ for the 
set of $\rthin$ classes without in-neighbors.
\end{definition}

As we shall see below, technical difficulties arise for distinct $\rthin$
classes that share the same set of in-neighbors. Hence, we briefly consider
the classes in $\mathcal{W}$. An example is shown Fig.\
\ref{fig:weirdTree}.

\begin{figure}[t]
  \begin{center}
    \begin{tabular}{lcr}
      \begin{minipage}{0.37\textwidth}
        \includegraphics[width=\textwidth]{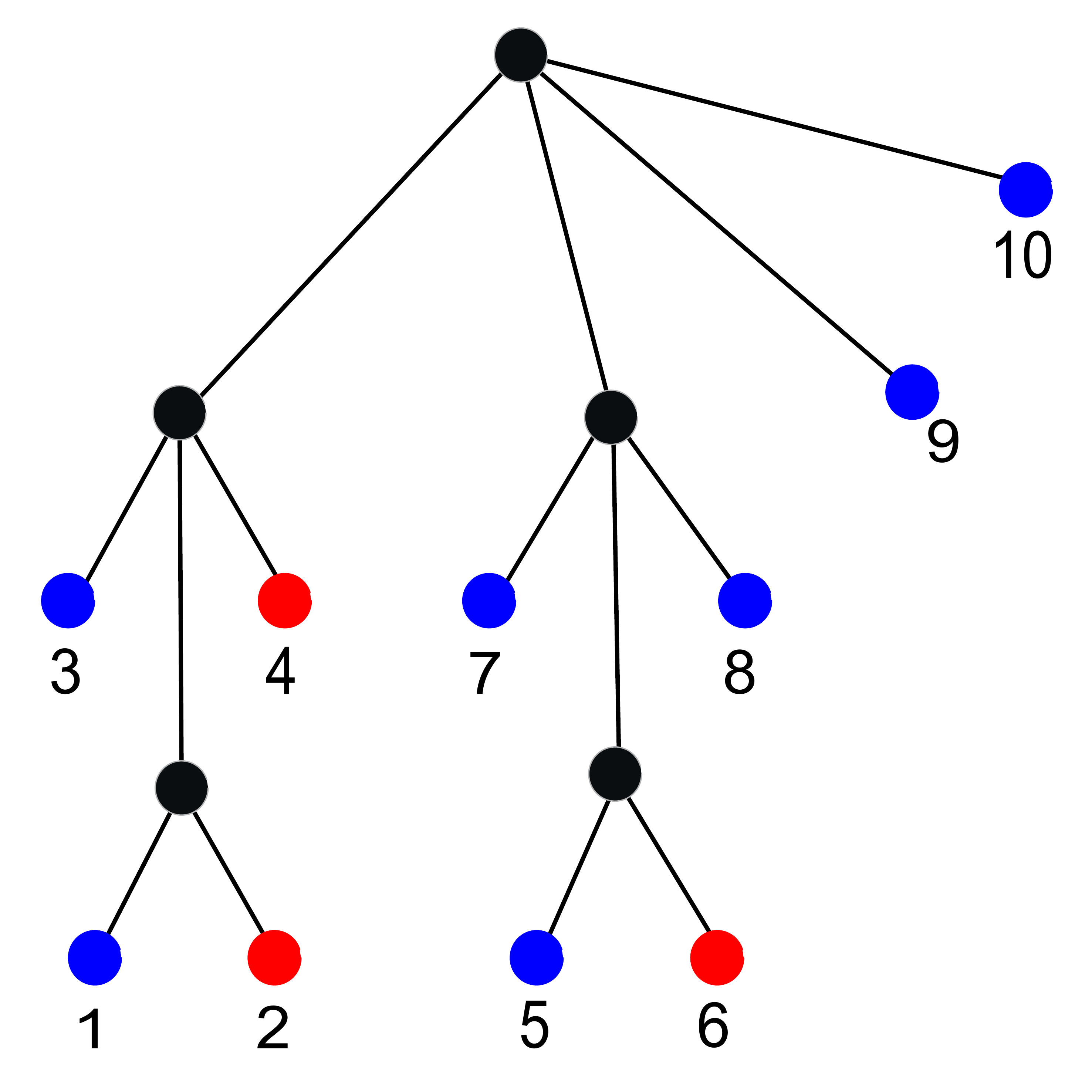}
      \end{minipage} 
      & & 
      \begin{minipage}{0.57\textwidth}
        \includegraphics[width=\textwidth]{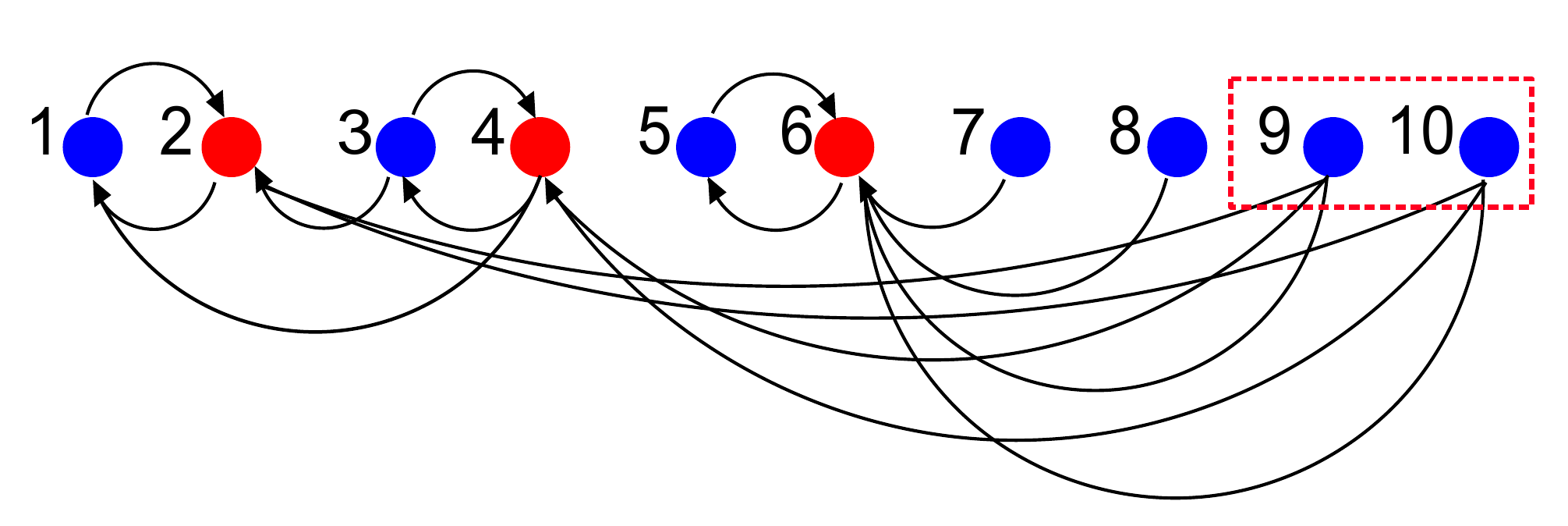} 
      \end{minipage}
      \end{tabular}
  \end{center}
  \caption[]{A 2-cBMG with $|\mathcal{W}|>1$ and its least resolved
    tree. The $\rthin$ class $\alpha=\{9,10\}$ consists of children of the
    root without in-neighbors. There is a second $\rthin$ class without
    in-neighbors, namely $\beta=\{7,8\}$. Hence
    $\mathcal{W}=\{\alpha,\beta\}$,
    $R(\alpha)=\{1,\dots,6\}=L\setminus(\alpha\cup\beta)$, while
    $R(\beta)=\{5,6\}$.}
\label{fig:weirdTree}
\end{figure}

\begin{lemma}
  \label{lem:weird}
  Let $G(T,\sigma)$ be a connected 2-cBMG explained by a tree $(T,\sigma)$.
  Then all $\rthin$ classes in $\mathcal{W}$ have the same color and the
  cardinality of $\mathcal{W}$ distinguishes three types of roots as
  follows:
  \begin{description} 
  \item[(i)] $\mathcal{W}=\emptyset$ if and only if
    $\rho_T=\rho_{\alpha}=\rho_{\beta}$ for two distinct $\rthin$ classes
    $\alpha$ and $\beta$.
  \item[(ii)] $|\mathcal{W}|>1$ if and only if there is a unique $\rthin$ class
    $\alpha^*\in \mathcal{W}$ that is characterized by 
    $\displaystyle R(\alpha^*) = 
      L\setminus\bigcup_{\beta\in\mathcal{W}} \beta$. 
    Furthermore,  $\rho_{\alpha^*}=\rho_T$.
  \item[(iii)] If $\mathcal{W}=\{\alpha\}$, then $\rho_{\alpha}=\rho_T$ and
    $R(\alpha)=L\setminus \alpha$.
  \end{description}
\end{lemma}
\begin{proof}
  By Thm.\ \ref{thm:connected} there is at least one child $v$ of the root
  $\rho_T$ of $T$ that itself is the root of a subtree with a single leaf
  color, i.e., $\sigma(L(T(v)))=\{s\}$. Assume for contradiction that there
  are two $\rthin$ classes $\alpha,\beta\in\mathcal{W}$ with
  $s=\sigma(\alpha)\ne\sigma(\beta)=t$. Then by definition
  $\lca(v,x)=\rho_T$ for all $x\in\beta$, and furthermore, $ux \in E(G)$
  for all $u\in L(T(v))$. Since $x\in\beta$ has an in-arc,
  $\beta\not\in\mathcal{W}$, a contradiction. All leaves in $\mathcal{W}$
  therefore have the same color.
    
  For the remainder of the proof we fix such a child $v$ of the root $\rho_T$.
  By construction all leaves below it belong to the same $\rthin$ class, which
  we denote by $\omega=L(T(v))$. W.l.o.g.\ we assume $\sigma(v)=s$.  Since
  $\rho_\omega=\rho_T$ by construction, we have
  $N(\omega)=L[t]$.
    
  \smallskip\noindent(i)\quad 
  Suppose $\mathcal{W}=\emptyset$. Then there is a $\beta\in\mathcal{N}_t$
  such that $\beta\NEW{\subseteq} N^-(\omega)$. For all $b\in\beta$ we have
  $\lca(b,\omega)\leq\lca(b,x)$ for all $x\in L[s]$. Since
  $\lca(b,\omega)=\rho_T$ we conclude $\rho_\beta=\rho_T=\rho_\omega$.

  Conversely, suppose $\alpha$ and $\beta$ are two distinct $\rthin$
  classes such that $\rho_\alpha=\rho_\beta=\rho_T$. By Lemma
  \ref{lem:prp1}(v), $\sigma(\alpha)\neq\sigma(\beta)$. W.l.o.g.\ assume
  $\sigma(\alpha)=s$ and $\sigma(\beta)=t$. Since
  $L(T(\rho_\alpha))=L(T(\rho_T)=L$, Lemma \ref{lem:prp1}(vi) implies that
  $N(\alpha)= L[t]$ and $N(\beta)=L[s]$. Therefore, $\alpha\in N^-(\gamma)$
  for all $\gamma\in \mathcal{N}_t$ and $\beta\in N^-(\gamma)$ for all
  $\gamma\in \mathcal{N}_s$. Hence $\mathcal{W}=\emptyset$.

  \smallskip\noindent(ii)\quad 
  If $\mathcal{W}\ne\emptyset$, (i) implies $\rho_\beta \neq \rho_T$ for
  all $\beta\in\mathcal{N}_t$, and hence $\rho_\beta \prec \rho_T$.  Thus,
  there is no $\beta\in\mathcal{N}_t$ with $\omega\subseteq N(\beta)$,
  i.e., $N^-(\omega)=\emptyset$ and thus $\omega\in\mathcal{W}$.
    
  Consider $\gamma\in\mathcal{N}_s$. We have $N^-(\gamma)\ne\emptyset$ if
  and only if there is $\zeta\in\mathcal{N}_t$ such that
  $\gamma\NEW{\subseteq} N(\zeta)$, i.e., if and only if
  $\gamma\subseteq N(L[t])$.  Since $N(\omega)=L[t]$ we have
  $\gamma\notin\mathcal{W}$ if and only if $\gamma\subseteq
  N(N(\omega))$. In other words,
  $N(N(\omega))=L[s]\setminus \bigcup_{\beta\in\mathcal{W}}\beta$.  Using
  \AX{(N2)} we have
  $$R(\omega)= N(\omega)\cup N(N(\omega)) = L[t]\cup
  \bigcup\{\gamma\in\mathcal{N}_s|N^-(\gamma)\ne\emptyset\} = L\setminus
  \bigcup_{\gamma\in\mathcal{W}} \gamma\,.$$
  Now suppose there is another $\alpha\in\mathcal{W}$ with
  $R(\alpha)=L\setminus \bigcup_{\gamma\in\mathcal{W}} \gamma$. We already
  know that $\sigma(\alpha)=s$ since all classes in $\mathcal{W}$ must have
  the same color. Hence $L[t]\subseteq R(\alpha)$. Consequently, $\zeta\in
  N(\omega)$ if and only if $\zeta\in N(\alpha)$ and thus
  $N(\alpha)=N(\omega)$. Since $\alpha,\omega\in\mathcal{W}$ implies
  $N^-(\alpha)=N^-(\omega)=\emptyset$, $\alpha$ and $\omega$ share both in-
  and out-neighbors, and thus $\alpha=\omega$. Therefore $\omega$ is
  unique.
	  
  \smallskip\noindent(iii)\quad From the proof of (ii), we know that if
  $|\mathcal{W}|=1$, then the unique member of
  $\mathcal{W}$ is $\omega$. We already know that $\rho_\omega=\rho_T$.
   
\end{proof}

\subsubsection{Sufficient Conditions}

We now turn to showing that the properties obtained in Theorem \ref{thm:N+}
are already sufficient for the characterization of 2-cBMGs. For this we show
that the extended reachable sets form a hierarchy whenever $(G,\sigma)$
satisfies the properties \AX{(N1)}, \AX{(N2)}, and \AX{(N3)}. 

Recall that a set system $\mathscr{H}\subseteq 2^L$ is a \emph{hierarchy}
on $L$ if (i) for all $A,B\in\mathscr{H}$ holds $A\subseteq B$,
$B\subseteq A$, or $A\cap B=\emptyset$ and (ii) $L\in \mathscr{H}$.

The following simple property we will be used throughout this section:
\begin{lemma} \label{lem:N++} If $G$ is a connected two-colored digraph
  satisfying \AX{(N1)}, then for any two $\rthin$ classes $\alpha$ and
  $\beta$ holds
  \begin{equation}
    N(\alpha) \cap N(\beta)=\emptyset 
    \quad\textrm{implies}\quad 
    N(N(\alpha)) \cap N(N(\beta))=\emptyset
  \end{equation}
  If $G$ satisfies \AX{(N2)}, then $R(\alpha) = N(\alpha) \cup N(N(\alpha))$. 
\end{lemma} 
\begin{proof} 
  For any $\gamma \subseteq N(\alpha)$ and any $\gamma' \subseteq
  N(\beta)$, \AX{(N1)} implies $N(\gamma)\cap N(N(\beta))=N(\gamma')\cap
  N(N(\alpha))= \emptyset$.  Recall that \AX{(N0)} holds by definition of
  $\rthin$ classes. Hence, $N(\alpha)$ is the disjoint union of $\rthin$
  classes, i.e., $N(\alpha) = \bigcup_{\gamma \subseteq N(\alpha)} \gamma$.
  Thus, $N(N(\alpha)) \cap N(N(\beta))=(\bigcup_{\gamma \subseteq
    N(\alpha)} N(\gamma)) \cap N(N(\beta)) =\emptyset$. The equation
  $R(\alpha) = N(\alpha) \cup N(N(\alpha))$ is an immediate consequence of
  \AX{(N2)}.   
\end{proof} 

\begin{lemma} 
  \label{lem:hierarchy} 
  Let $(G,\sigma)$ be a connected two-colored digraph satisfying properties
  \AX{(N1)}, \AX{(N2)}, and \AX{(N3)}. 
  \marginpar{\color{blue}\scriptsize The lemma also requires \AX{(N4)}.}
  Then, $\mathscr{H}:=\{R(\alpha)\mid
  \alpha\in\mathcal{N}\}$ is a hierarchy on $L\setminus
    \bigcup_{\alpha\in\mathcal{W}}\alpha$.
  \end{lemma}
  
\begin{proof}
  First we note that $R(\alpha) = N(\alpha) \cup N(N(\alpha))$ by property
  \AX{(N2)}. Furthermore, using \AX{(N0)}, we observe that $\beta\cap
  N(\alpha)\neq\emptyset$ implies $\beta\subseteq N(\alpha)$ for all
  $\rthin$ classes $\alpha$ and $\beta$. In particular, therefore,
  $N(\alpha)$ is a disjoint union of $\rthin$ classes, and thus
  $N(N(\alpha))=\bigcup_{\beta\subseteq N(\alpha)} N(\beta)$ is again a
  disjoint union of $\rthin$ classes. Hence, for any $\rthin$ class
  $\beta\neq\alpha$, we have either $\beta\subseteq R(\alpha)$ or
  $\beta\cap R(\alpha)=\emptyset$.  Note that the case $\alpha=\beta$ is
  trivial.

  Suppose first $\beta\subseteq R(\alpha)$. If $\beta\subseteq
  N(\alpha)$, then $R(\beta) = N(\beta) \cup N(N(\beta)) \subseteq 
  N(N(\alpha)) \cup N(N(N(\alpha))) \subseteq
  N(N(\alpha))\cup N(\alpha)$.  On the other hand, $\beta\subseteq
  N(N(\alpha))$ yields $R(\beta) \subseteq N(N(N(\alpha))) \cup
  N(N(N(N(\alpha))) \subseteq N(\alpha)\cup N(N(\alpha))$.
  Thus, $R(\beta)\subseteq R(\alpha)$.

  Exchanging the roles of $\alpha$ and $\beta$, the same argument shows
  that $\alpha\subseteq R(\beta)$ implies $R(\alpha)\subseteq R(\beta)$.
	
  Now suppose that neither $\alpha\subseteq R(\beta)$ nor $\beta\subseteq
  R(\alpha)$ and thus, by the arguments above, that $\alpha\cap
  R(\beta)=\beta\cap R(\alpha)=\emptyset$.  In particular, therefore,
  $\alpha\cap N(\beta)=\beta\cap N(\alpha)=\emptyset$ and thus property
  \AX{(N1)} implies $R(\alpha)\cap R(\beta)=(N(\alpha)\cap
  N(\beta))\cup(N(N(\alpha))\cap N(N(\beta)))$. If $N(\alpha)\cap
  N(\beta)=\emptyset$, then $R(\alpha)\cap R(\beta)=\emptyset$ by Lemma
  \ref{lem:N++}. If $N(\alpha)\cap N(\beta)\ne\emptyset$, then property
  \AX{(N3)} and $\alpha\cap R(\beta)=\beta\cap R(\alpha)=\emptyset$ implies
  either $N(\alpha)\subseteq N(\beta)$ or $N(\beta)\subseteq N(\alpha)$.
  Isotony of $N$ thus implies $N(N(\alpha))\subseteq N(N(\beta))$ or
  $N(N(\beta))\subseteq N(N(\alpha))$, respectively. Hence we have either
  $R(\alpha)\subseteq R(\beta)$ or $R(\beta)\subseteq R(\alpha)$. Therefore
  $\mathscr{H}$ is a hierarchy.

  Finally, we proceed to show that there is a unique set $R(\alpha^*)$ that
  is maximal w.r.t.\ inclusion and in particular, satisfies
  $R(\alpha^*)=L\setminus \bigcup_{\alpha\in\mathcal{W}} \alpha$.

  Assume, for contradiction, that there are two distinct elements
  $R(\alpha), R(\alpha^*)\in \mathscr{H}$ that are both maximal w.r.t.\
  inclusion.  Thus, $R(\alpha)\cap R(\alpha^*)=\emptyset$ and
  $\alpha\neq\alpha^*$. Moreover, since $\mathscr{H}$ is a hierarchy, for
  each $\beta\in \mathcal{N}$ with $R(\beta)\subseteq R(\alpha)$, we must
  have $R(\beta)\cap R(\alpha^*)=\emptyset$.
  \marginpar{\color{blue}\scriptsize $\beta\subseteq R(\alpha)$ requires
  $R(\beta)\ne\emptyset$ and thus \AX{(N4)}.}
  In particular, this implies
  $\beta\subseteq R(\alpha)$ for any $\beta\in \mathcal{N}$ with $R(\beta)
  \subseteq R(\alpha)$. As a consequence there is no $\beta\subseteq
  R(\alpha)$ and $\beta'\subseteq R(\alpha^*)$ such that $\beta\subseteq
  N(\alpha^*)$ and $\beta'\subseteq N(\alpha)$, respectively. Therefore,
  $R(\alpha)$ and $R(\alpha^*)$ are not connected; a contraction to the
  connectedness of $G$. Hence, $R(\alpha) = R(\alpha^*)$, i.e., the there
  is a unique set $R(\alpha^*)$ in $\mathscr{H}$ that is maximal w.r.t.\
  inclusion. It contains all $\rthin$ classes of $G$ that have non-empty
  in-neighborhood. Since by definition, all vertices of $G$ are assigned to
  exactly one $\rthin$ class, we conclude that $R(\alpha^*)=L\setminus
  \bigcup_{\alpha\in\mathcal{W}} \alpha$.   
\end{proof}

\begin{figure}[t]
\begin{center}
  \includegraphics[width=0.8\textwidth]{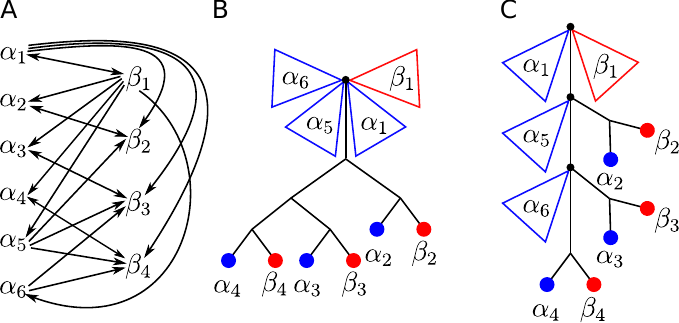}
\end{center}
\caption[]{(A) The two-colored digraph $(G,\sigma)$ satisfies \AX{(N1)},
  \AX{(N2)} and \AX{(N3)}. All $\alpha_i$ are $\rthin$ classes of
  $(G,\sigma)$ and belong to color ``blue'', the $\rthin$ classes $\beta_j$
  form the ``red'' color classes. Red (blue) triangles indicate subtrees
  that only contain red (blue) leaves. Note that
  $N^-(\alpha_1)=N^-(\alpha_5)=N^-(\alpha_6)$.  (B) The tree obtained from
  the hierarchy $\mathscr{H}=\{R(\alpha)\mid \alpha\in\mathcal{N}\}$ by
  attaching to the corresponding tree the elements of $\alpha$ as leaves to
  $R(\alpha)$ does not explain $(G,\sigma)$. It would imply
  $N^-(\alpha_1)=N^-(\alpha_5)=N^-(\alpha_6)$ and
  $N(\alpha_1)=N(\alpha_5)=N(\alpha_6)$, i.e.,
  $\alpha_1\rthin \alpha_5\rthin\alpha_6$. (C) The tree defined by the
  hierarchy $\mathscr{H'}=\{R'(\alpha)\mid \alpha\in\mathcal{N}\}$ with
  elements of $\alpha$ attached as leaves to $R'(\alpha)$ is the unique
  least resolved tree that explains $G$ (cf.\ Lemma
  \ref{lem:hierarchy-R'}).}
\label{fig:RvsR'}
\end{figure}

\NEW{Note that while $R(\alpha)$ is unique for a given $\rthin$ class
  $\alpha$, there may exist more than one $\rthin$ class that have the same
  reachable set} (see for instance $\alpha_2$ and $\beta_2$ in Fig.\
\ref{fig:RvsR'}(C)). In particular, there may even be $\rthin$ classes with
different color giving rise to the same element of $\mathscr{H}$. More
generally, we have $R(\alpha)=R(\beta)$ for $\alpha\ne\beta$ if and only if
$\alpha\in R(\beta)$ and $\beta\in R(\alpha)$.

A hierarchy $\mathscr{H}$ corresponds to a unique tree $T(\mathscr{H})$
defined as the Hasse diagram of $\mathscr{H}$, i.e., the vertices of
$T(\mathscr{H})$ are sets of $\mathscr{H}$, and $R_2$ is a child of $R_1$
iff $R_2\subset R_1$ and there is no $R_3$ such that
$R_2\subset R_3\subset R_1$. In particular, thus, two $\rthin$ classes
belong to the same interior vertex if $R(\alpha)=R(\beta)$. It is tempting
to use this tree to construct a tree $T$ explaining $(G,\sigma)$ by
attaching the elements of $\alpha$ as leaves to the node $R(\alpha)$ in
$T(\mathscr{H})$. The example in Fig.\ \ref{fig:RvsR'}(A) and (B) shows,
however, that this simply does not work.  The key issue arises from groups
of distinct $\rthin$ classes that share the same in-neighborhood because
they will in general be attached to the same node in $T(\mathscr{H})$,
i.e., they are indistinguishable. We therefore need a modification of the
definition of reachable sets that properly distinguishes such $\rthin$
classes in order to construct a hierarchy with the appropriate resolution
for the least resolved tree specified in Theorem \ref{thm:lr-2}. To this
end we define for every $\rthin$ class the auxiliary leaf set
\begin{equation} 
  \label{eq:def-Q}
    Q(\alpha) = \{x\in L \,\mid\, \exists \beta \in \mathcal{N}:\, 
      x\in\beta,\, N^-(\beta)=N^-(\alpha) \text{ and }
      N(\beta)\subseteq N(\alpha)\} 
\end{equation}
Note that $\alpha\subseteq Q(\alpha)$. For later reference we list several 
simple properties of $Q$. 
\begin{lemma}\label{lem:Q}
  \begin{description}
  \item[(i)]   $\beta\subseteq Q(\alpha)$ implies  
    $\sigma(\beta)=\sigma(\alpha)$.
  \item[(ii)]  $\beta\subseteq Q(\alpha)$ implies  
    $Q(\beta)\subseteq Q(\alpha)$.
  \item[(iii)] $\beta\subseteq Q(\alpha)$ implies  
    $R(\beta)\subseteq R(\alpha)$.
  \item[(iv)]  $\alpha\cap N(\beta)=\emptyset$ implies 
    $Q(\alpha)\cap\ N(\beta)=\emptyset$.
  \item[(v)]   $\alpha\cap N(N(\beta))=\emptyset$ implies 
    $Q(\alpha)\cap\ N(N(\beta))=\emptyset$.
  \end{description}
\end{lemma}
\begin{proof}
  \par\noindent(i) 
  follows directly from the definition.
  \par\noindent (ii) 
  Let $\beta\subseteq Q(\alpha)$, $\gamma\in\mathcal{N}$ and 
  $\gamma\subseteq Q(\beta)$. Then,
  $N^-(\gamma)=N^-(\beta)=N^-(\alpha)$ and $N(\gamma)\subseteq
  N(\beta)\subseteq N(\alpha)$, hence $\gamma\subseteq Q(\alpha)$ and 
  therefore $Q(\beta)\subseteq Q(\alpha)$.
  \par\noindent (iii) 
  By definition, $N(\beta)\subseteq N(\alpha)$. Monotonicity of $N$
  implies $N(N(\beta))\subseteq N(N(\alpha)$) and therefore,
  $R(\beta)\subseteq R(\alpha)$.
  \par\noindent (iv) 
  Assume that $\alpha\cap N(\beta)=\emptyset$, but 
  $\gamma\subseteq Q(\alpha)\cap\  N(\beta)\neq\emptyset$. 
  Thus, $\beta\subseteq N^-(\gamma)=N^-(\alpha)$,
  i.e., $\alpha\subseteq N(\beta)$; a contradiction.
  \par\noindent (v) 
  Assume that $\alpha\cap N(N(\beta))=\emptyset$, but 
  $\gamma\subseteq Q(\alpha)\cap\ N(N(\beta))\neq \emptyset$.
  Thus, there is a $\rthin$ class $\xi\subseteq N(\beta)$ such that
  $\xi\subseteq N^-(\gamma)=N^-(\alpha)$ and therefore, $\alpha\subseteq
  N(N(\beta))$; a contradiction.
   
\end{proof}

Finally we define, for any two-colored digraph $(G,\sigma)$, its
\emph{extended reachable set} as 
\begin{equation}
  \label{eq:def-R'}
  R'(\alpha):= R(\alpha)\cup Q(\alpha).
\end{equation}
Note that $\alpha\in R'(\alpha)$. Furthermore, the extended reachable set
$R'(\alpha)$ contains vertices with both colors for every $\rthin$ class
$\alpha$. Thus $|R'(\alpha)|>1$.
\marginpar{\color{blue}\scriptsize $|R'(\alpha)|>1$\\ requires \AX{(N4)}.}
We show next that for any 2-cBMG the
extended reachable sets form the hierarchy that yields the desired
least resolved tree.

\begin{lemma} 
  \label{lem:hierarchy-R'} 
  Let $(G,\sigma)$ be a connected two-colored digraph satisfying
  properties \AX{(N1)}, \AX{(N2)}, and \AX{(N3)}.
\marginpar{\color{blue}\scriptsize The lemma also requires \AX{(N4)}.}
  Then,
  $\mathscr{H'}:=\{R'(\alpha)\mid \alpha\in\mathcal{N}\}$ is a hierarchy on
  $L$.
\end{lemma}
\begin{proof}
  Consider two distinct $\rthin$ classes $\alpha,\beta\in\mathcal{N}$.  By
  definition $Q(\alpha)$ is the disjoint union of $\rthin$ classes.  The
  same is true for $R(\alpha)$ as argued in the proof of \NEW{Lemma}
  \ref{lem:hierarchy}, hence $R'(\alpha)=R(\alpha)\cup Q(\alpha)$ is also
  the disjoint union of $\rthin$ classes. Thus we have either
  $\beta\subseteq R'(\alpha)$ or $\beta\cap R'(\alpha)=\emptyset$.

  First assume $\beta\subseteq R'(\alpha)$. Thus we have
  $\beta\subseteq R(\alpha)$ or $\beta\subseteq Q(\alpha)$.
  If $\beta\subseteq Q(\alpha)$, i.e., $N(\beta)\subseteq
  N(\alpha)$ and consequently $R(\beta)\subseteq R(\NEW{\alpha})$, then Lemma
  \ref{lem:Q}(ii)+(iii) implies that $R'(\beta)\subseteq R'(\alpha)$.  If
  $\beta\subseteq R(\alpha)$ then $R(\beta)\subseteq R(\alpha)\subseteq
  R'(\alpha)$, shown as in the proof of Lemma \ref{lem:hierarchy}. It
  remains to show that $Q(\beta)\subseteq R'(\alpha)$. By definition, we
  have $N^-(\gamma)=N^-(\beta)$ for any $\gamma\subseteq Q(\beta)$.
  Therefore, $\beta\subseteq N(\alpha) \cup N(N(\alpha))$ implies
  $\gamma\subseteq N(\alpha)\cup N(N(\alpha))$. Hence, $\gamma\subseteq
  R(\alpha)\subseteq R'(\alpha)$.  In summary, for all $\beta\subseteq
  R'(\alpha)$ we have $R'(\beta)\subseteq R'(\alpha)$.

  The implication ``$\alpha\subseteq R'(\beta) \implies 
  R'(\alpha)\subseteq R'(\beta)$'' follows by exchanging $\alpha$ and $\beta$
  in the previous paragraph.
  
  Now suppose $\beta\cap R'(\alpha)=\alpha\cap R'(\beta)=\emptyset$. In
  particular, it then holds $\alpha\cap N(\beta)=\beta\cap
  N(\alpha)=\emptyset$ and $\alpha\cap N(N(\beta))=\beta\cap
  N(N(\alpha))=\emptyset$. Applying property \AX{(N1)} and Lemma
  \ref{lem:Q}(iv)+(v) yields $R'(\alpha)\cap R'(\beta)=
  \big(N(\alpha)\cap N(\beta)\big)\cup \big(N(N(\alpha))\cap
  N(N(\beta))\big) \cup \big(Q(\alpha)\cap Q(\beta)\big)$. First, let
  $N(\alpha)\cap N(\beta)=\emptyset$. This immediately implies
  $Q(\alpha)\cap Q(\beta)=\emptyset$ and from Lemma \ref{lem:N++} follows
  $N(N(\alpha))\cap N(N(\beta))=\emptyset$. Hence, $R'(\alpha)\cap
  R'(\beta)=\emptyset$. Now assume $N(\alpha)\cap
  N(\beta)\neq\emptyset$. By property \AX{(N3)} we conclude
  $N^-(\alpha)=N^-(\beta)$ and either $N(\alpha)\subseteq N(\beta)$ or
  $N(\beta)\subseteq N(\alpha)$. Consequently, either
  $N(N(\alpha))\subseteq N(N(\beta))$ and $Q(\alpha)\subseteq
  Q(\beta)$, or $N(N(\beta))\subseteq N(N(\alpha))$ and
  $Q(\beta)\subseteq Q(\alpha)$. Hence, it must either hold
  $R'(\alpha)\subseteq R'(\beta)$ or $R'(\beta)\subseteq R'(\alpha)$.

  It remains to show that $L\in \mathscr{H'}$. Similar arguments as in
  the proof of Lemma \ref{lem:hierarchy} can be applied in order to show
  that there is a unique element $R'(\alpha^*)$ that is maximal w.r.t.\
  inclusion in $\mathscr{H'}$. Since for any $\alpha\in \mathcal{N}$ it is
  true that $\alpha\in R'(\alpha)$, every $\rthin$ class of $G$ is
  contained in at least one element of $\mathscr{H'}$. Moreover, any vertex
  of $G$ is contained in exactly one $\rthin$ class. Hence,
  $L=R'(\alpha^*)\in\mathscr{H'}$.  
\end{proof}

Since $\mathscr{H'}$ is a hierarchy, its Hasse diagram is a tree
$T(\mathscr{H'})$. Its vertices are by construction exactly the extended
reachable sets $R'(\alpha)$ of $(G,\sigma)$. Starting from
$T(\mathscr{H'})$, we construct the tree $T^*(\mathscr{H'})$ by attaching
the vertices $x\NEW{\in \alpha}$ to the vertex $R'(\alpha)$ of
$T(\mathscr{H'})$. The tree $T^*(\mathscr{H'})$ has leaf set $L$. Since
$|R'(\alpha)|>1$ as noted below Equ.(\ref{eq:def-R'}), $T^*(\mathscr{H'})$
is a phylogenetic tree.

\begin{theorem}
  Let $(G,\sigma)$ be a connected 2-colored digraph. Then there exists a
  tree $T$ explaining $(G,\sigma)$ if and only if $G$ satisfies properties
  \AX{(N1)}, \AX{(N2)}, and \AX{(N3)}.
\marginpar{\color{blue}\scriptsize The theorem also requires \AX{(N4)}.}
  The tree $T^*(\mathscr{H'})$ is the
  unique least resolved tree that explains $(G,\sigma)$.
  \label{thm:char2}
\end{theorem}
\begin{proof} 
  The ``only if''-direction is an immediate consequence of Lemma
  \ref{lem:rthin-prop} and Theorem \ref{thm:N+}. For the ``if''-direction
  we employ Lemma \ref{lem:hierarchy-R'} and show that the tree
  $T^*(\mathscr{H'})$ constructed from the hierarchy $\mathscr{H'}$
  explains $(G,\sigma)$.
  
  Let $x\in L$ and $\alpha$ be the $\rthin$ class of $(G,\sigma)$ to which
  $x$ \NEW{belongs}. Denote by $\tilde N(x)$ the out-neighbors of $x$ in
  the graph explained by $T^*(\mathscr{H'})$. Therefore $y\in \tilde N(x)$
  if and only if $\sigma(y)\ne\sigma(x)$ and
  $\lca_{T^*(\mathscr{H}')}(x,y)$ is the interior node to which $x$ is
  attached in $T(\mathscr{H'})$, i.e., $R'(\alpha)$. Therefore,
  $y\in \tilde N(x)$ if and only if $\sigma(y)\ne\sigma(x)$ and
  $y\in R'(\alpha)$. By \AX{(N2)} this \NEW{is} the case if and only if
  $y\in N(x)$. Thus $\tilde N(x)=N(x)$. Since two digraphs are identical
  whenever all their out-neighborhoods are the same, the tree
  $T^*(\mathscr{H'})$ indeed explains $(G,\sigma)$.

  By construction and Theorem~\ref{thm:lr-2}, $(T^*(\mathscr{H'}),\sigma)$
  is a least resolved tree.   
\end{proof}

\subsection{Informative Triples}

An inspection of induced three-vertex subgraphs of a 2-cBMG $(G,\sigma)$
shows that several local configurations derive only from specific types of
trees.  More precisely, certain induced subgraphs on three vertices are
associated with uniquely defined triples displayed by the least resolved
tree $(T,\sigma)$ introduced in the previous section. Other induced
subgraphs on three vertices, however, may derive from two or three distinct
triples. The importance of triples derives from the fact that a
phylogenetic tree can be reconstructed from the triples that it displays by
a polynomial time algorithm traditionally referred to as \texttt{BUILD}
\cite{Semple:03}.

\NEW{\texttt{BUILD} makes use of a simple graph representation of certain
  subsets of triples: Given a triple set $R$ and a subset of leaves
  $L'\subseteq L$, the \emph{Aho-graph} $[R,L']$ has vertex set $L'$ and
  there is an edge between two vertices $x,y\in L'$ if and only if there
  exists a triple $xy|z\in R$ with $z\in L'$ \cite{Aho:81}. It is well
  known that $R$ is consistent if and only if $[R,L']$ is disconnected for
  every subset $L'\subseteq L$ with $|L'|>1$ \cite{Bryant:95}.
  \texttt{BUILD} uses Aho-graphs in a top-down recursion: First, $[R,L]$
  is computed and a tree $T$ consisting only of the root $\rho_T$ is
  initialized. If $[R,L]$ is connected and $|L|>1$, then \texttt{BUILD}
  terminates and returns \emph{``$R$ is not consistent''}.  Otherwise,
  \texttt{BUILD} adds the connected components $C_1,\dots,C_k$ of $[R,L]$
  as vertices to $T$ and inserts the edges $(\rho_T,C_i)$, $1\leq i\leq k$.
  \texttt{BUILD} recurses on the Aho-graphs $[R,C_i]$ (where vertex $C_i$
  in $T$ plays the role of $\rho_T$) until it arrives at single-vertex
  components. \texttt{BUILD} either returns the tree $T$ or identifies the
  triple set $R$ as ``not consistent''.  Since the Aho-graphs $[R,L']$ and
  their connected components are uniquely defined in each step of
  $\texttt{BUILD}$, the tree $T$ is uniquely defined by $R$ whenever it
  exists. $T$ is known as the \emph{Aho tree} and will be denoted by
  $\Aho(R)$.}

\NEW{It is natural to ask whether the triples that can be inferred directly
  from $(G,\sigma)$ are sufficient to (a) characterize 2-cBMGs and (b) to
  completely determine the least resolved tree $(T,\sigma)$ explaining
  $(G,\sigma)$.}

\begin{definition} 
\label{def:inftriples}
  Let $(G,\sigma)$ be a two-colored digraph. We say that a triple $ab|c$
  is \emph{informative} (for $(G,\sigma)$) if the three distinct vertices
  $a,b,c\in L$ induce a colored subgraph $G[a,b,c]$ isomorphic (in the
  usual sense, i.e., with recoloring) to the graphs $X_1$, $X_2$, $X_3$, or
  $X_4$ shown in Fig.~\ref{fig:triples}. The set of informative triples is
  denoted by $\mathscr{R}(G,\sigma)$.
\end{definition} 
  
\begin{figure}[htbp]
\begin{center}
  \includegraphics[width=0.85\textwidth]{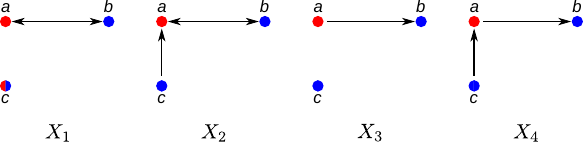}
\end{center}
\caption[]{Each of the three-vertex induced subgraphs $X_1$, $X_2$, $X_3$
  and $X_4$ gives a triple $ab|c$. If vertex $c$ i	n the drawing has two
  colors, then the color $\sigma(c)$ does not matter.  }
\label{fig:triples}
\end{figure}

\begin{lemma}
  \label{lem:consistent} 
  If $(G,\sigma)$ is a connected 2-cBMG, then each triple in
  $\mathscr{R}(G,\sigma)$ is displayed by any tree $T$ that explains
  $(G,\sigma)$.
\end{lemma}
\begin{proof}
  Let $(T,\sigma)$ be a tree that explains $(G,\sigma)$.  Assume that there
  is an induced subgraph $X_1$ in $(G,\sigma)$. W.l.o.g.\ let $\sigma(c) =
  \sigma(b)$.  Since there is no arc $(a,c)$ but an arc $(a,b)$, we have
  $\lca(a,b)\prec \lca(a,c)$, which implies that $T$ must display the
  triple $ab|c$.  By the same arguments, if $X_2$, $X_3$ or $X_4$ is an
  induced subgraph in $(G,\sigma)$, then $T$ must display the triple
  $ab|c$.   
\end{proof}

In particular, therefore, if $(G,\sigma)$ is 2-cBMG, then
$\mathscr{R}(G,\sigma)$ is consistent. It is tempting to conjecture that
consistency of the set $\mathscr{R}(G,\sigma)$ of informative triples is
already sufficient to characterize a 2-cBMG. The example in
Fig.~\ref{fig:counter-triples} shows, however, that this is not the case.

\begin{figure}[tbp]
  \begin{center}
    \includegraphics[width=0.8\textwidth]{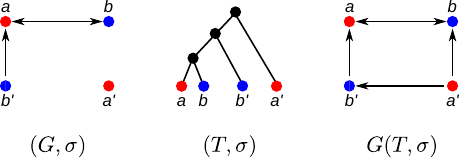}
  \end{center}
  \caption[]{The four-vertex graph $(G,\sigma)$ on the l.h.s.\  cannot be a 
    2-cBMG because there is no out-arc from $a'$. The four induced
    subgraphs are of type $X_1$, $X_2$, $X_3$ (with red and blue 
    exchanged) and arc-less, respectively resulting in the set 
    $R(G,\sigma) = \{ab|b',ab|a',ab'|a'\}$ of informative triples.
    This set is consistent and displayed by the Aho tree $T$ shown in the 
    middle. It is not difficult to check that every edge of $T$ is 
    distinguished by one informative triple. Therefore $R(G,\sigma)$ 
    identifies the leaf-colored tree $(T,\sigma)$ \cite{GSS:07}. 
    However, the graph $G(T,\sigma)$ explained by the tree $(T,\sigma)$ is 
    not isomorphic to the graph $(G,\sigma)$ from which the triples 
    were inferred.
  }
  \label{fig:counter-triples}
\end{figure}

\begin{lemma}\label{lem:edge-dist}
  Let $(T,\sigma)$ be a least resolved tree explaining a connected
  2-cBMG $(G,\sigma)$. Then every inner edge of $T$ is distinguished by at
  least one triple in $\mathscr{R}(G,\sigma)$.
\end{lemma}
\begin{proof}
  Let $(T,\sigma)$ be a least resolved tree w.r.t.\ to $(G,\sigma)$ and
  $e=uv$ be an inner edge of $T$. Since $(T,\sigma)$ is least resolved for
  $(G,\sigma)$, Thm.\ \ref{thm:lr-2} implies that the edge $e$ is relevant,
  and hence, there exists a $\alpha\in\mathcal{N}$ such that
  $v=\rho_\alpha$. By Cor. \ref{cor:rel_root}, we have $a\in \child(v)$ for
  any $a\in\alpha$. Lemma \ref{lem:prp1}(ii) implies that $T(v)$ contains a
  $\rthin$ class $\beta$ with $\sigma(\alpha)\neq \sigma(\beta)$ and
  $b\in\beta$.
	
  \textit{Case A:} Suppose that $\rho_\beta=\rho_\alpha$ and therefore,
  $ab, ba \in E(G)$. If $u$ is the root of some $\rthin$ class with
  $c\in\gamma$, then Lemma \ref{lem:prp1}(vi) implies $ca\in E(G)$,
  $cb\notin E(G)$ for $\sigma(c)=\sigma(b)$ and $cb\in E(T)$, $ca\notin
  E(T)$ for $\sigma(c)=\sigma(a)$. In all cases, we have neither $bc\in
  E(G)$ nor $ac\in E(G)$, since $ab, ba \in E(G)$.  Therefore, we always
  obtain a 3-vertex induced subgraph that is isomorphic to $X_2$ (see Fig.\
  \ref{fig:triples}) and $ab|c \in \mathscr{R}(G,\sigma)$. On the other
  hand, if there is no $\rthin$ class $\gamma$ such that $u=\rho_\gamma$,
  then $u$ is the root of $(T,\sigma)$ by Cor. \ref{cor:rel_root}. Since
  $(T,\sigma)$ is phylogenetic and $u$ is no root of any $\rthin$ class,
  there must be an inner vertex $w\in \child(u)\setminus \NEW{\{v\}}$ such that
  $w=\rho_\gamma$ for some $\gamma \in \mathcal{N}$. Since $T(\rho_\gamma)$
  contains leaves of both colors by Lemma \ref{lem:prp1}(ii), for any leaf
  $c\in L(T(\rho_\gamma))$ there is no edge between $c$ and $b$ as well as
  between $c$ and $a$. Taken together, we obtain the induced subgraph $X_1$
  and the triple $ab|c$.
	
  \textit{Case B:} Now assume $\rho_\beta\prec\rho_\alpha$ and there is no
  other $\beta'\in\mathcal{N}$ with $\sigma(\beta')=\sigma(\beta)$ and
  $\rho_\alpha=\rho_{\beta'}$. By definition of $\rho_\beta$, we have
  $\lca(b,a')\prec \lca(b,a)$ for some $a'$ with $\sigma(a)=\sigma(a')$,
  i.e., $ba\notin E(G)$. Moreover, Lemma \ref{lem:prp1}(vi) implies $b\in
  N(a)$, thus $ab\in E(G)$. Similar to Case A, first suppose that $u$ is
  the root of some $\rthin$ class of $(G,\sigma)$. Since $e$ is relevant,
  there is a $\gamma\in \mathcal{N}$ with $u=\rho_\gamma$ and
  $\sigma(\gamma)\neq \sigma(\alpha)$. Otherwise, if $\sigma(\gamma)=
  \sigma(\alpha)$ and there is no other $\gamma'\in \mathcal{N}$ with
  $u=\rho_{\gamma'}$, Lemma \ref{lem:prp1}(vi) implies
  $N(\alpha)=N(\gamma)$ and $N^-(\alpha)=N^-(\gamma)$, i.e., $\alpha$ and
  $\gamma$ belong to the same $\rthin$ class with root $u$. Hence, $v$ is
  not the root of any $\rthin$ class; a contradiction. Consequently, we
  have $\sigma(\gamma)\neq \sigma(\alpha)$, thus $ca\in E(G)$ by Lemma
  \ref{lem:prp1}(vi) but $ac\notin E(G)$. This yields  the triple
  $ab|c$ that is derived from the subgraph $X_4$. If $u$ is no root of any
  $\rthin$ class, analogous arguments as in \textit{Case A} show that there
  is an inner vertex $w\in \child(u)\setminus v$ such that the tree $T(w)$
  contains leaves of both colors. In particular, there exists a leaf $c\in
  L(T(w))$ and since $u$ is not the root of $\alpha$, $\beta$ or the
  $\rthin$ class that $c$ belongs to, there is no arc between $c$ and $a$
  or $b$ in $(G,\sigma)$.  Hence, we again obtain the triple $ab|c$ which
  in this case is derived from $X_3$.
	
  In every case we have $v=\lca(a,b)\prec \lca(a,c)=u$, i.e., the triple
  $ab|c$ distinguishes $uv$.  
   
\end{proof}

Lemma~\ref{lem:edge-dist} suggests that the leaf-colored Aho tree
$(\Aho(\mathscr{R}(G,\sigma)),\sigma)$ of the set of informative triples
$\mathscr{R}(G,\sigma)$ explains a given 2-cBMG $(G,\sigma)$. The following
result shows that this is indeed the case and sets the stage for the main
result of this section, a characterization of 2-cBMGs in terms of
informative triples.

\begin{theorem}
  Let $(G,\sigma)$ be a connected 2-cBMG. Then $(G,\sigma)$ is explained by
  the Aho tree of the set of informative triples, i.e.,
  $(G,\sigma)=G(\Aho(\mathscr{R}(G,\sigma)),\sigma)$.
\label{thm:Aho1} 
\end{theorem}
\begin{proof}
  Let $(\tilde T, \sigma)$ be the unique least resolved tree that explains
  $(G,\sigma)$. For a fixed vertex $v\in L$ we write
  $(G',\sigma')=(G\setminus \{v\},\sigma_{|L\setminus\{v\}})$. Let
  $(\tilde T', \sigma')$ be the unique least resolved tree that explains
  $(G',\sigma')$ and let
  $(T',\sigma'):=(\Aho(\mathscr{R}(G',\sigma')),\sigma')$ be the
  leaf-colored Aho tree of the informative triples of $(G',\sigma')$.
			
  First consider the case $L=\{x,y\}$. Since $(G,\sigma)$ is a connected
  2-cBMG, we have $\sigma(x)\neq \sigma(y)$ and $xy,yx \in E(G)$.  It is
  easy to see that both the least resolved tree w.r.t.\ $(G,\sigma)$ and
  $\Aho(\mathscr{R}(G,\sigma))$ correspond to the path $x-\rho_T-y$ with
  end points $x$ and $y$. Thus
  $(G,\sigma)=G(\Aho(\mathscr{R}(G,\sigma)),\sigma)$.

  Now let $|L|>2$ and assume that the statement of the proposition is
  false.  Then there is a minimal graph $(G,\sigma)$ such that
  $(G,\sigma)\neq G(T,\sigma)$, i.e., $(G',\sigma')=G(T',\sigma')$ holds
  for every choice of $v\in V(G)$. Since $(G,\sigma)$ is connected, Theorem
  \ref{thm:connected} implies that there is a $\rthin$ class $\alpha$ of
  $(G,\sigma)$ such that $\rho_\alpha=\rho_{\tilde T}$. We fix a vertex $v$
  in this class $\alpha$ and proceed to show that $(G,\sigma)=G(T,\sigma)$,
  a contradiction. Let $\sigma(\alpha)=s$ and let $(\tilde T-v,\sigma')$ be
  the tree that is obtained by removing the leaf $v$ and its incident edge
  from $(\tilde T,\sigma)$. Clearly, the out-neighborhood of every leaf of
  color $s$ is still the same in $(\tilde T-v,\sigma')$ compared to
  $(\tilde T, \sigma)$. Moreover, Lemma \ref{lem:prp1}(vi) implies that
  $N(x)$ remains unchanged in $(\tilde T-v,\sigma')$ for any
  $x\in L[t]\setminus \{v\}$ that belongs to a $\rthin$ class $\beta$ with
  $\rho_\beta \neq \rho_{\tilde T}$. If $\rho_\beta = \rho_{\tilde T}$,
  then $N(x)=L[s]$ in $(\tilde T,\sigma)$ by Lemma \ref{lem:prp1}(vi) and
  thus $N(x)=L[s]\setminus \{v\}$ in $(\tilde T-v,\sigma')$. We can
  therefore conclude that $(\tilde T-v,\sigma')$ explains the induced
  subgraph $(G',\sigma')$ of $(G,\sigma)$.

  Now, we distinguish two cases:

  \smallskip\par\noindent\textit{Case A:} 
  Let $|\child(\rho_{\tilde T})\cap L|>1$, which implies
  $|\child(\rho_{\tilde T-v})\cap L|\ge 1$. Hence, the root of
  $(\tilde T-v, \sigma')$ has at least two children and, in particular,
  $G(\tilde T-v, \sigma')$ is connected by Theorem \ref{thm:connected}.
  Since $(\tilde T,\sigma)$ is least resolved, Theorem \ref{thm:lr-2}
  implies that any inner edge of $(\tilde T-v,\sigma')$ is non-redundant,
  and hence $(\tilde T',\sigma')=(\tilde T-v,\sigma')$. Consequently, we
  can recover $(\tilde T, \sigma)$ from $(\tilde T,\sigma')$ by inserting
  the edge $\rho_{\tilde T'}v$. If $N^-(\alpha)=\emptyset$, then
  $vx\in E(G)$ but $xv\notin E(G)$ for any $x\in L[t]$. Hence, any
  informative triple that contains $v$ is induced by $X_2$ or $X_4$, and is
  thus of the form $xy|v$ with $\sigma(x)\neq\sigma(y)$. This implies
  $v\in \child(\rho_T)$.  On the other hand, if there is a
  $\beta\in \mathcal{N}$ with $\sigma(\beta)=t$ and
  $\rho_\beta=\rho_{\tilde T}$, we have $vu\in E(G)$ and $uv\in E(G)$ with
  $u\in L[t]$ if and only if $u\in\beta$ by Lemma
  \ref{lem:rthin-cases}(i). Then, there is no 3-vertex induced subgraph of
  $(G,\sigma)$ of the form $X_1$, $X_2$, $X_3$, or $X_4$ that contains both
  $u$ and $v$, and any informative triple that contains either $u$ or $v$
  is again of the form $xy|v$ and $xy|v$ respectively. As before, this
  implies $v\in \child(\rho_T)$. Hence, $(T,\sigma)$ is obtained from
  $(T',\sigma')$ by insertion of the edge $\rho_{T'}v$. Since
  $(G',\sigma')=G(T',\sigma')$, we conclude that $(T,\sigma)$ explains
  $(G,\sigma)$, and arrive to the desired contradiction.

  \smallskip\par\noindent\textit{Case B:} 
  If $|\child(\rho_{\tilde T})\cap L|=1$, then $(\tilde T-v, \sigma')$ is
  not least resolved since either (a) the root is of degree 1 or (b) there
  exists no $u\in\child(\rho_{\tilde T})\setminus \{v\}$ such that
  $\sigma(u)\neq \{s,t\}$ (see Theorem \ref{thm:connected}). In the latter
  case, the graph $(G',\sigma')$ is not connected. To convert
  $(\tilde T-v,\sigma')$ into the least resolved tree
  $(\tilde T',\sigma')$, we need to contract all edges $\rho_{\tilde T}u$
  with $u\in\child(\rho_{T'})\setminus\{v\}$.  Clearly, we can recover
  $(G,\sigma)$ from $(G',\sigma')$ by reverting the prescribed steps.
  Analogous arguments as in \textit{Case A} show that again any informative
  triple in $\mathscr{R}(G,\sigma)$ that contains $v$ is of the form $xy|v$
  with $\sigma(x)\neq\sigma(y)$. If $(G'\sigma')$ is connected, then any
  triple in $\mathscr{R}(G,\sigma)\setminus \mathscr{R}(G',\sigma')$ is of
  this form and hence as above, we conclude that $v\in \child(\rho_T)$ and
  $(G,\sigma)=G(T,\sigma)$. If $(G'\sigma')$ is not connected, then
  $\mathscr{R}(G,\sigma)\setminus \mathscr{R}(G',\sigma')$ contains also
  all triples $xy|z$ induced by $X_1$ and $X_3$ that emerged from
  connecting all components of $(G',\sigma')$ by insertion of $v$. However,
  since $\lca(x,y,z)=\rho_{\tilde T}$, we conclude that
  $v\in \child(\rho_T)$ and thus $(G,\sigma)=G(T,\sigma)$ again yields the
  desired contradiction.   
\end{proof}
\NEW{We finally arrive at the main result of this section.}
\begin{theorem}
  A connected 2-colored digraph $(G,\sigma)$ is a 2-cBMG if and only if
    $(G,\sigma)=G(\Aho(\mathscr{R}(G,\sigma)),\sigma)$.
\label{thm:2cbmg-triples}
\end{theorem} 
\begin{proof}
  If $(G,\sigma)$ is a 2-cBMG, then Theorem \ref{thm:Aho1} guarantees that
  $(G,\sigma)=G(\Aho(\mathscr{R}(G,\sigma)),\sigma)$.  If $(G,\sigma)$ is
  not a 2-cBMG, then either $\mathscr{R}(G,\sigma)$ is inconsistent or its
  Aho tree $\Aho(\mathscr{R}(G,\sigma))$ explains a different graph
  $G(T,\sigma)\ne(G,\sigma)$ because by assumption $(G,\sigma)$ cannot be
  explained by any tree.   
\end{proof}

If $(G,\sigma)$ is not connected, then the informative triples of
Definition \ref{def:inftriples} are not sufficient by themselves to infer a
tree that explains $(G,\sigma)$. However, it follows from Theorems
\ref{thm:connected} and \ref{thm:2cbmg-triples}, that the desired tree
$(T,\lambda)$ can be obtained by attaching the Aho trees of the connected
components as children of the root of $(T,\lambda)$. It can be understood
as the Aho tree of the triple set
\begin{equation}
  \mathscr{R}(G,\sigma) = \bigcup_i \mathscr{R}(G_i,\sigma_i) \cup 
  \mathscr{R}_C(G,\sigma)
\end{equation}
where the $\mathscr{R}(G_i,\sigma_i)$ are the sets of informative triples
of the connected components and $\mathscr{R}_C(G,\sigma)$ consists of all
triples of the form $xy|z$ with $x,y\in L(G_i)$ and $z\in L(G_j)$ for all
pairs $i\ne j$. The triple set $\mathscr{R}_C(G,\sigma)$ simply specifies
the connected components of $(G,\sigma)$. Note that with this augmented
definition of $\mathscr{R}$, Thm.~\ref{thm:2cbmg-triples} remains true also
for 2-cBMGs that are not connected.

\section{\textbf{\emph{n}}-colored Best Match Graphs}
\label{sect:ncBMG}

In this section we generalize the results about 2-cBMGs to an arbitrary
number of colors.  As in the two-color case, we write $x\rthin y$ if and
only if $x$ and $y$ have the same in- and out-neighbors. Moreover, for
given colors $r,s,t\in S$ we write $(G_{st},\sigma_{st}):=G[L[s]\cup L[t]]$
and $(G_{rst},\sigma_{rst}):=G[L[r]\cup L[s]\cup L[t]]$ for the respective
induced subgraphs.  Since $G$ is multipartite \NEW{and every vertex has at
  least one out-neighbor of each color except its own}, we can conclude
also for general cBMGs that $x\rthin y$ implies $\sigma(x)=\sigma(y)$.
Denote by $x\rthin_{st} y$ the thinness relation of Definition
\ref{def:rthin} on $(G_{st},\sigma_{st}):=G[L[s]\cup L[t]]$.
\begin{fact}\label{obs:rthin-n}
  If $\sigma(x)=\sigma(y)=s$, then $x\rthin y$ holds if and only if
  $x\rthin_{st} y$ for all $t\ne s$.
\end{fact}
We can therefore think of the relation $\rthin$ as the common refinement of
the relations $\rthin_{st}$ based on the induced 2-cBMGs for all colors
$s,t$.  In particular, therefore, all elements of a $\rthin$ class of an
$n$-cBMG appear as sibling leaves in the different least resolved trees,
each explaining one of the induced 2-cBMGs. Next we generalize the
notion of roots.
\begin{definition}
  Let $(G,\sigma)$ be an $n$-cBMG and suppose $\sigma(\alpha)=r\neq
  s$. Then the \emph{root $\rho_{\alpha}$ of the $\rthin$ class $\alpha$
    with respect to color $s$} is
  \begin{equation*}
    \rho_{\alpha,s}=
    \max_{\substack{x\in\alpha \\ y\in N_s(\alpha)}} \lca(x,y).
  \end{equation*}
\end{definition}

\begin{fact}\label{obs:2-colorISdisplayed}
  Consider an $n$-cBMG $(G,\sigma)$ that is explained by a tree
  $(T,\sigma)$.  By observation \ref{obs-1}, the subgraph
  $(G_{st},\sigma_{st})$ induced by any two distinct colors $s,t\in S$ is a
  2-BMG and thus explained by a corresponding least resolved tree
  $(T_{st},\sigma_{st})$. Uniqueness of this least resolved tree implies
  that the tree $(T,\sigma)$ must display $(T_{st},\sigma_{st})$. In other
  words, $(T,\sigma)$ is a refinement of $(T_{st},\sigma_{st})$.
\end{fact}

\begin{fact}\label{obs:2-color}
  Let $(G,\sigma)$ be an $n$-cBMG that is explained by a tree $(T,\sigma)$,
  and $a,b,c \in L$ leaves of three distinct colors. Then the 3-cBMG
  $(G(T_{\{a,b,c\}}),\sigma)$ is the complete graph on $\{a,b,c\}$ with
  bidirectional edges.
\end{fact}
Therefore, no further refinement can be obtained from triples of three
different colors. Thus, the two-colored triples inferred from the induced
2-cBMGs for all color pairs may already be sufficient to construct
$(T,\sigma)$. This suggests, furthermore, that every $n$-cBMG is explained
by a unique least resolved tree.  An important tool for addressing this
conjecture is the following generalization of condition (vi) of Lemma
\ref{lem:prp1}.
\begin{lemma}\label{lem:N-n}
  Let $(G,\sigma)$ be a (not necessarily connected) $n$-cBMG explained
  by $(T,\sigma)$ and let $\alpha$ be a $\rthin$ class of
  $(G,\sigma)$. Then $N_s(\alpha)=L(T(\rho_{\alpha,s}))\cap L[s]$ for all
  $s\in S\setminus \{\sigma(\alpha)\}$.
\end{lemma}
\begin{proof}
  The definition of $\rho_{\alpha,s}$ implies 
  $N_s(\alpha)\subseteq L(T(\rho_{\alpha,s}))\cap L[s]$. In particular, 
  there is a leaf $y\in N_s(\alpha)$ such that 
  $\lca(y,\alpha)=\rho_{\alpha,s}$. Now consider an arbitrary leaf 
  $x\in L(T(\rho_{\alpha,s}))\cap L[s]\setminus N_s(\alpha)$. 
  By construction we have 
  $\lca(x,\alpha)\preceq \rho_{\alpha,s}=\lca(y,\alpha)$ and therefore
  $x\in N_s(\alpha)$.
   
\end{proof}

We are now in the position to characterize the redundant edges.
\begin{lemma}\label{lem:red-n}
  Let $(G,\sigma)$ be a (not necessarily connected) $n$-cBMG explained by
  $(T,\sigma)$. Then the edge $e=uv$ is redundant in $(T,\sigma)$ if and
  only if (i) $e$ is an inner edge of $T$ and (ii) for every color
  $s\in\sigma(L(T(u)\setminus T(v)))$, there is no $\rthin$ class
  $\alpha\in\mathcal{N}$ with $v=\rho_{\alpha,s}$.
\end{lemma}
\begin{proof}
  Let $(T_e,\sigma)$ be the tree that is obtained from $(T,\sigma)$ by
  contraction of the edge $e=uv$ and assume that $(T_e,\sigma)$ explains
  $(G,\sigma)$. First we note that $e$ is an inner edge and thus, in
  particular, $L(T_e)=L(T)$. Otherwise, i.e., if $e$ is an outer edge, then
  $v\notin L(T_e)$; $(T_e,\sigma)$ does not explain $(G,\sigma)$.  Now
  consider an inner edge $e$.  Since $(T,\sigma)$ is phylogenetic, there
  exists a leaf $y\in L(T(u)\setminus T(v))$ \NEW{of some color
    $s\in\sigma(L(T(u)\setminus T(v)))$}. Assume that there is a $\rthin$
  class $\alpha$ of $G$ such that $v=\rho_{\alpha,s}$. Note that
  $s\neq\sigma(\alpha)$ by definition of $\rho_{\alpha,s}$. Lemma
  \ref{lem:N-n} implies that $y\notin N(\alpha)$ in $(G,\sigma)$.  After
  contraction of $e$, we have $\lca(\alpha,y)=\rho_{\alpha,s}$, thus
  $y\in N(\alpha)$ by Lemma \ref{lem:N-n}. Hence, $(T_e,\sigma)$ does not
  explain $G$; a contradiction.

  Conversely, assume that $e$ is an inner edge and for every
  $s\in\sigma(L(T(u)\setminus T(v)))$, there is no $\alpha\in\mathcal{N}$
  such that $v=\rho_{\alpha,s}$, i.e., for every $\alpha\in \mathcal{N}$
  and every color $s\neq \sigma(\alpha)$ we \NEW{either have (i)
    $v\succ \rho_{\alpha,s}$, (ii) $v\prec \rho_{\alpha,s}$, or (iii) $v$
    and $\rho_{\alpha,s}$ are incomparable. In the first two cases,}
  contraction of $e$ implies $v\succeq \rho_{\alpha,s}$ or
  $v\preceq \rho_{\alpha,s}$ in $(T_e,\sigma)$,
  respectively. \NEW{Therefore, since $L(T(w))=L(T_e(w))$ for any $w$
    incomparable to $v$, we have} $L(T(w))=L(T_e(w))$ for any node
  $w\neq v$. Moreover, it follows from Lemma \ref{lem:N-n} that
  $N_s(\alpha)=\{y\mid y\in L(T(\rho_{\alpha,s})),\sigma(y)=s\}$. This
  implies that the set $N_s(\alpha)$ remains unchanged after contraction of
  $e$ for all $\rthin$ classes $\alpha$ and all color $s\in S$. In other
  words, the in- and out-neighborhood of any leaf remain the same in
  $(T_e,\sigma)$. Hence, we conclude that $(T,\sigma)$ and $(T_e,\sigma)$
  explain the same graph $(G,\sigma)$.   
\end{proof}

Before we consider the general case, we show that 3-cBMGs like 2-cBMGs are
explained by unique least resolved trees.

\begin{lemma}\label{lem:lr-3}
  Let $(G,\sigma)$ be a connected 3-cBMG. Then there exists a unique least
  resolved tree $(T,\sigma)$ that explains $(G,\sigma)$.
\end{lemma}
\begin{proof}
  This proof uses arguments very similar to those in the proof of
  uniqueness result for 2-cBMGs. In particular, as in the proof of Theorem
  \ref{thm:lr-2}, we assume for contradiction that there exist 3-colored
  digraphs that are explained by two distinct least resolved trees. Let
  $(G,\sigma)$ be a minimal graph (w.r.t.\ the number of vertices) that is
  explained by the two distinct least resolved trees $(T_1,\sigma)$ and
  $(T_2,\sigma)$. W.l.o.g.\ we can choose a vertex $v$ and assume that its
  color is $r\in S$, i.e., $v\in L[r]$.  Using the same notation as in the
  proof of Theorem \ref{thm:lr-2}, we write $(T'_1,\sigma')$ and
  $(T'_2,\sigma')$ for the trees that are obtained by deleting $v$ from
  $(T,\sigma)$. These trees explain the uniquely defined graphs
  $(G'_1,\sigma')$ and $(G'_2,\sigma')$, respectively. Again, Lemma
  \ref{lem:sub} implies that $(G',\sigma'):=(G[L\setminus\{v\}],\sigma')$
  is a subgraph of both $(G'_1,\sigma')$ and $(G'_2,\sigma')$. Similar to
  the case of 2-cBMGs, we characterize the additional edges that are
  inserted into $(G'_1,\sigma')$ and $(G'_2,\sigma')$ compared to
  $(G',\sigma')$ in order to show that
  $(G'_1,\sigma')=(G'_2,\sigma')$. Assume that $uy$ is an edge in
  $(G'_1,\sigma')$ but not in $(G',\sigma')$. By analogous arguments as in
  the proof of Theorem \ref{thm:lr-2}, we find that $uv\in E(G)$ and in
  particular $N_r(u)=\{v\}$, i.e., $u$ has no out-neighbors of color $r$ in
  $(G',\sigma')$.

  Moreover, we have $u\in L[s]$, where $s\in S\setminus \{r\}$. Similar to
  the 2-color case, we now determine the outgoing arcs of $u$ in
  $(G'_1,\sigma')$ and $(G'_2,\sigma')$ by reconstructing the local
  structure of $(T_1,\sigma)$ and $(T_2,\sigma)$ in the vicinity of $v$.

  Observation \ref{obs-1} implies that the least resolved tree
  $(T_{rs},\sigma_{rs})$ explaining $(G_{rs},\sigma_{rs})$ is displayed by
  both $(T_1,\sigma)$ and $(T_2,\sigma)$. The local structure of
  $(T_{rs},\sigma_{rs})$ around $v$ is depicted in Fig.\
  \ref{fig:lr_unique}. Using the notation in the figure, $\{v\}$ is a
  $\rthin$ class by itself, $\alpha=\{v\}$, there is a $\rthin$ class
  $\beta'\subseteq L[s]$ with $N_r(\beta')=\{\alpha\}$ and
  $N_s(\alpha)=\{\beta'\}$, and there may or may not exist a
  $\beta\subseteq L[s]$ with $N_r(\beta)=N_r(\beta')=\{\alpha\}$ and
  $N_s(\alpha)\cap\beta'=\emptyset$. In addition, we have $\gamma\subseteq
  L[r]$, which is the $\preceq$-minimal $\rthin$ class of color $r$ such
  that $\rho_\gamma\succ \rho_\beta, \rho_{\beta'}$.  Recall that $uc$ with
  $c\in\gamma$ are all the edges on $L[r]\times L[s]$ that have been
  additionally inserted in both $(G'_1,\sigma')$ and
  $(G'_2,\sigma')$. Since every $\rthin$ class has at least one
  out-neighbor of each color and given the relationship between $\alpha$
  and $\beta'$, there exists a $\rthin$ class $\delta\subseteq L[t]$, where
  $t\in S\setminus\{r,s\}$, with $\alpha\subseteq N_r(\delta)$ and
  $\beta'\subseteq N_s(\delta)$ such that there is no other
  $\delta'\subseteq L[t]$ with $\rho_{\delta'}\prec\rho_\delta$.  If
  $N_r(\delta)\setminus \{\alpha\}\neq \emptyset$, then $\rho_\delta\succeq
  \rho_\gamma$ by Lemma \ref{lem:N-n}, and in particular there is no
  additional edge of the form $wa$ with $w\in L[t]$ and $a\in L[r]$ that is
  contained in $(G'_1,\sigma')$ and/or $(G'_2,\sigma')$ but not in
  $(G',\sigma')$. Therefore, only edges of the form $uc$ with $c\in\gamma$
  are additionally inserted into $(G'_1,\sigma')$ and $(G'_2,\sigma')$, and
  we conclude that $(G'_1,\sigma')=(G'_2,\sigma')$, which implies
  $(T'_1,\sigma')=(T'_2,\sigma')$ and therefore, since $v$ was arbitrary,
  $(T_1,\sigma')=(T_2,\sigma')$; a contradiction.

  Now consider the case $N_r(\delta)\setminus \{\alpha\}= \emptyset$.
  Since $\gamma\notin N_r(\delta)$, Lemma \ref{lem:N-n} ensures that
  $\rho_\delta \not\succeq \rho_\gamma$. The roots $\rho_{\gamma}$ and
  $\rho_{\delta}$ are comparable since $\alpha$ is an out-neighbor of both
  $\gamma$ and $\delta$.  Thus $\rho_\delta\prec \rho_\gamma$ and hence
  $N_r(\delta)=\{\gamma\}$ in $(T'_1,\sigma')$ as well as in
  $(T'_2,\sigma')$ after deletion of $v$. We still need to distinguish two
  cases: either we have $N_s(\delta)=\{\beta'\}$ or
  $N_s(\delta)=\{\beta',\beta\}$. In the first case, we have
  $\rho_\delta=\rho_{\beta'}=\rho_\alpha$ in $(T'_1,\sigma')$ as well as in
  $(T'_2,\sigma')$. In the second case, we obtain
  $\rho_\delta=\rho_{\beta}$, again this holds for both $(T'_1,\sigma')$
  and $(T'_2,\sigma')$. As before, we can conclude that
  $(T'_1,\sigma')=(T'_2,\sigma')$ and therefore
  $(T_1,\sigma')=(T_2,\sigma')$; a contradiction.   
\end{proof}

If $(G,\sigma)$ is not connected, we can build a least resolved tree
$(T,\sigma)$ analogously to the case of 2-cBMGs: we first construct the
unique least resolved tree $(T_i,\sigma_i)$ for each component
$(G_i,\sigma_i)$. Using Theorem \ref{thm:connected} we then insert an
additional root for $(T,\sigma)$ to which the roots of the $(G_i,\sigma_i)$
are attached as children. We proceed by showing that this construction
corresponds to the unique least resolved tree.

\begin{theorem} \label{thm:uT-conn} Let $(G,\sigma)$ be a (not necessarily
  connected) $n$-cBMG with $n\in \{2,3\}$. Then there exists a unique least
  resolved tree $(T,\sigma)$ that explains $(G,\sigma)$.
\end{theorem}
\begin{proof}
  Denote by $(G_i,\sigma_i)$ the connected components of $(G,\sigma)$.  By
  Theorem \ref{thm:lr-2} and Lemma \ref{lem:lr-3} there is a unique least
  resolved tree $(T_i,\sigma_i)$ that explains $(G_i,\sigma_i)$.  Hence, if
  $(G,\sigma)$ is connected, we are done.
	
  Now assume that there are at least two connected components.  Let
  $(T,\sigma)$ be a least resolved tree that explains $(G,\sigma)$.
  Theorem \ref{thm:connected} implies that there is a vertex
  $u\in \child(\rho_T)$ such that $L(G_i)\subseteq L(T(u))$ for each
  connected component $(G_i,\sigma_i)$. Hence, the subtree
  $(T(u),\sigma_{L(T(u))})$ displays the least resolved tree
  $(T_i,\sigma_i)$ explaining $(G_i,\sigma_i)$. Moreover, since
  $(T,\sigma)$ is least resolved, $\rho_Tu$ is a relevant edge, i.e., there
  must be a color $s\in \sigma(L(T\setminus T(u)))$ and a $\rthin$ class
  $\alpha$ such that $u=\rho_{\alpha,s}$ by Lemma \ref{lem:red-n}.

  This implies in particular that there exists a leaf
  $x\in L(T(u))\cap L[s]$. Lemma \ref{lem:N-n} now implies that the
  elements of $\alpha$ are connected to any element of color $s$ in the
  subtree $(T(u),\sigma_{L(T(u))})$. Furthermore, any leaf $y\in L(T(u))$
  has at least one out-neighbor of color $s$ in $L(T(u))$.  Hence, we can
  conclude that the graph $G(T(u),\sigma_{L(T(u))})$ induced by the subtree
  $(T(u),\sigma_{L(T(u))})$ is connected.

  Since $L(G_i)\subseteq L(T(u))$ and $(T(u),\sigma_{L(T(u))})$ explains
  the \emph{maximal connected} subgraph $(G_i,\sigma_i)$, we conclude that
  $G(T(u),\sigma_{L(T(u))}) = (G_i,\sigma_i)$.  By construction, both
  $(T(u),\sigma_{L(T(u))})$ and $(T_i,\sigma_i)$ are least resolved trees
  explaining the same graph, hence Theorem \ref{thm:lr-2} and Lemma
  \ref{lem:lr-3} imply $(T(u),\sigma_{L(T(u))})=(T_i,\sigma_i)$. In
  particular, thus, $\rho_{T_i}=u$.
	
  As a consequence, any least resolved tree $(T,\sigma)$ that explains
  $(G,\sigma)$ must be composed of the disjoint trees $(T_i,\sigma_i)$ that
  are linked to the root via the relevant edge $\rho_T\rho_{T_i}$.  Since
  every $(T_i,\sigma_i)$ and the construction of the edges
  $\rho_T\rho_{T_i}$ is unique, $(T,\sigma)$ is unique.   
\end{proof}

The characterization of redundant edges in trees explaining 2-cBMGs
together with the uniqueness of the least resolved trees for 3-cBMGs can be
used to characterize redundant edges in the general case, thereby
establishing the existence of a unique least resolved tree for $n$-cBMGs.

\begin{theorem}\label{thm:lr-n}
  For any connected $n$-cBMG $(G,\sigma)$, there exists a unique least
  resolved tree $(T',\sigma)$ that explains $(G,\sigma)$. The tree
  $(T',\sigma)$ is obtained by contraction of all redundant edges in an
  arbitrary tree $(T,\sigma)$ that explains $(G,\sigma)$. The set of all
  redundant edges in $(T,\sigma)$ is given by
  \begin{equation*}
    \mathfrak{E}_T= \left\{e=uv \mid v\notin L(T), v\ne\rho_{\alpha,s}
      \text{ for all } 
      s\in\sigma(L(T(u)\setminus T(v))) \text{ and }
      \alpha\in\mathcal{N} \right\}.
  \end{equation*}
  Moreover, $(T',\sigma)$ is displayed by $(T,\sigma)$.
\end{theorem}
\begin{proof}
  Using arguments analogous to the 2-color case one shows that there is a
  least resolved tree $(T',\sigma)$ that can be obtained from $(T,\sigma)$
  by contraction of all redundant edges.  The set of redundant edges is
  given by $\mathfrak{E}_T$ by Lemma \ref{lem:red-n}. By construction,
  $(T',\sigma)$ is displayed by $(T,\sigma)$.  It remains to show that
  $(T',\sigma)$ is unique. Observation \ref{obs-1} implies that for any
  pair of distinct colors $s$ and $t$ the corresponding unique least
  resolved tree $(T_{st},\sigma_{st})$ is displayed by $(T',\sigma)$.  The
  same is true for the least resolved tree $(T_{rst},\sigma_{rst})$ for any
  three distinct colors $r,s,t\in S$. Since for any 2-cBMG as well as for
  any 3-cBMG, the corresponding least resolved tree is unique (see Theorem
  \ref{thm:lr-2} and Lemma \ref{lem:lr-3}), it follows for any three
  distinct leaves $x,y,z\in L[r]\cup L[s]\cup L[t]$ that there is either a
  unique triple that is displayed by $(T_{rst},\sigma_{rst})$ or the least
  resolved tree $(T_{rst},\sigma_{rst})$ contains no triple on
  $x,y,z$. Note that we do not require that the colors $r,s,t$ are pairwise
  distinct. Instead, we use the notation $(T_{rst},\sigma_{rst})$ to also
  include the trees explaining the induced 2-cBMGs.  Observation
  \ref{obs-1} then implies that $\mathscr{R^*}:=\bigcup_{r,s,t \in S}
  r(T_{rst})\subseteq r(T')$.  Now assume that there are two distinct least
  resolved trees $(T_1,\sigma)$ and $(T_2,\sigma)$ that explain
  $(G,\sigma)$.  In the following we show that any triple displayed by
  $T_1$ must be displayed by $T_2$ and thus, $r(T_1)=r(T_2)$.

  Fig.\ \ref{fig:not-dspl} shows that there may be triples $xy|z \in
  r(T_1)\setminus \mathscr{R^*}$. Assume, for contradiction, 
  that $xy|z \notin r(T_2)\setminus\mathscr{R^*}$.  Fix the notation such 
  that $z\in\alpha$, $\sigma(x)=r$, $\sigma(y)=s$, and $\sigma(z)=t$. 
  We do not assume here that $r,s,t$ are necessarily pairwise distinct.
	
  In the remainder of the proof, we will make frequent use of the
  following\par\noindent\textbf{Observation:} \emph{If the tree $T$ is a
    refinement of $T'$, then we have $u\preceq_{T'} v$ if and only if
    $u\preceq_T v$ for all $u,v\in V(T')$.}
  \par\noindent In particular, $u\prec_{T'} v$ 
  (i.e., $u\preceq_{T'} v$ and $u\neq v$) implies $u\prec_{T} v$. The
  converse of the latter statement is still true if $u$ is a leaf in $T'$
  but not necessarily for arbitrary inner vertices $u$ and $v$.
	
  Let $u=\lca_{T_1}(x,y,z)$. The assumption $xy|z \in r(T_1)$ implies that
  there is a vertex $v\in \child(u)$ such that $v\succeq \lca_{T_1}(x,y)$. 
  Since $(T_1,\sigma)$ is least resolved the characterization of
  relevant edges ensures that there is a color $p\in \sigma(L(T_1(u)\setminus
  T_1(v)))$ and a $\rthin$ class $\beta$ with $\sigma(\beta)=q$ such that
  $v=\rho_{\beta,p}$.  In particular, there must be leaves $a\in L(T_1(v))$
  and $a^*\in L(T_1(u)\setminus T_1(v))$ with
  $\sigma(a)=\sigma(a^*)=p$. As a consequence we know that
  $a^* \notin N_p(b)$ for any $b\in \beta$.

  We continue to show that the edge $uv$ must also be contained in the
  least resolved tree $(T_{pq},\sigma_{pq})$ that explains the (not
  necessarily connected) graph $(G_{pq},\sigma_{pq})$. By Thm.\
  \ref{thm:uT-conn}, $(T_{pq},\sigma_{pq})$ is unique. Assume, for
  contradiction, that $uv$ is not an edge in $T_{pq}$. Recalling the
  arguments in Observation \ref{obs:2-colorISdisplayed}, the tree
  $(T_1,\sigma)$ must display $(T_{pq},\sigma_{pq})$. Thus, if $uv$ is not
  an edge in $T_{pq}$, then $v^*\coloneqq u=v$ in $T_{pq}$. By
  construction, we therefore have $v^*=\rho_{\beta,p}$ in
  $(T_{pq},\sigma_{pq})$. Since $(T_{pq},\sigma_{pq})$ is least resolved,
  it follows from Cor.  \ref{cor:rel_root} that $b\in \child(v^*)$ for all
  $b\in \beta$ in $(T_{pq},\sigma_{pq})$. The latter, together with
  $a,a^*\preceq_{T_{pq}} v^*$, implies that $\lca_{T_{pq}}(a,\beta) =
  \lca_{T_{pq}}(a^*,\beta)= v^*$.  However, this implies $a^* \in
  N_p(\beta)$, a contradiction.

  To summarize, the edge $uv$ must be contained in the least resolved tree
  $(T_{pq},\sigma_{pq})$. Moreover, by Observation
  \ref{obs:2-colorISdisplayed}, $(T_{pqo},\sigma_{pqo})$ is a refinement of
  $(T_{pq},\sigma_{pq})$ for every color $o\in S$. Hence, we have
  $v\prec_{T_{pqo}} u$, which is in particular true for the color $o\in
  \{r,s,t\}$.  Moreover, we know that $x \prec_{T_{pqr}} v$ and $y
  \prec_{T_{pqs}} v$ because $(T_1,\sigma)$ is a refinement of both
  $(T_{pqr},\sigma_{pqr})$ and $(T_{pqs},\sigma_{pqs})$.

  Since $(T_2,\sigma)$ is also a refinement of both
  $(T_{pqr},\sigma_{pqr})$ and $(T_{pqs},\sigma_{pqs})$, we have
  $x,y\prec_{T_2} v \prec_{T_2} u$.  Furthermore, $v \prec_{T_1}
  \lca_{T_1}(v,z)=u$ and $z \not\preceq_{T_1}$ implies that
  $z\prec_{T_{pqt}} u$ and $z \not\preceq_{T_{pqt}} v$. Therefore, $z
  \prec_{T_2} u$ and $z \not\preceq_{T_2} v$. Combining these facts about
  partial order of the vertices $v,u,x,y$ and $z$ in $T_2$, we obtain 
  $xy|z\in r(T_2)$; a contradiction.  
		
  Hence, $r(T_1) = r(T_2)$. Since $r(T_1)$ uniquely identifies the
  structure of $T_1$ (cf.\ \citet[Thm.\ 6.4.1]{Semple:03}), we conclude
  that $(T_1,\sigma)=(T_2,\sigma)$.  The least resolved tree explaining
  $(G,\sigma)$ is therefore unique.
       
\end{proof}

\begin{figure}[]
  \begin{center} 
        \includegraphics[width=0.4\textwidth]{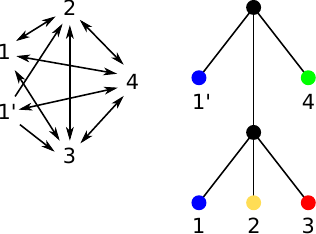}
        \caption[]{A connected graph $(G,\sigma)$ and the corresponding
          least resolved tree $(T,\sigma)$ on five vertices of four colors:
          blue (1 and 1'), yellow (2), red (3), and green (4). The triple
          $23|4$ is displayed by $(T,\sigma)$ but it is not displayed by
          the least resolved tree $(T',\sigma')$ that explains the induced
          subgraph $(G',\sigma')$ with $V(G')=\{2,3,4\}$ since
          $(T',\sigma')$ is simply the star tree on $\{2,3,4\}$. Hence,
          $23|4\notin \mathscr{R}^*=\bigcup_{r,s,t\in S}r(T_{rst})$.}
        \label{fig:not-dspl} 
  \end{center}
\end{figure}

\begin{corollary} 
  Every $n$-cBMG $(G,\sigma)$ is explained by the unique least resolved
  tree $(T,\sigma)$ consisting of the least resolved trees $(T_i,G_i)$
  explaining the connected components $(G_i,\sigma_i)$ and an additional root
  $\rho_T$ to which the roots of the $(T_i,G_i)$ are attached as children.
\end{corollary} 
\begin{proof} 
  It is clear from the construction that $(T,\sigma)$ explains
  $(G,\sigma)$.  The proof that his is the only least resolved tree
  parallels the arguments in the proof of Theorem \ref{thm:uT-conn} for
  2-cBMGs and 3-cBMGs.
\end{proof}

Since a tree is determined by all its triples, it is clear now that the
construction of a tree that explains a connected $n$-cBMG is essentially a
supertree problem: it suffices to find a tree, if it exists, that displays
the least resolved trees explaining the induced subgraphs on 3 colors.
\NEW{In the following, we write
  $$R:=\bigcup_{s,t \in S}r(T^*_{s,t})$$ 
  for the union of all triples in the least resolved trees
  $(T^*_{st},\sigma_{st})$ explaining the 2-colored subgraphs
  $(G_{st},\sigma_{st})$ of $(G,\sigma)$. In contrast, the set of all
  \emph{informative} triples of $(G,\sigma)$, as specified in Def.\
  \ref{def:inftriples}, is denoted by $\mathscr{R}(G,\sigma)$. As an
  immediate consequence of Lemma \ref{lem:consistent} we have
  \begin{equation} 
    \mathscr{R}(G,\sigma)\subseteq R
  \end{equation} 
}

\begin{theorem}
  A connected colored digraph $(G,\sigma)$ is an $n$-cBMG if and only if
  (i) all induced subgraphs $(G_{st},\sigma_{st})$ on two colors are
  2-cBMGs and (ii) the union $R$ of \NEW{all triples} obtained from their
  least resolved trees $(T_{st},\sigma_{st})$ forms a consistent set.
  \marginpar{\color{blue}\scriptsize
    The theorem also needs the condition\\
    (iii) $(G(\Aho(R),\sigma))=(G,\sigma)$.}
  In particular, $\Aho(R)$ is the unique least resolved tree that explains
  $(G,\sigma)$.
\label{thm:ncBMG} 
\end{theorem}
\begin{proof}
  Let $(G,\sigma)$ be an $n$-cBMG that is explained by a tree
  $(T,\sigma)$. Moreover, let $s$ and $t$ be two distinct colors of $G$ and
  let $L':= L[s]\cup L[t]$ be the subset of vertices with color $s$ and
  $t$, respectively.  Observation \ref{obs-1} states that the induced
  subgraph $(G[L'],\sigma)$ is a 2-cBMG that is explained by
  $(T_{L'},\sigma')$. In particular, the least resolved tree
  $(T^*_{L'},\sigma')$ of $(T_{L'},\sigma')$ also explains $(G[L'],\sigma)$
  and $T^*_{L'} \le T_{L'}\le T$ by Theorem \ref{thm:lr-n}, i.e.,
  $r(T^*_{L'})\subseteq r(T)$. Since this holds for all pairs of two
  distinct colors, the union of the triples obtained from the set of all
  least resolved 2-cBMG trees
  $R$ is displayed by
  $T$. In particular, therefore, $R$ is consistent.

  Conversely, suppose that $(G[L'],\sigma)$ is a 2-cBMG for any two
  distinct colors $s,t$ and $R$ is consistent. Let
  $\Aho(R)$ be the tree that is constructed by \texttt{BUILD} for
  the input set $R$. This tree displays $R$ and is a
  least resolved tree \cite{Aho:81} in the sense that we cannot contract
  any edge in $\Aho(R)$ without loosing a triple from
  $R$.  By construction, any triple that is displayed by
  $(T_{st},\sigma_{st})$ is also displayed by $\Aho(R)$,
  i.e. $(T_{st},\sigma_{st})\le \Aho(R)$.
\marginpar{\color{blue}\scriptsize This statement is incorrect.}
  {\color{red}Hence, for any
  $\alpha\in\mathcal{N}$ and any color $s\neq\sigma(\alpha)$ the
  out-neighborhood $N_s(\alpha)$ is the same w.r.t.\ $(T_{st},\sigma_{st})$
  and w.r.t.\ $\Aho(R)$.}
\begingroup\color{gray}
  Since this is true for any $\rthin$
  class of $G$, also all in-neighborhoods are the same in
  $\Aho(R)$ and the corresponding
  $(T_{st},\sigma_{st})$. Therefore, we conclude that $\Aho(R)$
  explains $(G,\sigma)$, i.e., $(G,\sigma)$ is an $n$-cBMG. 

  In order to see that $\Aho(R)$ is a least resolved tree
  explaining $(G,\sigma)$, we recall that the contraction of an edge leaves
  at least on triple unexplained, see \citet[Prop.\
  4.1]{SEMPLE2003489}. Since $R$ consists of all the triples
  $r(T_{st})$ that in turn uniquely identify the structure of
  $(T_{st},\sigma_{st})$ (cf.\ \citet[Thm.\ 6.4.1]{Semple:03}), none of
  these triples is dispensable. The contraction of an edge in
  $\Aho(R)$ therefore yields a tree that no longer displays
  $(T_{st},\sigma_{st})$ for some pair of colors $s,t$ and thus no longer
  explains $(G,\sigma)$. Thus, $\Aho(R)$ contains no redundant
  edges and we can apply Theorem \ref{thm:lr-n} to conclude that
  $\Aho(R)$ is the unique least resolved tree that explains
  $(G,\sigma)$.
  \endgroup
  \marginpar{\color{blue}\scriptsize See Corrigendum for the proof of the
    corrected statement.}
\end{proof}

\begin{figure}[htb]
\begin{center}
  \includegraphics[width=0.95\textwidth]{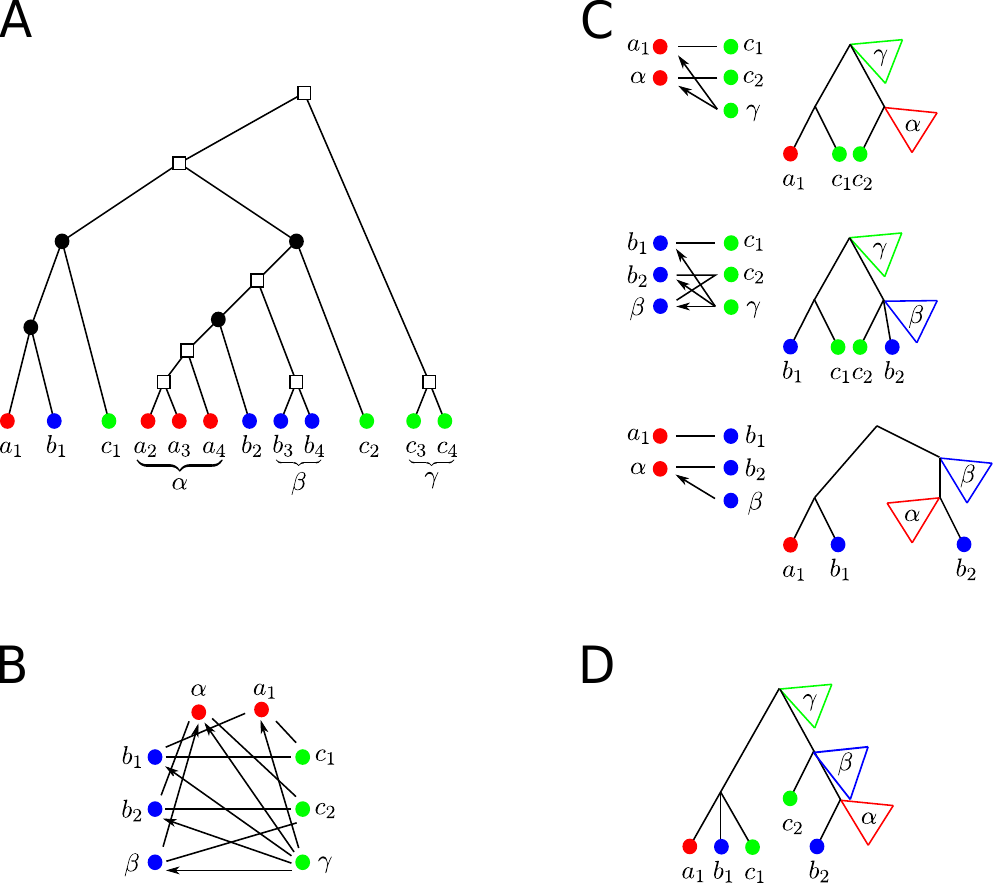} 
\end{center}
\caption[]{Construction of the least resolved tree explaining the colored
  best match graph. \NEW{Panel (A) recalls the event-labeled gene tree of
    the evolutionary scenario shown in Fig.\ \ref{fig:exmpl-distance}.}
  There are three $\rthin$ classes with more than one element:
  $\alpha=\{a_2,a_3,a_4\}$, $\beta=\{b_3,b_4\}$ and $\gamma=\{c_3,c_4\}$ in
  the 3-cBMG graph $(G,\sigma)$ shown in panel (B).  For simplicity of
  presentation, the $\rthin$ classes are already collapsed into single
  vertices. Panel (C) lists the three induced subgraphs of $(G,\sigma)$ on
  two colors together with their least resolved trees. By construction,
  $(G,\sigma)$ is the union of the three subgraphs on two colors. (D) The
  Aho-Tree for the set of all triples obtained from the least resolved
  trees shown in (C). This tree explains the graph $(G,\sigma)$ and is the
  unique least resolved tree w.r.t.\ $(G,\sigma)$.}
\label{fig:n-color}
\end{figure}

Fig.\ \ref{fig:n-color} summarizes the construction of the least resolved
tree from the 3-colored digraph $(G,\sigma)$ shown \NEW{in} Fig.\
\ref{fig:n-color}(B). For simplicity we assume that we already know that
$(G,\sigma)$ is indeed a 3-cBMG. For each of the three colors the example
has four genes. In addition to singleton there are three non-trivial $\rthin$
classes $\alpha=\{a_2,a_3,a_4\}$, $\beta=\{b_3,b_4\}$ and $\gamma=\{c_3$,
$c_4\}$.  Following Theorem \ref{thm:ncBMG}, we extract for each of the
three pairs of colors the induced subgraphs $(G_{st},\sigma_{st})$ and
construct the least resolved trees that explain them (Fig.\
\ref{fig:n-color}(C)).  Extracting all triples from these least resolved
trees on two colors yields the triple set $\mathscr{R}$, which in this case
is consistent. Theorem \ref{thm:ncBMG} implies that the tree
$\Aho(\mathscr{R})$ (shown in the lower right corner) explains $(G,\sigma)$
and is in particular the unique least resolved tree w.r.t.\ $(G,\sigma)$.

  We close this section by showing that in fact the informative triples of
  all $(G_{st},\sigma_{st})$ are already sufficient to decide whether
  $(G,\sigma)$ is an $n$-cBMG or not. More precisely, we show
  \begin{lemma}
    If $(G,\sigma)$ is an $n$-cBMG  then
      $\Aho(\mathscr{R}(G,\sigma))=\Aho(R)$.
  \end{lemma}
\begin{proof}  
  We first observe that the two triple sets $R$ and
  $\mathscr{R}:=\mathscr{R}(G,\sigma)$ have the same Aho tree
  $\Aho(R) = \Aho(\mathscr{R})$ if, in each step of \texttt{BUILD}, the
  respective Aho-graphs $[R,L']$ and $[\mathscr{R},L']$, as defined at the
  beginning of this section, have the same connected components. It is not
  necessary, however, that $[R,L']$ and $[\mathscr{R},L']$ are
  isomorphic. In the following set $T=\Aho(R)$.

  If $T$ is the star tree on $L$, then
  $\mathscr{R}\subseteq R = \emptyset$, thus $[R,L]=[\mathscr{R},L]$ is the
  edgeless graph on $L$, hence in particular
  $\Aho(\mathscr{R})=\Aho(R)$. 

	Now suppose $T$ is not the star tree. Then
  there is a vertex $w\in V^0(T)$ such that $L(T(w))=\child(w)$. For
  simplicity, we write $L_w:=L(T(w))$. Since $(T(w),\sigma_{L_w})$ is a
  star tree, we can apply the same argument again to conclude that
  $[R_{|L_w},L_w]=[\mathscr{R}_{|L_w},L_w]$, hence both Aho-graphs have the
  same connected components.
  Now let $u=\rho_T$ and assume by induction that $[R_{|L_{u'}},L_{u'}]$
  and $[\mathscr{R}_{|L_{u'}},L_{u'}]$ have the same connected components
  for every $u'\prec_T u$, and thus, in particular, for
  $v\in\child(u)$. Consequently, for any $v_i\in\child(v)$ the set
  $L_{v_i}$ is connected in $[\mathscr{R}_{|L_{v}},L_{v}]$. Since
  $\mathscr{R}_{|L_{v}}\subseteq \mathscr{R}_{|L_{u}}$, the set $L_{v_i}$
  must also be connected in $[\mathscr{R}_{|L_{u}},L_{u}]$ for every
  $v_i\in\child(v)$ (cf.\ Prop.\ 8 in \cite{Bryant:95}). It remains to show
  that all $L_{v_i}$ are connected in
  $[\mathscr{R}_{|L_{u}},L_{u}]$.

  Since $(T,\sigma)$ is least resolved w.r.t.\ $(G,\sigma)$, it follows
  from Theorem \ref{thm:lr-n} that $v=\rho_{\alpha,s}$ for some color
  $s\in \sigma(L(T(u)\setminus T(v)))$ and an $\rthin$ class $\alpha$ with
  $\sigma(\alpha)\neq s$. In particular, therefore,
  $s\notin \sigma(L_{v_i})$ if $\alpha\in L_{v_i}$ (say $i=1$). By
  definition of $s$, there must be a $v_j\in \child(v)\setminus \{v_1\}$
  (say $j=2$) such that $s\in \sigma(L_{v_2})$. Let
  $y\in L_{v_2}\cap L[s]$. Lemma \ref{lem:N-n} implies $y\in N_s(\alpha)$,
  i.e., $\alpha y \in E(G)$. Moreover, by definition of $s$, there must be
  a leaf $y'\in L(T(u)\setminus T(v))\cap L[s]$. Since
  $\lca(\alpha,y)\prec_T\lca(\alpha,y')$, we have $\alpha y'\notin E(G)$,
  whereas $y'\alpha$ may or may not be contained in
  $(G,\sigma)$. Therefore, the induced subgraph on $\{\alpha y y'\}$ is of
  the form $X_1$, $X_2$, $X_3$, or $X_4$ and thus provides the informative
  triple $\alpha y | y'$. It follows that $L_{v_1}$ and $L_{v_2}$ are
  connected in $[\mathscr{R}_{|L_{u}},L_{u}]$. In particular,
  this implies that any $L_{v_j}$ with
  $\sigma(L_{v_j})\subseteq \sigma(L_v)$ containing $s$ is connected to any
  $L_{v_i}$ that does not contain $s$.  Since $(G,\sigma)$ is connected,
  such a set $L_{v_i}$ always exists by Theorem \ref{thm:connected}. Now
  let $L_1:=\{L_{v_j}\mid v_j\in\child(v), s\in\sigma(L_{v_j}) \}$ and
  $L_2:=\{L_{v_i}\mid v_i\in\child(v), s\notin\sigma(L_{v_i}) \}$.  It then
  follows from the arguments above that $L_1$ and $L_2$ form a complete
  bipartite graph, hence $[\mathscr{R}_{|L_{u}},L_{u}]$ is connected.  
\end{proof}

\NEW{%
As an immediate consequence, Theorem \ref{thm:ncBMG} can be rephrased as:
\begin{corollary}
  A connected colored digraph $(G,\sigma)$ is an $n$-cBMG if and only if
  (i) all induced subgraphs $(G_{st},\sigma_{st})$ on two colors are
  2-cBMGs and (ii) the union $\mathscr{R}$ of informative triples
  $\mathscr{R}(G_{st},\sigma_{st})$ obtained from the induced subgraphs
  $(G_{st},\sigma_{st})$ forms a consistent set. In particular,
  $\Aho(\mathscr{R})$ is the unique least resolved tree that explains
  $(G,\sigma)$.
\label{cor:ncBMG} 
\end{corollary}
}

\section{Algorithmic Considerations} 

The material in the previous two sections can be translated into practical
algorithms that decide for a given colored graph $(G,\sigma)$ whether it is
an $n$-cBMG and, if this is the case, compute the unique least resolved tree
that explains $(G,\sigma)$. The correctness of Algorithm~\ref{alg:ncBMG}
follows directly from Theorem \ref{thm:ncBMG} (for a single connected
component) and Theorem \ref{thm:connected} regarding the composition of
connected components. It depends on the construction of the unique least
resolved tree for the connected components of the induced 2-cBMGs, called
\texttt{LRTfrom2cBMG}() in the pseudocode of
Algorithm~\ref{alg:ncBMG}. There are two distinct ways of computing these
trees: either by constructing the hierarchy $T(\mathscr{H})$ from the
extended reachable sets $R'$ (Algorithm~\ref{alg:2cBMG}) or via
constructing the Aho tree from the set of informative triples
(Algorithm~\ref{alg:2cBMGit}). While the latter approach seems simpler, we
shall see below that it is in general slightly less efficient. Furthermore,
we use a function \texttt{BuildST}() to construct the supertree from a
collection of input trees. Together with the computation of $\Aho()$ from a
set of triples, it will be briefly discussed later in this section.

\begin{algorithm}
\caption{Unique least resolved tree of $n$-cBMG} 
\label{alg:ncBMG}
\algsetup{linenodelimiter=}
\begin{algorithmic}
\REQUIRE Vertex colored digraph $(G(L,E),\sigma)$. 
\IF { there is $xy\in E$ with $\sigma(x)=\sigma(y)$} 
    {\STATE \textbf{exit}(\emph{``not a BMG''})}
\ENDIF
\STATE determine connected components $(G_i(L_i,E_i),\sigma_i)$
\IF { $\sigma(L_i)\ne \sigma(L_j)$ for some components $i$, $j$} 
    { \STATE \textbf{exit}(\emph{``not a BMG''})}
\ENDIF 
\FORALL {connected components $(G_i(L_i,E_i),\sigma_i)$} 
  \FORALL {colors $s,t\in S$, $s\ne t$} 
     \STATE determine the induced subgraph 
        $(G_{st}(L_{st},E_{st}),\sigma_{st})$ with colors $s,t$ 
     \STATE determine connected components $(G_{st,i},\sigma_{st,i})$ 
     \FORALL {connected components $(G_{st,i},\sigma_{st,i})$} 
        \STATE $(T_{st,i},\sigma_{st,i}) \leftarrow$ 
           \texttt{LRTfrom2cBMG}($G_{st,i},\sigma_{st,i}$)
        \IF {$(T_{st,i},\sigma_{st,i})=\varnothing$} 
           { \STATE \textbf{exit}(\emph{``not a BMG''})}
        \ENDIF
     \ENDFOR 
      \STATE $(T_{st},\sigma_{st}) \leftarrow$ root $r_{st}$ with 
         children $(T_{st,i},\sigma_{st,i})$
  \ENDFOR 
  \STATE $(T_i,\sigma_i)\leftarrow$ 
     \texttt{BuildST($\bigcup_{s,t} (T_{st},\sigma_{st})$)}
  \IF {$(T_i,\sigma_i)=\varnothing$} 
      { \STATE  \textbf{exit}(\emph{``not a BMG''})}
  \ENDIF
\ENDFOR
\STATE $(T,\sigma)\leftarrow$ root $r$ with 
         children $(T_i,\sigma_i)$  
\STATE \textbf{return} $(T,\sigma)$
\end{algorithmic}
\end{algorithm}

\begin{algorithm}
\caption{Unique least resolved tree of connected 2-cBMG} 
\label{alg:2cBMG}
\algsetup{linenodelimiter=}
\begin{algorithmic}
\REQUIRE Two-colored connected bipartite digraph $(G(L,E),\sigma)$. 
\STATE compute $\rthin$ classes 
\STATE compute $N(\alpha)$ and $N(N(\alpha))$ for all $\alpha$ 
\IF {\AX{(N2)} does not hold for all $\alpha$} 
   \STATE \textbf{return} $\varnothing$
\ENDIF 
\IF {\AX{(N3)} does not hold for all $\alpha,\beta$} 
   \STATE \textbf{return} $\varnothing$
\ENDIF
\STATE compute table $Y_{\alpha\beta}=1$ iff
       $N(\alpha)\cap N(N(\beta))\ne\emptyset$ 
\IF {\AX{(N1)} does not hold for all $\alpha,\beta$} 
   \STATE \textbf{return} $\varnothing$
\ENDIF
\STATE compute $R(\alpha)$, $Q(\alpha)$, and $R'(\alpha)=
       R(\alpha)\cup Q(\alpha)$ for all $\alpha$
\STATE tabulate $P_{\alpha,\beta}=1$ iff $R'(\alpha)\subseteq R'(\beta)$.
\STATE compute Hasse $T(\mathfrak{H})$ diagram by transitive reduction
\IF {$T(\mathfrak{H})$ is not a tree} 
   \STATE \textbf{return} $\varnothing$
\ENDIF 
\IF {there are siblings $R'(\alpha)$ and $R'(\beta)$ in 
  $T(\mathfrak{H})$ with non-empty intersection}
   \STATE \textbf{return} $\varnothing$
\ENDIF 
\STATE construct $T^*(\mathfrak{H})$ by attaching the leaves to
       $T(\mathfrak{H})$
\STATE \textbf{return} $T^*(\mathfrak{H})$
\end{algorithmic}
\end{algorithm}

\begin{algorithm}
\caption{Unique least resolved tree of connected 2-cBMG via triples} 
\label{alg:2cBMGit}
\algsetup{linenodelimiter=}
\begin{algorithmic}
\REQUIRE Two-colored connected bipartite digraph $(G(L,E),\sigma)$. 

\STATE extract informative triple set $\mathscr{R}$ from $(G,\sigma)$ 
\STATE $(T,\sigma)\leftarrow\Aho(\mathscr{R},\sigma)$ 
\STATE compute $G(T,\sigma)$ 
\IF {$G(T,\sigma)=(G,\sigma)$} 
  \STATE \textbf{return} $(T,\sigma)$ 
\ELSE 
  \STATE \textbf{return} $\varnothing$ 
\ENDIF
\end{algorithmic}

\end{algorithm}

Let us now turn to analyzing the computational complexity of Algorithm
\ref{alg:ncBMG}, \ref{alg:2cBMG}, and \ref{alg:2cBMGit}. We start with the
building blocks necessary to process the 2-cBMG and consider performance
bounds on individual tasks.

\paragraph{From $(T,\sigma)$ to $(G,\sigma)$.} Given a leaf-labeled tree
$(T,\sigma)$ we first consider the construction of the corresponding cBMG.
The necessary lowest common ancestor queries can be answered in constant
time after linear time preprocessing, see e.g.\
\cite{Harel:84,Schieber:88}.  The $\lca()$ function can also be used to
express the partial orders among vertices since \NEW{we} have $x\preceq y$
if and only if $\lca(x,y)=y$. In particular, therefore,
$\lca(x,y)\preceq\lca(x,y')$ is true if and only if
$\lca(\lca(x,y),\lca(x,y'))=\lca(\lca(x,y),y')=\lca(x,y')$. Thus
$(G,\sigma)$ can be constructed from $(T,\sigma)$ \NEW{by computing
  $\lca(x,y)$ in constant time for each leaf $x$ and each $y\in L[s]$.
  Since the last common ancestors for fixed $x$ are comparable, their
  unique minimum can be determined in $O(|L[s]|)$ time. Thus we can
  construct all best matches in $O( |L|+ |L|\sum_s |L|)=O(|L|^2)$ time.}

\paragraph{Thinness classes.}
Recall that each connected component of a cBMG $(G,\sigma)$ has vertices
with all $|S|\ge 2$ colors (we disregard the trivial case of the edge-less
graph with $|S|=1$) and thus every $x\in V$ has a non-zero
out-degree. Therefore $|E|\ge|L|$, i.e.,
$O(|L|+|E|)=O(|E|)\NEW{=} O(|L|^2)$.

Consider a collection $\mathcal{F}$ of $n=|\mathcal{F}|$ subsets on $L$
with a total size of $m=\sum_{A\in\mathcal{F}}|A|$. Then the set inclusion
poset of $\mathcal{F}$ can be computed in $O(nm)$ time and $O(n^2)$ space
as follows: For each $A\in\mathcal{F}$ run through all elements $x$ of all
other sets $B\in\mathcal{F}$ and mark $B\not\subseteq A$ if $x\notin A$,
resulting in a $n\times n$ table $P_{\mathcal{F}}$ storing the set
inclusion relation.  More sophisticated algorithms that are slightly more
efficient under particular circumstances are described in
\cite{Pritchard:95,Elmasry:10}.

\NEW{In order to compute the thinness classes, we observe that the
  symmetric part of $P_{\mathcal{F}}$ corresponds to equal sets.} The
classes of equal sets can be obtained as connected components by breadth
first search on the symmetric part of $P_{\mathcal{F}}$ with an effort of
$O(n^2)$.  \NEW{This procedure is separately applied to the in- and
  out-neighborhoods of the cBMG.  Using an auxiliary graph in which
  $x,y\in L$ are connected if they are in the same component for both the
  in- and out- neighbors, the thinness classes can now be obtained by
  another breath first search in $O(n^2)$.} Since we have $n=|L|$ and
$m=|E|$ and thus the sets of vertices with equal in- and out-neighborhoods
can be identified in $O(|L|\,|E|)$ total time.

\paragraph{Recognizing 2-cBMGs.} 
Since \AX{(N0)} holds for all graphs, it will be useful to construct the
table $X$ with entries $X_{\alpha,\beta}=1$ if $\alpha\subseteq N(\beta)$
and $X_{\alpha,\beta}=0$ otherwise. This table can be constructed in
$O(|E|)$ time by iterating over all edges and retrieving (in constant time)
the $\rthin$ classes to which its endpoints belong. The $N(N(\alpha))$ can
now be obtained in $O(|E|\,|L|)$ by iterating over all edges $\alpha\beta$
and adding the classes in $N(\beta)$ to $N(N(\alpha))$.  We store this
information in a table \NEW{with entries $Q_{\alpha,\beta}=1$ if
  $\alpha\in N(N(\beta))$ and $Q_{\alpha,\beta}=0$ otherwise,} in order to
be able to decide membership in constant time later on.

A table $Y_{\alpha\beta}$ with $Y_{\alpha\beta}=0$ if
$N(\alpha)\cap N(N(\beta))=\emptyset$ and $Y_{\alpha\beta}=1$ if there is
an overlap between $N(\alpha)$ and $N(N(\beta))$ can be computed in
$O(|L|^3)$ time from the membership \NEW{tables $X$ and $Q$ for
  neighborhoods $N(\,.\,)$ and next-nearest neighborhoods $N(N(\,.\,))$,
  respectively.} From the membership table for $N(N(\alpha))$ and
$N(\gamma)$ we obtain $N(N(N(\alpha)))$ in $O(|E|\,|L|)$ time, making use
of the fact that $\sum_{\alpha}|N(\alpha)|=|E|$.  For fixed
$\alpha,\beta\in\mathcal{N}$ it only takes constant time to check the
conditions in \AX{(N1)} and \AX{(N3)} since all set inclusions and
intersections can be tested in constant time using the auxiliary data
derived above. The inclusion \AX{(N2)} can be tested directly in $O(|L|)$
time for each $\alpha$. We can summarize considerations above as
\begin{lemma} 
  A 2-cBMG can be recognized in $O(|L|^2)$ space and $O(|L|^3)$ time 
  with Algorithm~\ref{alg:2cBMG}.
\end{lemma}

\paragraph{Reconstruction of $T^*(\mathscr{H})$.}
For each $\alpha\in\mathcal{N}$, the reachable set $R(\alpha)$ can be found
by a breadth first search in $O(|E|)$ time, and hence with total complexity
$O(|E|\,|L|)$. For each $\alpha$, we can find all $\beta\in\mathcal{N}$
with $N^-(\beta)=N^-(\alpha)$ and $N(\beta)\subseteq N(\alpha)$ in $O(|L|)$
time by simple look-ups in the set inclusion table for the in- and
out-neighborhoods, respectively. Thus we can find all auxiliary leaf sets
$Q(\alpha)$ in $O(|L^2|)$ time and the collection of the $R'(\alpha)$ can
be constructed in $O(|E|\,|L|)$.

The construction of the set inclusion poset is also useful to check whether
the $\{R'(\alpha)\}$ form a hierarchy. In the worst case we have a tree of
depth $|L|$ and thus $m=O(|L|^2)$. Since the number of $\rthin$ classes is
bounded by $O(|L|)$, the inclusion poset of the reachable sets can be
constructed in $O(|L|^3)$. The Hasse diagram of the partial order is the
unique transitive reduction of the corresponding digraph. In our setting,
this also takes $O(|L|^3)$ time \cite{Gries:89,Aho:72}, since the inclusion
poset of the $\{R'(\alpha)\}$ may have $O(|L|^2)$ edges. It is now easy to
check whether the Hasse diagram is a tree or not. If the number of edges is
at least the number of vertices, the answer is negative. Otherwise, the
presence of a cycle can be verified e.g.\ using breadth first search in
$O(|L|)$ time. It remains to check that the non-nested sets $R(\alpha)$ are
indeed disjoint. It suffices to check this for the children of each vertex
in the Hasse tree. Traversing the tree top-down this can be verified in
$O(|L|^2)$ time since there are $O(|L|)$ vertices in the Hasse diagram and
the total number of elements in the subtrees is $O(|L|)$.

Summarizing the discussion so far, and using the fact that the vertices
$x\in\alpha$ can be attached to the corresponding vertices $R'(\alpha)$ in
total time $O(|L|)$ we obtain
\begin{lemma} 
  The unique least resolved tree $T^*(\mathscr{H}')$ of a connected 2-cBMG
  $(G,\sigma)$ can be constructed in $O(|L|^3)$ time and $O(|L|^2)$ space
  with Algorithm~\ref{alg:2cBMG}.
\end{lemma} 

\paragraph{Informative triples.} 
Since all informative triples $\mathscr{R}(G,\sigma)$ come from an induced
subgraph that contains at least one edge, it is possible to extract
$\mathscr{R}(G,\sigma)$ for a connected 2-cBMG in $O(|E|\,|L|)$
time. Furthermore, the total number of vertices and edges in
$\mathscr{R}(G,\sigma)$ is also bounded by $O(|E|\,|L|)$, hence the
algorithm of Deng and Fern{\'a}ndez-Baca can be used to construct the tree
$\Aho(\mathscr{R}(G,\sigma))$ for a connected 2-cBMG in
$O(|E|\,|L| \log^2(|E|\,|L|) )$ time \cite{Deng:18}. The graph $(G',\sigma)$ explained by
this tree can be generated in $O(|L|^3)$ time, and checking whether
$(G,\sigma)=(G',\sigma)$ requires $O(|L|^2)$ time. Asymptotically, the
approach via informative triples, Alg.~\ref{alg:2cBMGit}, is therefore at
best as good as the direct construction of the least resolved tree
$T^*(\mathscr{H}')$ with Alg.~\ref{alg:2cBMG}.

\paragraph{Effort in the $n$-color case.} 
For $n$-cBMGs it is first of all necessary to check all pairs of induced
2-cBMGs. The total effort for processing all induced 2-cBMGs is
$O(\sum_{s<t} (|L[s]|+|L[t]|)^3) \le O(|S|\,|L|\,\ell^2 + |L|^2\ell)$ with
$\ell:=\max_{s\in S} |L[s]|$, as shown by a short computation.

The 2-cBMG for colors $s$ and $t$ is of size $O(L[s]+L[t])$ hence the total
size of all $|S|(|S|-1)/2$ 2-cBMGs is $O(|S|\,|L|)$. The total effort to
construct a supertree from these 2-cBMGs is therefore only
$O(|L|\,|S| \log^2(|L|\,|S|))$ \NEW{\cite{Deng:18}}, and thus negligible
compared to the effort of building the 2-cBMGs.

\NEW{Using Lemma \ref{cor:ncBMG} it is also possible to use the set of all
  informative triples directly. Its size is bounded by $O(|L|\,|E|)$, hence
  the algorithm of \citet{Henzinger:99} can used to construct the supertree
  on $O( |L|\,|E|\log^2(|L|\,|E| )$. This bound is in fact worse than for
  the strategy of constructing all 2-cBMGs first.}

We note, finally, that for practical applications the number of genes
between different species will be comparable, hence $O(\ell)=O(|L|/|S|)$.
The total effort of recognizing an $n$-cBMG in a biologically realistic
application scenario amounts to $O(|L|^3/|S|)$. In the worst case scenario
with $O(\ell)=O(|L|)$, the total effort is $O(|S|\, |L|^3)$.

\section{Reciprocal Best Match Graphs} 

Several software tools implementing methods for tree-free orthology
assignment are typically on reciprocal best matches, i.e., the symmetric
part of a cBMG, which we will refer to as \emph{colored Reciprocal Best
  Match Graph} (cRBMG). Orthology is well known to have a cograph structure
\cite{Hellmuth:13a,HW:16a,HW:16book}. The example in
Fig.~\ref{fig:counterCOG} shows, however, that cRBMG in general are not
cographs. It is of interest, therefore to better understand this class of
colored graphs and their relationships with cographs.

\begin{figure}
  \begin{center} 
        \includegraphics[width=0.5\textwidth]{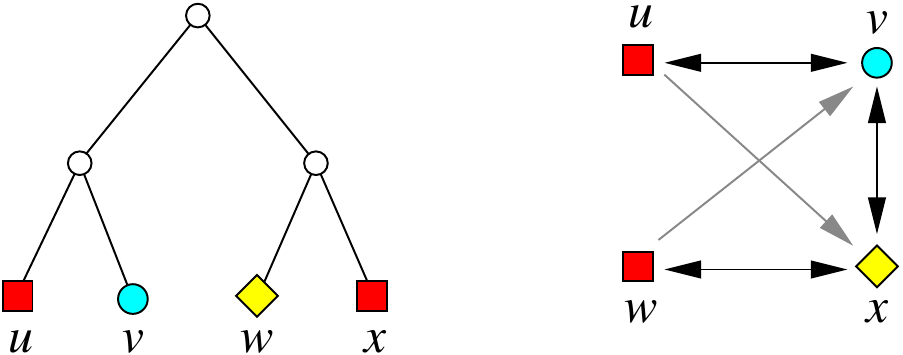}
        \caption{Colored Reciprocal Best Match Graphs are not necessarily
          cographs. This simple counterexample contains the path $u-v-x-w$
          as symmetric part. It corresponds to a species tree of the form
          $({\color{red}\blacksquare}({\color{cyan}\text{\Large\textbullet}}
          {\color{yellow}\blacklozenge}))$ and a duplication pre-dating the
          two speciations, with the speciation of
          ${\color{cyan}\text{\Large\textbullet}}$ and
          ${\color{yellow}\blacklozenge}$ being followed by complementary
          loss of one of the two copies.}
        \label{fig:counterCOG} 
  \end{center}
\end{figure} 

\begin{definition}
  A vertex-colored undirected graph $G(V,E,\sigma)$ with $\sigma:V\to S$ is
  a \emph{colored reciprocal best match graph} (cRBMG) if there is a tree
  $T$ with leaf set $V$ such that $xy\in E$ if and only if
  $\lca(x,y)\preceq \lca(x,y')$ for all $y'\in V$ with
  $\sigma(y')=\sigma(y)$ and $\lca(x,y)\preceq \lca(x',y)$ for all $x'\in
  V$ with $\sigma(x')=\sigma(x)$.
\end{definition} 
By definition $G(V,E,\sigma)$ is a cRBMG if and only if there is a cBMG
$(G',\sigma)$ with vertex set $V$ and edges $xy\in E(G)$ if and only if
both $(x,y)$ and $(y,x)$ are arcs in $(G',\sigma)$. In particular,
therefore, a cRBMG is the edge-disjoint union of the edge sets of the
induced cRBMGs by pairs of distinct colors $s,t\in S$.

\begin{figure}[htbp]
  \begin{center} 
    \includegraphics[width=0.7\textwidth]{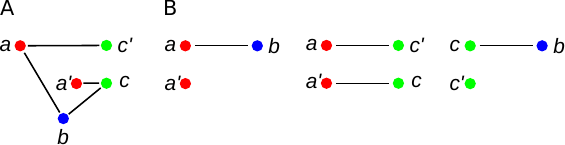}
    \caption[]{(A) A symmetric graph on three colors. (B) Each induced
      subgraph on two colors is a reciprocal Best Match Graph and a
      disjoint union of complete bipartite graphs. However, the
      corresponding symmetric graph on three colors shown in (A) does not
      have a tree representation. }
    \label{fig:counterex_sym}
  \end{center}
\end{figure}

\begin{corollary} 
  Every 2-cRBMG is the disjoint union of complete bipartite graphs. 
\end{corollary}
\begin{proof}
  By Lemma~\ref{lem:rthin-cases} there are arcs $(x,y)$ and $(y,x)$ if and
  only if $x\in\alpha\subseteq N(\beta)$ and $y\in\beta\subseteq
  N(\alpha)$. In this case $\rho_{\alpha}=\rho_{\beta}$. By
  Lemma~\ref{lem:prp1}(v) then $\sigma(\alpha)\ne\sigma(\beta)$. The same
  results also implies in a 2-cRBMG there are at most two $\rthin$ classes
  with the same root. Thus the connected components of a 2-cRBMG are the
  complete bipartite graphs formed by pairs of $\rthin$ classes with a
  common root, as well as isolated vertices corresponding to all other
  leaves of $T$. 
   
\end{proof}
The converse, however, is not true, as shown by the counterexample in
Figure~\ref{fig:counterex_sym}. The complete characterization of cRBMGs
does not seem to follow in a straightforward manner from the properties of
the underlying cBMGs. It will therefore be addressed elsewhere.

\section{Concluding Remarks} 

The main result of this contribution is a complete characterization of
colored best match graphs (cBMGs), a class of digraphs that arises
naturally at the first stage of many of the widely used computational
methods for orthology assignment. A cBMG $(G,\sigma)$ is explained by a
unique least resolved tree $(T,\sigma)$, which is displayed by the true
underlying tree. We have shown here that cBMGs can be recognized in cubic
time (in the number of genes) and with the same complexity it is possible
to reconstruct the unique least resolved tree $(T,\sigma)$. \NEW{Related
  graph classes, for instance directed cographs \cite{Crespelle:06}, which
  appear in generalizations of orthology relations \cite{Hellmuth:17a}, or
  the Fitch graphs associated with horizontal gene transfer
  \cite{Geiss:18a}, have characterizations in terms of forbidden induced
  subgraphs.  We suspect that this not the case for best match graphs
  because they are not hereditary.}

Reciprocal best match graphs, i.e., the symmetric subgraph of $(G,\sigma)$,
form the link between cBMGs and orthology relations.  The characterization
of cRBMGs, somewhat surprisingly, does not seem to be a simple consequence
of the results on cBMGs presented here. We will address this issue in
future work.

Several other questions seem to be appealing for future work. Most
importantly, what if the vertex coloring is not known \emph{a priori}? What
are the properties of BMGs in general? For connected 2-cBMGs the question
is simple, since the bipartition is easily found by a breadth first search.
In general, however, we suspect that -- similar to many other coloring
problems -- it is difficult to decide whether a digraph $G$ admits a
coloring $\sigma$ with $n=|S|$ colors such that $(G,\sigma)$ is an
$n$-cBMG. In the same vein, we may ask for the smallest number $n$ of
colors, if it exists, such that $G$ can be colored as an $n$-cBMG.

\NEW{As discussed in the introduction, usually sequence similarities are
  computed. In the presence of large differences in evolutionary rates
  between paralogous groups, maximal sequence similarity does not guarantee
  maximal evolutionary relatedness. It is often possible, however, to
  identify such problematic cases. Suppose the three species $a$, $b$, and
  $c$ form a triple $ab|c$ that is trustworthy due to independent
  phylogenetic information. Now consider a gene $x$ in $a$, two candidate
  best matches $y'$ and $y''$ in $b$, and a candidate best match $z$ in
  $c$. To decide whether $\lca(x,y')\prec \lca(x,y'')$ or not, we can use
  the support for the three possible unrooted quadruples formed by the
  sequences $\{x,y',y'',z\}$ to decide whether
  $\lca(x,y')\prec \lca(x,y'')$, which can be readily computed as the
  likelihoods of the three quadruples or using quartet-mapping
  \cite{NieseltStruwe:01}.  If the best supported quadruples is
  $(xy'|y''z)$ or $(xy''|y'z)$ it is very likely that
  $\lca(x,y')\prec \lca(x,y'')$ or $\lca(x,y'')\prec \lca(x,y')$,
  respectively, while $(xz|y'y'')$ typically indicates
  $\lca(x,y'')=\lca(x,y')$. This inference is correct as long a $z$ is
  correctly identified as outgroup to $x,y',y''$, which is very likely
  since all three of $y',y'',z$ are candidate best matches of $x$ in the
  first place.  Aggregating evidence over different choices of $z$ thus
  could be used to increase the confidence.  An empirical evaluation of
  this approach to improve \texttt{blast}-based best hit data is the
  subject of ongoing research.}

From a data analysis point of view, finally, it is of interest to ask
whether an $n$-colored digraph $(G,\sigma)$ that is not a cBMG can be
edited by adding and removing arcs to an $n$-cBMG. \NEW{This idea has been
  used successfully to obtain orthologs from noisy, empirical reciprocal
  best hit data, see e.g.\
  \citet{Hellmuth:13a,Lafond:14,Hellmuth:15a,Lafond:16,Dondi:17}. We
  propose that a step-wise approach could further improve the accuracy of
  orthology detection. In the first step, empirical (reciprocal) best hit
  data obtained with \texttt{ProteinOrtho} or a similar tool would be
  edited to conform to a cBMG or a cRBMG. These improved data are edited in
  a second step to the co-graph structure of an orthology relation. Details
  on cRBMGs and their connections with orthology will be discussed in
  forthcoming work.}

\section*{Acknowledgements}
  The nucleus of this work was performed during a retreat meeting around
  New Year 2018 in Gro{\ss}lobming, Styria, Austria. Support the by the
  German Academic Exchange Service (DAAD, PROALMEX grant no.\ 57274200),
  the Mexican Consejo Nacional de Ciencia y Tecnolog{\'i}a (CONACyT, 278966
  FONCICYT 2), and the German Federal Ministry of Education and Research
  (BMBF, project no.\ 031A538A, de.NBI-RBC) is gratefully acknowledged.


\clearpage

\newtheorem{innercustomgeneric}{\customgenericname}
\providecommand{\customgenericname}{}
\newcommand{\newcustomtheorem}[2]{%
        \newenvironment{#1}[1]
        {%
                \renewcommand\customgenericname{#2}%
                \renewcommand\theinnercustomgeneric{##1}%
                \innercustomgeneric
        }
        {\endinnercustomgeneric}
}

\newcustomtheorem{customthm}{Theorem}
\newcommand{\G}{G}

\centerline{\Large Corrigendum to ``Best Match Graphs''}

\begin{center}\large
  David Schaller$^1$
  Manuela Gei{\ss}$^2$ 
  Edgar Ch{\'a}vez$^3$
  Marcos Gonz{\'a}lez Laffitte$^3$
  Alitzel L{\'o}pez S{\'a}nchez$^4$
  B{\"a}rbel M.\ R.\ Stadler$^1$
  Dulce I.\ Valdivia$^5$
  Marc Hellmuth$^6$
  Maribel Hern{\'a}ndez Rosales$^5$
  Peter F.\ Stadler$^{1,2,7-10}$
\end{center}

\begin{center}\small
$^1$Max-Planck-Institute for Mathematics in the Sciences,
  Inselstra{\ss}e 22, D-04103 Leipzig\\
  \email{sdavid@bioinf.uni-leipzig.de, baer@bioinf.uni-leipzig.de}\\
$^2$Software Competence Center Hagenberg GmbH,
  Softwarepark 21, 4232 Hagenberg, Austria\\
  \email{manuela.geiss@scch.at}\\
$^3$Instituto de Matem{\'a}ticas, UNAM Juriquilla,
  Blvd.\ Juriquilla 3001,
  76230 Juriquilla, Quer{\'e}taro, QRO, M{\'e}xico\\
  \email{echavezaparicio@gmail.com, marcoslaffitte@im.unam.mx}\\
$^4$Department of Computer Science,
  Universit{\'e} de Sherbrooke,
  2500 Boul. de l'Universit{\'e}, Sherbrooke, J1K 2R1, Canada\\
  \email{Alitzel.Lopez.Sanchez@USherbrooke.ca}\\
$^5$Centro de Investigaci{\'o}n y de Estudios Avanzandos del IPN
  (CINVESTAV), Irapuato Unit,
  Libramiento Norte, Carretera Panamericana Irapuato-León,
  Kil{\'o}metro 9.6, 36821 Irapuato, Gto., M{\'e}xico\\
  \email{dulce.valdivia@cinvestav.mx, maribel.hr@cinvestav.mx}
$^6$School of Computing, University of Leeds, E C Stoner Building,
  Leeds LS2 9JT, UK\\
  \email{mhellmuth@mailbox.org}
$^7$ Bioinformatics Group, Department of Computer Science;
  Interdisciplinary Center of Bioinformatics;
  German Centre for Integrative Biodiversity Research (iDiv)
  Halle-Jena-Leipzig; Competence Center for Scalable Data Services
  and Solutions; and Leipzig Research Center for Civilization Diseases,
  Leipzig University,
  H{\"a}rtelstra{\ss}e 16-18, D-04107 Leipzig\\
  \email{studla@bioinf.uni-leipzig.de}\\
$^8$Inst.\ f.\ Theoretical Chemistry, University of Vienna,
  W{\"a}hringerstra{\ss}e 17, A-1090 Wien, Austria\\
$^9$Facultad de Ciencias, Universidad National de Colombia, Sede
Bogot{\'a}, Colombia\\
$^{10}$Santa Fe Institute, 1399 Hyde Park Rd., Santa Fe,
  NM 87501, USA 
\end{center}

\marginpar{\color{blue}\scriptsize The corrigendum refers to the
  published version of this manuscript.}
\begin{abstract}
  Two errors in the article \emph{Best Match Graphs} [Gei{\ss} et al., J
  Math Biol (2019) 78: 2015-2057] are corrected. One concerns the tacit
  assumption that digraphs are sink-free, which has to be added as an
  additional precondition in Lemma~9, Lemma~11, and Theorem~4. The second
  correction concerns an additional necessary condition in Theorem~9
  required to characterize best match graphs.
\end{abstract}

\section*{Best match graphs (BMGs) must be sink-free}

Throughout \cite{Geiss:19a} we have tacitly assumed that all vertex-colored
digraphs $(\G,\sigma)$ satisfy the following property, which, by construction,
is true for all colored best match graphs (cBMGs):
\begin{itemize}
\item[]\textit{For each vertex $x$ with color $\sigma(x)$, there is an arc
    $xy$ \emph{to} at least one vertex $y$ of every other color
    $\sigma(y)\neq \sigma(x)$.}
\end{itemize}
All properly 2-colored digraphs appearing in the text are therefore assumed
to be sink-free, i.e., the out-neighborhoods of their $\rthin$-classes are
assumed to be non-empty:
\begin{description}
\item[\AX{(N4)}] $N(\alpha)\ne\emptyset$ for all $\alpha\in\mathcal{N}$.
\end{description}
This assumption was not clearly stated in the text.

Property \AX{(N4)} is required in the proof of Lemma~9 [last line on page
2032]: Here $R(\beta)\cap R(\alpha^*)=\emptyset$ only implies
$\beta\subseteq R(\alpha)$ for the $\rthin$ classes $\beta$ with
$R(\beta)\subseteq R(\alpha)$ if $R(\beta)\ne\emptyset$, which in turn is
equivalent to $N(\beta)\ne\emptyset$. Furthermore, $N(\alpha)=\emptyset$
implies $Q(\alpha)=\alpha$ and thus $R'(\alpha)=\alpha$.  Property
\AX{(N4)} is therefore also necessary to ensure that $|R'(\alpha)|>1$ and
thus that the tree $T(\mathscr{H}')$ is phylogenetic [page 2035, just
before Thm.~4]. In summary, Lemma~9, Lemma~11, and Theorem~4
\marginpar{\color{blue}\scriptsize Thm.4 is Thm.6 in the preprint version.}
require
\AX{(N4)} as additional precondition.


\section*{Corrected Characterization of $n$-cBMGs}

The second paragraph of the proof of Theorem 9
\marginpar{\color{blue}\scriptsize Thm.9 is Thm.14 in the preprint version.}
in \cite[page
2045]{Geiss:19a} incorrectly states that ``\emph{for any
  $\alpha\in\mathcal{N}$ and any color $s\neq\sigma(\alpha)$ the
  out-neighborhood $N_s(\alpha)$ is the same w.r.t.\ $(T_{st},\sigma_{st})$
  and w.r.t.\ $\Aho(R)$.}'', leading to the incorrect conclusion that
$\G(\Aho(R),\sigma)=(\G,\sigma)$ whenever $\Aho(R)$ exists. We recall that
the triple set $R$ is defined as the union
\begin{equation*}
  R \coloneqq \bigcup_{s,t\in S} r(T^*_{st})
\end{equation*}
of all triples in the least resolved trees $(T^*_{st},\sigma_{st})$ that
explain the induced subgraphs $(\G_{st},\sigma_{st})$ of $(\G,\sigma)$, and
the Aho tree $\Aho(R)$ is defined on the leaf set $L=V(\G)$ (which may not
have been clear from the wording in the text). We shall show in
Prop.~\ref{prop:aho-explains-subgraph} below that $\G(\Aho(R),\sigma)$ is
always a subgraph of $(\G,\sigma)$ whenever $R$ is consistent.  The example
in Fig.~\ref{fig:counterex-thm9} shows, however, that
$\G(\Aho(R),\sigma)\ne(\G,\sigma)$ is possible because $\Aho(R)$ can
contain triples that are not present in any of the 2-colored trees
$(T_{st},\sigma_{st})$.

\begin{figure}[ht]
  \begin{center}
    \includegraphics[width=0.95\linewidth]{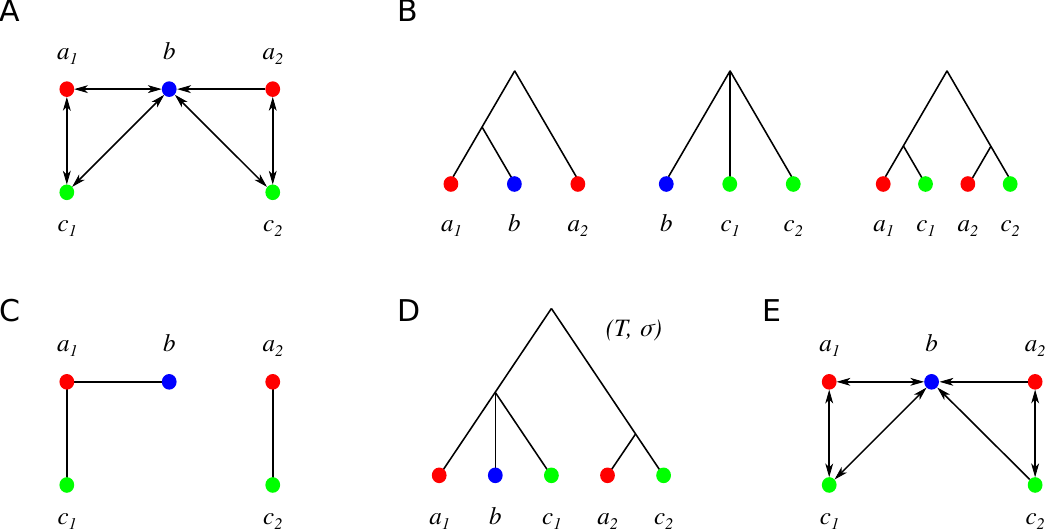}
  \end{center}
  \caption[]{Counterexample for the original version of Theorem~9. (A) A
    colored digraph with vertex set $L$ that is not a 3-cBMG. (B) The least
    resolved subtrees for the three 2-colored induced subgraphs. The
    union of their triples is
    $R\coloneqq\{a_1b|a_2,\; a_1c_1|a_2,\; a_1c_1|c_2,\; a_2c_2|a_1,\;
    a_2c_2|c_1\}$.  (C) The Aho-graph $[R, L]$. In particular, $R$ forms a
    consistent set. (D) The tree $T\coloneqq\Aho(R)$.  (E) The $3$-cBMG
    $\G(T,\sigma)$.  The arc $bc_2$ that was present in $(\G,\sigma)$ is
    missing in $\G(T,\sigma)$.}
  \label{fig:counterex-thm9}
\end{figure}

As a consequence, the characterization of $n$-cBMGs requires the equality
$(\G,\sigma)=\G(\Aho(R),\sigma)$ as an additional condition. The corrected
result, with the correction underlined, reads as follows:
\begin{customthm}{9}
  A connected colored digraph $(\G,\sigma)$ is an $n$-cBMG if and only if
  (i) all induced subgraphs $(\G_{st},\sigma_{st})$ on two colors are
  2-cBMGs, (ii) the union $R$ of all triples obtained from their least
  resolved trees $(T_{st},\sigma_{st})$ forms a consistent set, 
  \underline{and (iii) $\G(\Aho(R),\sigma)=(\G,\sigma)$}. In 
    particular, $(\Aho(R),\sigma)$ is the
  unique least resolved tree that explains $(\G,\sigma)$.
\end{customthm}
Condition (i) in Thm.~9, i.e., the requirement that all 2-colored induced
subgraphs $(\G_{st},\sigma_{st})$ of $(\G,\sigma)$ are 2-cBMGs, is
necessary to ensure that the least resolved trees $(T_{st},\sigma_{st})$
exist and thus that the triple sets $r(T_{st})$ -- and therefore also the
set $R$ of all triples displayed by the 2-colored induced subgraphs -- are
well-defined. Consistency of $R$ is necessary for the existence of
$\Aho(R)$. Clearly, Condition (iii) is sufficient to ensure that
$(\G,\sigma)$ is an $n$-cBMG. Hence, it remains to show that Condition
(iii) is also necessary. This is achieved in Prop.~\ref{prop:AhoR} below.

Before we proceed, we note that none of the necessary corrections has
consequences for the algorithmic aspects outlined in Section~5 of
\cite{Geiss:19a}. Sink-freeness can be checked trivially in $O(|E|)$
time. Regarding the recognition of $n$-cBMGs, it suffices to construct the
tree $T\coloneqq \Aho(\mathscr{R}(\G,\sigma))$, which can be done in
$O(|E||L|\log^2(|E||L|))$ time using the algorithm in \cite{Deng:18}.  The
construction of $\G(T,\sigma)$ can then be achieved in $O(|L|^2)$ time
e.g.\ using Algorithm~1 in the Supplement of \cite{Geiss:20b}, and the
final check whether $\G(T,\sigma)=(\G,\sigma)$ also requires $O(|L|^2)$
operations. The total effort therefore remains dominated by the
construction of the least resolved tree $T$.

\section*{Proof of Theorem 9}
\marginpar{\color{blue}\scriptsize Thm.9 is Thm.14 in the preprint version.}
Instead of proving the corrected version of Thm.~9 directly, we first state
and prove a slightly stronger and more convenient result,
Thm.~\ref{thm9_new} below, and then proceed to derive Thm.~9.  To this end,
we first generalize Def.~8 in \cite[page 2036]{Geiss:19a} to digraphs with
an arbitrary number of colors:
\begin{definition}{\cite[Def.~2.7]{Schaller:20}}\label{def:informative_triples}
  Let $(\G,\sigma)$ be a colored digraph. We say that a triple $ab|b'$ is
  \emph{informative} for $(\G,\sigma)$ if $a$, $b$ and $b'$ are pairwise
  distinct vertices in $\G$ such that (i)
  $\sigma(a)\neq\sigma(b)=\sigma(b')$ and (ii) $ab\in E(\G)$ and
  $ab'\notin E(\G)$. The set of informative triples is denoted by
  $\mathscr{R}(\G,\sigma)$.
\end{definition}
We briefly argue that, for 2-colored digraphs, the definition of
informative triples given here is equivalent to the one given in
\cite{Geiss:19a}: By definition, an informative triple of some colored
digraph has vertices with exactly two colors, and thus is also an
informative triple in one of its 2-colored induced subgraphs. It is easy to
check that, for 2-colored digraphs, Def.~\ref{def:informative_triples} is
equivalent to Def.~8 in \cite{Geiss:19a}, since the four induced subgraphs
shown in Fig.~8 in \cite[page 2036]{Geiss:19a} correspond to the presence
or absence of the two optional arcs $ba$ and $ca$ in the informative triple
$ab|c$ (as defined here).

We will also make use of a generalization of Lemma~12 in \cite[page
2036]{Geiss:19a}:
\begin{lemma}{\cite[Lemma~2.8]{Schaller:20}}
  \label{lem:informative_triples}
  Let $(\G,\sigma)$ be an $n$-cBMG and $ab|b'$ an informative triple for
  $(\G,\sigma)$.  Then, every tree $(T,\sigma)$ that explains $(\G,\sigma)$ 
  displays
  the triple $ab|b'$, i.e. $\lca_T(a,b)\prec_T\lca_T(a,b')=\lca_T(b,b')$.
\end{lemma}
Given a digraph $(\G,\sigma)$ for which $R$ exists,
Lemma~\ref{lem:informative_triples} in particular implies that
\begin{equation}
  \mathscr{R}(\G,\sigma)\subseteq R.
  \label{eq:1}
\end{equation}
With these preliminaries, we are ready to formulate our new main result as
\begin{theorem}
  A colored digraph $(\G,\sigma)$ is an $n$-cBMG if and only if
  $\G(\Aho(\mathscr{R}(\G,\sigma)),\sigma) = (\G,\sigma)$.  Moreover,
  $\Aho(\mathscr{R}(\G,\sigma))$ is the unique least resolved tree
  explaining an n-cBMG $(\G,\sigma)$.
  \label{thm9_new}
\end{theorem}

In the proof of Thm.~\ref{thm9_new}, we will need the following,
simplified, characterization of redundant edges:
\begin{lemma}[\textnormal{\citet{Schaller:20}, Lemma~2.10}]
  \label{lem:redundant_edges}
  Let $(\G,\sigma)$ be an $n$-cBMG explained by a tree $(T,\sigma)$.  The
  edge $e=uv$ with $v\prec_T u$ in $(T,\sigma)$ is redundant w.r.t.\
  $(\G,\sigma)$ if and only if (i) $e$ is an inner edge of $T$ and (ii)
  there is no arc $ab\in E(\G)$ such that $\lca_T(a,b)=v$ and
  $\sigma(b)\in \sigma(L(T(u))\setminus L(T(v)))$.
\end{lemma}
We note that the proofs of Lemma~\ref{lem:informative_triples}
\cite[Lemma~2.8]{Schaller:20} and Lemma~\ref{lem:redundant_edges}
\cite[Lemma~2.10]{Schaller:20} only require the definition of best match
graphs, and are thus independent of the results proved in \cite{Geiss:19a}.

Following \citet{bryant1995extension}, an inner edge $e$ of a rooted tree
$T$ is \emph{distinguished} by a triple $ab|c\in r(T)$ if the path from $a$
to $c$ in $T$ intersects the path from $b$ to the root $\rho_T$ precisely
on the edge e.  In other words, $e=uv$ with $v\prec_T u$ is distinguished
by $ab|c$ if $\lca_T(a, b) = v$ and $\lca_T(a, b, c) = u$.
Lemma~\ref{lem:redundant_edges} immediately implies the following
generalization of Lemma~13 in \cite{Geiss:19a}:
\begin{corollary}
  Let $(\G,\sigma)$ be an $n$-cBMG explained by a tree $(T,\sigma)$.  An
  inner edge $e$ of $(T,\sigma)$ is non-redundant w.r.t.\ $(\G,\sigma)$ if
  and only if it is distinguished by an informative triple $ab|b'$ for
  $(\G,\sigma)$.  In particular, if $(T,\sigma)$ is least resolved, then
  each of its inner edges is distinguished by an informative triple.
  \label{cor:non-redundant}
\end{corollary}

In addition, we will need the following two technical results relating
subtrees and induced subgraphs of $n$-cBMGs.
\begin{lemma}\label{lem:subgrap}
  Let $(T,\sigma)$ be a tree explaining an $n$-cBMG $(\G,\sigma)$.  Then
  $\G(T(u),\sigma_{|L(T(u))}) = (\G[L(T(u))],\sigma_{|L(T(u))})$ holds for
  every $u\in V(T)$.
\end{lemma}
\begin{proof}
  Let
  $(\G_1, \sigma')\coloneqq \G\left(T(u), \sigma_{|L(T(u))}\right)$ and
  $(\G_2, \sigma')\coloneqq (\G[L(T(u))],\sigma_{|L(T(u))})$.  By definition,
  we have $V(\G_1)=V(\G_2)=L(T(u))$.  First assume that $xy\in E(\G_1)$
  for some $x,y\in L(T(u))$.  Hence, it holds
  $\lca_{T(u)}(x,y)\preceq_{T(u)}\lca_{T(u)}(x,y')$ for all $y'$ with
  $\sigma(y)=\sigma(y')$ in $T(u)$ and thus, since $T(u)$ is a subtree of
  $T$, we have $\lca_T(x,y)\preceq_T\lca_T(x,y')$ for all $y'$ with
  $\sigma(y)=\sigma(y')$ in $T$. Therefore, $xy\in E(\G)$.  Since
  $x,y\in L(T(u))$ and $\G_2$ is the subgraph of $\G$ induced by $L(T(u))$,
  we have $xy\in E(\G_2)$ and thus $E(\G_1)\subseteq E(\G_2)$. Now assume
  $xy\in E(\G_2)$ for some $x,y\in L(T(u))$.  Hence, $xy\in E(\G)$.
  Consequently, there is no leaf $y'$ in $T$ with
  $\sigma(y')=\sigma(y)\ne\sigma(x)$ such that
  $\lca_{T}(x,y')\prec_{T}\lca_{T}(x,y)\preceq_{T} u$.  This clearly also
  holds for the subtree $T(u)$. Therefore, we have $xy\in E(\G_1)$ and
  thus $E(\G_2)\subseteq E(\G_1)$.  
\end{proof}

\begin{lemma}
  \label{lem:least}
  If $(T,\sigma)$ is least resolved for an $n$-cBMG $(\G,\sigma)$, then the
  subtree $T(u)$ is least resolved for the $n$-cBMG
  $\G(T(u),\sigma_{|L(T(u))})$ for each $u\in V(T)$.
\end{lemma}
\begin{proof}
  The statement is trivially satisfied if $T(u)$ does not contain any inner
  edges, which is exactly the case if either $u\in L(T)$ or $u\in V^0(T)$
  with $\child_T(u)\subseteq L(T)$. Thus, let $u\in V^0(T)$ and
  $\child_T(u)\cap V^0(T)\neq \emptyset$.  Since $(T,\sigma)$ is least
  resolved, it does not contain redundant edges.  Let $vw$ be an inner edge
  of $T(u)$ with $w\prec_T v\preceq_T u$, and note that $vw$ must also be
  an inner edge in $T$.  By Lemma~\ref{lem:redundant_edges} and since $vw$
  is not redundant in $T$, there is an arc $ab\in E(\G)$ such that
  $\lca_T(a,b)=w$ and $\sigma(b)\in \sigma(L(T(v))\setminus L(T(w)))$.
  Since $u\succeq_T v$, Lemma~\ref{lem:subgrap} implies that $ab$ is also
  an arc in $\G(T(u),\sigma_{|L(T(u))})$ and $\lca_{T(u)}(a,b)=v$.  Hence,
  in particular, we have $\sigma(b)\in\sigma(L(T(v))\setminus L(T(w)))$. We
  can now apply Lemma~\ref{lem:redundant_edges} to conclude that $vw$ is
  not redundant in $T(u)$.  Since $vw$ was chosen arbitrarily, we conclude
  that $T(u)$ does not contain any redundant edge and thus, it must be
  least resolved for $\G(T(u),\sigma_{|L(T(u))})$ for all $u\in V(T)$.
\end{proof}
We finally relate the subtrees $T(u)$ to the construction of the Aho-graph
as specified in \cite[Sec.~3.4]{Geiss:19a}. Given a set of triples $R$
on $L$, we will write $R_{|L'}$ for the set of triples $ab|c\in R$ with
$a,b,c\in L'\subseteq L$.

\begin{lemma}\label{lem:bmg-aho-components}
  Let $(T,\sigma)$ be least resolved for an $n$-cBMG $(\G,\sigma)$ with
  informative triple set $\mathscr{R}\coloneqq\mathscr{R}(\G,\sigma)$.
  Then, $L(T(v))$ is a connected component in the Aho-graph
  $[\mathscr{R}_{|L(T(u))}, L(T(u))]$ for every inner vertex $u$ and each
  of its children $v\in \child_{T}(u)$.
\end{lemma}
\begin{proof}
  We proceed by induction on $L \coloneqq V(\G)$. The statement
  trivially holds for $|L|=1$.  Hence, suppose that $|L|>1$ and assume that
  the statement is true for every $n$-cBMG with less than $|L|$ vertices.

  Let $u$ be an inner vertex of $T$ and $v\in \child_{T}(u)$.  If $uv$ is
  an outer edge, i.e.\ $v$ is a leaf, then $L(T(v))$ is trivially
  connected. Now suppose that $uv$ is an inner edge of $T$.  By
  Lemma~\ref{lem:subgrap} and~\ref{lem:least},
  $(\G[L(T(v))],\sigma_{|L(T(v))})$ is explained by the least resolved tree
  $(T(v), \sigma_{|L(T(v))})$. By the induction hypothesis, $L(T(w))$ forms
  a connected component in $[\mathscr{R}_{|L(T(v))}, L(T(v))]$ for all
  children $w\in \child_T(v)$.  Together with
  $\mathscr{R}_{|L(T(v))}\subseteq \mathscr{R}_{|L(T(u))}$, this implies
  that the elements in $L(T(w))$ are also connected in
  $[\mathscr{R}_{|L(T(u))}, L(T(u))]$ for all $w\in \child_T(v)$.  Since
  $uv$ is an inner edge of the least resolved tree $(T,\sigma)$, we can
  apply Cor.~\ref{cor:non-redundant} to conclude that there is an
  informative triple $ab|b'$ in $(\G,\sigma)$ that distinguishes $uv$,
  i.e.\ $\lca_T(a,b)=v$ and $b'\in L(T(u))\setminus L(T(v))$ with color
  $\sigma(b')=\sigma(b)$.  Hence, $ab|b'$ is also contained in
  $[\mathscr{R}_{|L(T(u))}, L(T(u))]$.  In particular, there are children
  $w,w'\in \child_T(v)$ such that $a\preceq_T w$ and $b\preceq_T w'$, and
  the edge $ab$ connects $L(T(w))$ and $L(T(w'))$ in
  $[\mathscr{R}_{|L(T(u))}, L(T(u))]$.

  Now suppose that there is an additional child
  $w''\in \child_T(v)\setminus\{w,w'\}$. We distinguish two cases. Either
  there is a leaf $b''\preceq_T w''$ with $\sigma(b'')=\sigma(b)$ or
  no such leaf exists.  If there is such a leaf $b''$, then $ab''$
  forms an arc in $(\G,\sigma)$ and $ab''|b'$ is an informative triple
  making $L(T(w))$ and $L(T(w''))$ connected in
  $[\mathscr{R}_{|L(T(u))}, L(T(u))]$. Otherwise, take an arbitrary leaf
  $c\preceq_T w''$.  Since $\sigma(b)\notin\sigma(L(T(w'')))$, we have
  $\sigma(c)\neq \sigma(b)$ and thus, there is an arc $cb$ in
  $(\G,\sigma)$.  Since $\lca_T(c,b') = u \succ_T v = \lca_T(c,b)$, the arc
  $cb'$ is not contained in $(\G,\sigma)$.  Hence, $cb|b'$ is an
  informative triple making $L(T(w'))$ and $L(T(w''))$ connected in
  $[\mathscr{R}_{|L(T(u))}, L(T(u))]$.
	
  Therefore, the subgraph in $[\mathscr{R}_{|L(T(u))}, L(T(u))]$ induced by 
  $L(T(v))$ must be connected.

  It remains to show that $L(T(v))$ is a connected component in
  $[\mathscr{R}_{|L(T(u))}, L(T(u))]$ and thus, that there are no edges
  $ab$ in $[\mathscr{R}_{|L(T(u))}, L(T(u))]$ with $a\in L(T(v))$ and
  $b\in L(T(u))\setminus L(T(v))$.  Assume, for contradiction, that there
  exists such an edge $ab$.  Hence, this edge must be supported by an
  informative triple w.l.o.g.\ $ab|b'$ with
  $\sigma(a)\neq \sigma(b)=\sigma(b')$ and $b'\in L(T(u))$.  Lemma
  \ref{lem:informative_triples} implies that $ab|b'$ must be displayed by
  $T$. However, $\lca_T(a,b) = u = \lca_T(a,b,b')$ implies that such a
  triple cannot exit. Thus, $L(T(v))$ is a connected component in
  $[\mathscr{R}_{|L(T(u))}, L(T(u))]$. 
\end{proof}
The least resolved tree of an $n$-cBMG therefore coincides with the Aho
tree of its informative triples. In more detail, we have
\begin{proposition}\label{prop:aho-unique-lrt}
  If $(\G,\sigma)$ is an $n$-cBMG, then
  $(\Aho(\mathscr{R}(\G,\sigma)),\sigma)$ is the unique least resolved tree
  for $(\G,\sigma)$.
\end{proposition}
\begin{proof}
  Since $(\G,\sigma)$ is an $n$-cBMG, Lemma~\ref{lem:informative_triples}
  implies that there is a tree displaying all triples in
  $\mathscr{R}(\G,\sigma)$.  In particular, therefore,
  $\Aho(\mathscr{R}(\G,\sigma))$ exists.  Moreover, there must be a least
  resolved tree $(T^*,\sigma)$ for $(\G,\sigma)$.  To see this, consider an
  arbitrary tree $(T,\sigma)$ that explains $(\G,\sigma)$, and repeatedly
  identify and contract a redundant edge until no redundant edges remain.
  By definition, the resulting tree still explains $(\G,\sigma)$ and is
  least resolved.  By Lemma~\ref{lem:bmg-aho-components} and by
  construction of $(\Aho(\mathscr{R}(\G,\sigma)),\sigma)$, any least
  resolved tree $(T^*,\sigma)$ for $(\G,\sigma)$ coincides with the latter.
  The uniqueness of $\Aho(\mathscr{R}(\G,\sigma))$ therefore implies that
  the least resolved tree is also unique.
\end{proof}

We now have all the pieces in place to complete the proof of the main
result:
\begin{proof}\emph{of Theorem~\ref{thm9_new}}
  If $(\G,\sigma)$ is an
  $n$-cBMG, then Prop.~\ref{prop:aho-unique-lrt} implies that
  $(\Aho(\mathscr{R}(\G,\sigma)),\sigma)$ is its unique least resolved
  tree, and thus $\G(\Aho(\mathscr{R}(\G,\sigma)),\sigma) = (\G,\sigma)$.
  Conversely, $\G(\Aho(\mathscr{R}(\G,\sigma)),\sigma)$ is an $n$-cBMG.
\end{proof}

None of the intermediate results used to prove Thm.~9 in
\cite{Geiss:19a} were used in our proof of Thm.~\ref{thm9_new}. It is
worth noting that Thm.~\ref{thm9_new} immediately implies Thms.~5, 6,
and~7, as well as the existence of a unique least resolved tree in
Thms.~2 and~8 of \cite{Geiss:19a}. It allows us to obtain the least
resolved tree of an $n$-cBMG without the need to explicitly construct the
least resolved trees of all its 2-colored induced subgraphs.

To show the correctness of Thm.~9, it only remains to show
\begin{proposition}\label{prop:AhoR}
  If $(\G,\sigma)$ is a $n$-cBMG, then
  $\Aho(\mathscr{R}(\G,\sigma))=\Aho(R)$.
\end{proposition}
\begin{proof}
  For brevity set $\mathscr{R}\coloneqq\mathscr{R}(\G,\sigma)$. From
  Eq.~(\ref{eq:1}), i.e., $\mathscr{R}\subseteq R$, we immediately have
  $\mathscr{R}_{|L(T(u))}\subseteq R_{|L(T(u))}$ for every inner vertex $u$
  of $T$. Moreover, by Thm.~\ref{thm9_new}, $(T,\sigma)$ with
  $T\coloneqq \Aho(\mathscr{R})$ is the least resolved tree that explains
  $(\G,\sigma)$.

  Hence, we can apply the same arguments as in the proof of
  Lemma~\ref{lem:bmg-aho-components} to conclude that $L(T(v))$ forms a
  connected component in the Aho-graph $[R_{|L(T(u))}, L(T(u))]$ for every
  inner vertex $u$ and each of its children $v\in \child_{T}(u)$. More
  precisely, note that connectedness of any such $L(T(v))$ is guaranteed by
  the informative triples.  Now assume, for contradiction, that there is an
  edge $ab$ in $[R_{|L(T(u))}, L(T(u))]$ with $a\in L(T(v))$ and
  $b\in L(T(u))\setminus L(T(v))$ connecting $L(T(v))$ and $L(T(v'))$ for
  some child $v'\in \child_T(u)\setminus\{v\}$.  In this case, there is a
  triple $ab|c\in R_{|L(T(u))}$ and thus, $a,b,c\in L(T(u))$ and
  $\lca_T(a,b,c)=u$. By definition of $R$ and Observation 4 in
  \cite{Geiss:19a}, $ab|c$ must be displayed by $T$. However,
  $a,b,c\in L(T(u))$ and $\lca_{T}(a,b)=u =\lca_{T}(a,b,c)$ imply that
  $ab|c$ is not displayed by $T$; a contradiction. Therefore,
  $(T,\sigma)=(\Aho(R),\sigma)$, which completes the proof. 
\end{proof}

For completeness, we show that conditions (i) and (ii) of Thm.~9 
ensure that $G(\Aho(R),\sigma)$ and $G(\Aho(\mathscr{R}(\G,\sigma)),\sigma)$
are subgraphs of $(G,\sigma)$. 
\begin{proposition}\label{prop:aho-explains-subgraph}
  Let $(\G,\sigma)$ be a properly $n$-colored digraph with all
  2-colored induced subgraphs being 2-cBMGs.  Then the following two
  statements hold:
  \begin{enumerate}[noitemsep, nolistsep, label=(\roman*)]
    \item If $\mathscr{R}(\G,\sigma)$ is consistent, then 
    $\G(\Aho(\mathscr{R}(\G,\sigma)),\sigma)\subseteq(\G,\sigma)$.
    \item If $R$ is consistent, then
    $\G(\Aho(R),\sigma)\subseteq(\G,\sigma)$.
  \end{enumerate}
\end{proposition}
\begin{proof}
  We set $(\G',\sigma')\coloneqq\G(\Aho(\mathscr{R}(\G,\sigma)),\sigma)$.
  Since $\Aho(\mathscr{R}(\G,\sigma))$ is defined on $V(G)$, we have
  $V(\G')= V(\G)$ and $\sigma'=\sigma$.  Now assume, for contradiction,
  that there is an arc $ab\in E(\G')$ such that $ab\notin E(\G)$.  By
  assumption, the induced subgraph $(\G_{st},\sigma_{st})$ of $(G,\sigma)$,
  where $s=\sigma(a)$ and $t=\sigma(b)$, is a 2-cBMG and thus sink-free.
  Therefore, there must be a vertex $b'$ of color $\sigma(b)$ with
  $ab'\in E(\G)$.  Hence, $ab'|b$ is informative for $(\G,\sigma)$ and
  contained in $\mathscr{R}(\G,\sigma)$.  In particular, $ab'|b$ must be
  displayed by $\Aho(\mathscr{R}(\G,\sigma))$; contradicting that $ab$ is
  an arc in $(\G',\sigma')$.  Hence, statement~(i) is true.
  
  Statement~(ii) can be shown by similar arguments together with 
  Eq.~(\ref{eq:1}), i.e., $\mathscr{R}(\G,\sigma)\subseteq R$.
\end{proof}

\end{document}